\documentclass[11pt, reqno]{amsart}
\usepackage{amsfonts,amsmath, amssymb,amsthm,amscd}
\usepackage[usenames,dvipsnames]{color}
\usepackage{bm}
\usepackage{mathrsfs}
\usepackage{yfonts}
\usepackage{verbatim}
\usepackage{hyperref}
\usepackage{graphicx}
\usepackage[utf8]{inputenc}
\usepackage{latexsym}
\usepackage{lscape}
\usepackage{epsfig}
\usepackage{subfigure}
\usepackage{dsfont}
\usepackage[left=2.5cm, right=2.5cm, top=3cm, bottom=3cm]{geometry}
\usepackage{tikz}
\usetikzlibrary{shapes.multipart}
\usetikzlibrary{patterns}
\usepackage{setspace}
\usepackage{mathtools}
\usepackage{lmodern}

\begin{document}
\newcommand{\cn}[1]{\overline{#1}}
\newcommand{\e}[0]{\epsilon}
\newcommand{\EE}{\ensuremath{\mathbb{E}}}
\newcommand{\qq}[1]{(q;q)_{#1}}
\newcommand{\A}{\ensuremath{\mathcal{A}}}
\newcommand{\GT}{\ensuremath{\mathbb{GT}}}
\newcommand{\link}{\ensuremath{Q}}
\newcommand{\PP}{\ensuremath{\mathbb{P}}}
\newcommand{\frakP}{\ensuremath{\mathfrak{P}}}
\newcommand{\frakQ}{\ensuremath{\mathfrak{Q}}}
\newcommand{\frakq}{\ensuremath{\mathfrak{q}}}
\newcommand{\R}{\ensuremath{\mathbb{R}}}
\newcommand{\Rplus}{\ensuremath{\mathbb{R}_{+}}}
\newcommand{\C}{\ensuremath{\mathbb{C}}}
\newcommand{\Z}{\ensuremath{\mathbb{Z}}}
\newcommand{\Weyl}[1]{\ensuremath{\mathbb{W}}^{#1}}
\newcommand{\Zgzero}{\ensuremath{\mathbb{Z}_{>0}}}
\newcommand{\Zgeqzero}{\ensuremath{\mathbb{Z}_{\geq 0}}}
\newcommand{\Zleqzero}{\ensuremath{\mathbb{Z}_{\leq 0}}}
\newcommand{\Q}{\ensuremath{\mathbb{Q}}}
\newcommand{\T}{\ensuremath{\mathbb{T}}}
\newcommand{\Y}{\ensuremath{\mathbb{Y}}}
\newcommand{\M}{\ensuremath{\mathbf{M}}}
\newcommand{\MM}{\ensuremath{\mathbf{MM}}}
\newcommand{\W}[1]{\ensuremath{\mathbf{W}}_{(#1)}}
\newcommand{\WM}[1]{\ensuremath{\mathbf{WM}}_{(#1)}}
\newcommand{\Zsd}{\ensuremath{\mathbf{Z}}}
\newcommand{\Fsd}{\ensuremath{\mathbf{F}}}
\newcommand{\symBM}{\ensuremath{\mathbf{W}}}
\newcommand{\symFE}{\ensuremath{\mathbf{S}}}

\newcommand{\Real}{\ensuremath{\mathrm{Re}}}
\newcommand{\Imag}{\ensuremath{\mathrm{Im}}}
\newcommand{\re}{\ensuremath{\mathrm{Re}}}

\newcommand{\Sym}{\ensuremath{\mathrm{Sym}}}

\newcommand{\bfone}{\ensuremath{\mathbf{1}}}

\newcommand{\whitenoise}{\ensuremath{\mathscr{\dot{W}}}}
\newcommand{\alphaW}[1]{\ensuremath{\mathbf{\alpha W}}_{(#1)}}
\newcommand{\alphaWM}[1]{\ensuremath{\mathbf{\alpha WM}}_{(#1)}}
\newcommand{\malpha}{\ensuremath{\hat{\alpha}}}
\newcommand{\walpha}{\ensuremath{\alpha}}
\newcommand{\edge}{\textrm{edge}}
\newcommand{\dist}{\textrm{dist}}

\newcommand{\OO}[0]{\Omega}
\newcommand{\F}[0]{\mathfrak{F}}
\newcommand{\poly}[0]{R}
\def \Ai {{\rm Ai}}
\def \sgn {{\rm sgn}}
\def \SS {\mathcal{S}}
\newcommand{\poles}{\mathbb{A}}
\def \ss {\mathcal{X}}
\newcommand{\var}{{\rm var}}

\newcommand{\Res}[1]{\underset{{#1}}{\mathrm{Res}}}
\newcommand{\Resfrac}[1]{\mathrm{Res}_{{#1}}}

\newcommand{\ul}[2]{\underline{#1}_{#2}}
\newcommand{\qhat}[1]{\widehat{#1}^{q}}
\newcommand{\La}[0]{\Lambda}
\newcommand{\la}[0]{\lambda}
\newcommand{\ta}[0]{\theta}
\newcommand{\w}[0]{\omega}
\newcommand{\ra}[0]{\rightarrow}
\newcommand{\vectoro}{\overline}
\newtheorem{theorem}{Theorem}[section]
\newtheorem{partialtheorem}{Partial Theorem}[section]
\newtheorem{conj}[theorem]{Conjecture}
\newtheorem{lemma}[theorem]{Lemma}
\newtheorem{proposition}[theorem]{Proposition}
\newtheorem{corollary}[theorem]{Corollary}
\newtheorem{claim}[theorem]{Claim}
\newtheorem{formal}[theorem]{Critical point derivation}
\newtheorem{experiment}[theorem]{Experimental Result}
\newtheorem{question}{Question}

\def\todo#1{\marginpar{\raggedright\footnotesize #1}}
\def\change#1{{\color{green}\todo{change}#1}}
\def\note#1{\textup{\textsf{\color{blue}(#1)}}}

\theoremstyle{definition}
\newtheorem{remark}[theorem]{Remark}

\theoremstyle{definition}
\newtheorem{example}[theorem]{Example}

\theoremstyle{definition}
\newtheorem{definition}[theorem]{Definition}

\theoremstyle{definition}
\newtheorem{definitions}[theorem]{Definitions}

\title{Random-walk in Beta-distributed random environment}

\begin{abstract}
We introduce an exactly-solvable model of random walk in random environment that we call the Beta RWRE. This is a random walk in $\Z$ which performs nearest neighbour jumps with transition probabilities drawn according to the Beta distribution. We also describe a related directed polymer model, which is a limit of the $q$-Hahn interacting particle system.
Using a Fredholm determinant representation for the quenched probability distribution function of the walker's position, we are able to prove second order cube-root scale corrections to the large deviation principle satisfied by the walker's position, with convergence to the Tracy-Widom distribution. We also show that this limit theorem can be interpreted in terms of the maximum of strongly correlated random variables: the positions of independent walkers in the same environment. The zero-temperature counterpart of the Beta RWRE can be studied in a parallel way. We also prove a Tracy-Widom limit theorem for this model. 
\end{abstract}


\author[G. Barraquand]{Guillaume Barraquand}
\address{G. Barraquand,
Columbia University,
Department of Mathematics,
2990 Broadway,
New York, NY 10027, USA}
\email{barraquand@math.columbia.edu}

\author[I. Corwin]{Ivan Corwin}
\address{I. Corwin, Columbia University,
Department of Mathematics,
2990 Broadway,
New York, NY 10027, USA,
and Clay Mathematics Institute, 10 Memorial Blvd. Suite 902, Providence, RI 02903, USA,
and Massachusetts Institute of Technology,
Department of Mathematics,
77 Massachusetts Avenue, Cambridge, MA 02139-4307, USA}
\email{ivan.corwin@gmail.com}

\maketitle

\begin{center}This is an updated version from May 2021  correcting minor mistakes from the published version of this paper \cite{barraquand2017random}. \end{center} 
\setcounter{tocdepth}{1}
\tableofcontents
\hypersetup{linktocpage}

We study an exactly solvable one-dimensional random walk in space-time i.i.d. random environment. It is a random walk on $\Z$ which performs nearest neighbour steps, according to transition probabilities following the Beta distribution and drawn independently at each time and each location. We call this model the Beta RWRE. Using methods of integrable probability, we find an exact Fredholm determinantal formula for the Laplace transform of the quenched probability distribution of the walker's position. An asymptotic analysis of this formula allows to prove  a very precise limit theorem. It was already known that such a random walk satisfies a quenched large deviation principle  \cite{rassoul2013quenched}. We show that for the Beta RWRE, the second order correction to the large deviation principle fluctuates on the cube-root scale with Tracy-Widom statistics. This brings the scope of KPZ universality to random walks in dynamic random environment, and the Beta RWRE is the first RWRE for which such a limit theorem has been  proved. Moreover, our result translates in terms of the maximum of the locations of independent walkers in the same environment. Hence, the Beta RWRE can also be considered as a toy model for studying maxima of strongly correlated random variables.

Our route to discover the exact solvability of the Beta RWRE was through an equivalent directed polymer model with Beta weights, which is itself a limit of the $q$-Hahn TASEP (introduced in \cite{povolotsky2013integrability} and further studied in \cite{corwin2014q}). However, we show that the RWRE/polymer model can be analysed independently of its interacting particle system origin, via a rigorous variant of the replica method. 

Our work generalizes a study of similar spirit, where a limit of the discrete-time geometric $q$-TASEP \cite{borodin2013discrete} was related to the strict weak lattice polymer \cite{corwin2014strict} (see also \cite{o2014tracy}). It should be emphasized that this procedure of translating the algebraic structure of interacting particle systems to directed polymer models was already fruitful in \cite{borodin2014macdonald}, where formulas for the $q$-TASEP allowed to study the law of continuous directed polymers related to the KPZ equation.
  
\section{Definitions and main results} 
\label{sec:intro}

\subsection{Random walk in space-time i.i.d. Beta environment}
\begin{definition}
\label{def:BetaRWRE}
Let $(B_{x, t})_{x\in \Z, t\in \Z_{\geqslant 0}}$ be a collection of independent random variables following the Beta distribution, with parameters $\alpha$ and $\beta$. We call this collection of random variables the environment of the walk.  Recall that if a random variable $B$ is drawn according to the $Beta(\alpha, \beta)$ distribution,  then for $0\leqslant r \leqslant 1$,
$$ \mathbb{P}\left( B\leqslant r\right) = \int_0^r x^{\alpha-1} (1-x)^{\beta-1} \frac{\Gamma(\alpha+\beta)}{\Gamma(\alpha)\Gamma(\beta)}\ \mathrm{d}x.$$
In this environment, we define the random walk in space-time Beta environment (abbreviated Beta-RWRE) as a random walk $(X_t)_{t\in \Z_{\geqslant 0}}$ in $\Z$, starting from $0$ and such that 
\begin{itemize}
\item $X_{t+1}=X_t +1$ with probability $B_{X_t, t}$ and
\item $X_{t+1}=X_t -1$ with probability $1- B_{X_t, t}$.
\end{itemize}
A sample path is depicted in Figure \ref{fig:betaRWRE}. We denote by $\mathsf{P}$ and $\mathsf{E}$ (resp. $\mathbb{P}$ and $\mathbb{E}$) the measure and expectation associated to the random walk (resp. to the environment).
\end{definition}

Let $P(t, x)=\mathsf{P}(X_t\geqslant x)$. This is a random variable with respect to $\mathbb{P}$. Our first aim is to show that the Beta RWRE model is exactly solvable, in the sense that we are able to find the distribution of $P(t,x)$, by exploiting an exact formula for the Laplace transform of $P(t,x)$. 

\begin{remark}
The random walk $(\mathbf{X}_t)_t$ in $\Z^2$, where $\mathbf{X}_t:=(t, X_t)$ is a random walk in random environment in the classical sense, i.e. the environment is not dynamic (see Figure \ref{fig:betaRWRE}). It is a very particular case of random walk in Dirichlet random environment \cite{enriquez2006random}. Dirichlet RWREs have generated some interest  because it can be shown using connections between Dirichlet law and P\'olya urn scheme that the annealed law of such random walks is the same as that of oriented-edge-reinforced random walks \cite{enriquez2002edge}. However, since the random walk $(\mathbf{X}_t)$ can go through a given edge of $\Z^2$ at most once, the connection to self-reinforced random walks is irrelevant for the Beta RWRE.  
\label{rem:RWinZ2}
\end{remark}

\begin{figure}
\begin{tikzpicture}[scale=0.7]
\clip (-0.8, -5.5) rectangle (15.5, 5.5);
\draw[gray, dotted] (-1, -10) grid (16, 10);
\fill (0,0) circle(0.1);
\draw (-0.3, -0.3) node{$0$};
\begin{scope}[rotate=45, scale= sqrt(2)]
\draw[gray] (0, -11) grid (11, 0);
\draw[ultra thick] (0,0) -- (1, 0) -- (1, -1) -- (1, -2) -- (2, -2) -- (3, -2) -- (4, -2) -- (4, -3) -- (4, -4) -- (5, -4) -- (5, -5) -- (5, -6)-- (5, -7) -- (5, -8) -- (6, -8)--(6, -9)-- (7, -9);
\draw[->, ultra thick, gray] (2, -4) --(3, -4);
\draw[->, ultra thick, gray] (2, -4) --(2, -5);
\fill[gray] (2, -4) circle(0.05);
\draw[gray] (2, -4) node[anchor=east] {\footnotesize{$(x, t)$}};
\draw[gray] (2.5, -3.7) node{\footnotesize{$B_{x, t}$}};
\draw[gray] (1.5, -4.5) node{\footnotesize{$1-B_{x, t}$}};
\end{scope}
\draw[->] (0,-5.5) -- (0,5.5);
\draw (0,5.2) node[anchor=west] {$X_t$};
\draw[->] (-1,0) -- (15.5,0);
\draw (15.5,0) node[anchor=north east] {$t$};
\end{tikzpicture}
\caption{The graph of $t\mapsto X_t$ for the Beta RWRE. 
One sees that that the random walk  $\mathbf{X}_t:=(t, X_t)$ is also a (directed) random walk in a random environment in $\Z^2$.
}
\label{fig:betaRWRE}
\end{figure}
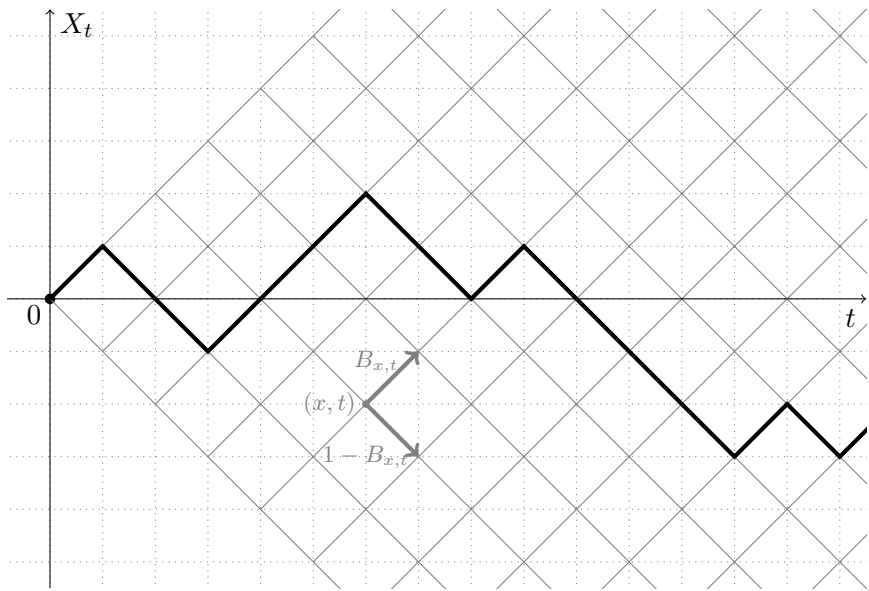

\begin{remark}\mbox{}
\begin{itemize}
\item The Beta distribution with parameters $(1,1)$ is the uniform distribution on $(0,1)$.
\item For $B$ a random variable with $Beta(\alpha, \beta)$ distribution, $1-B$ is distributed according to a Beta distribution with parameters $(\beta, \alpha)$. Consequently, exchanging the parameters $\alpha$ and $\beta$ of the Beta RWRE corresponds to applying a symmetry with respect to the time axis. 
\end{itemize}
\end{remark}

\subsection{Definition of the Beta polymer}

\subsubsection{Point to point Beta polymer}
\begin{definition}
\label{def:beta}
A point-to-point Beta polymer is a measure $Q_{t,n}$ on lattice paths $\pi$ between $(0,1)$ and $(t,n)$. At each site $(s, k)$ the path is allowed to 
\begin{itemize}
\item jump horizontally to the right from $(s,k) $ to $(s+1,k)$, 
\item or jump diagonally to the upright from $(s,k) $ to $(s+1,k+1)$.
\end{itemize}
An admissible path is shown in Figure \ref{fig:betapolymer}. Let $B_{i,j}$ be independent random variables distributed according to the Beta distribution with parameters $\mu$ and $\nu-\mu$ where $0<\mu <\nu$.
The measure $Q_{t,n}$ is defined by 
$$ Q_{t,n}\left( \pi \right) = \frac{\prod_{e\in \pi} w_e}{Z(t,n)} $$
where the products is taken over edges of $\pi$ and the weights $w_e$ are defined by 
$$ w_e = \begin{cases} B_{ij} &\mbox{if } e=(i-1,j)\to(i,j) \\
1 &\mbox{if }  e=(i-1, i)\to(i,i+1)\\
1-B_{i,j}&\mbox{if } e=(i-1,j-1)\to(i,j) \mbox{ with }i\geqslant j,
\end{cases} $$
and $Z(t,n)$ is a normalisation constant called the partition function,
$$ Z(t,n) = \sum_{\pi : (0,1)\to(t,n)} \prod{w_e}. $$
The free energy of the beta polymer is $\log Z(t,n)$. The partition function of the beta polymer satisfies the recurrence 
\begin{equation}
\begin{cases}
Z(t, n)  = Z(t-1,n) B_{t,n} + Z(t-1,n-1)(1- B_{t,n}) & \text{for } t \geqslant n > 1 ,\\
Z(t,t+1) = Z(t-1, t) &\text{for }t>0, \\
Z(t,1) = Z(t-1,1) B_{t,1} &\text{for }t>0.
\end{cases}
\label{eq:recurrencerelation}
\end{equation}
with the initial data 
\begin{equation}
 Z(0,1) = 1.
 \label{eq:initialconditionpoint2point}
\end{equation}
\end{definition}
\begin{figure}
\begin{tikzpicture}[scale=0.6]
\draw[->, thick, gray] (-2, 8) node[anchor= east]{$\tilde{Z}(s,k)=1$}   to[bend left] (1.9, 6.1);
\fill[gray] (2, 6) circle(0.07);
\draw[thick, gray] (2, 6) node[anchor=south west]{\footnotesize{$(s,k)$}};
\draw[thick, gray] (0, 6) --(2, 6);
\draw[thick, gray] (0,5) -- (1,5);
\draw[thick, gray] (0,4) -- (2,6);
\draw[thick, gray] (0,5) -- (1,6);
\draw[->, thick] (0, -1) -- (0, 10.3);
\draw[->, thick] (-1,0)--( 16.3, 0);
\draw[thick] (0,1) -- (9.1,10.1);
\draw[thick] (0,1) -- (16.1,1);
\foreach \k in {1,2, ..., 16}
	{\draw[gray, dotted] (\k, 0) -- (\k, 10.1);}
\foreach \k in { 2, 3, ..., 10}
	{\draw[gray] (\k-1,\k) -- (16.1, \k);}
\foreach \k in {2,3, ..., 10}
	{\draw[gray, dotted] (0,\k) -- (\k - 1, \k);}
	\clip (-2, -1) rectangle (16.1, 10.3);
\foreach \k in{ 1, 2, ..., 16}
	{\draw[gray] (\k, 1) -- (\k+9.1,10.1);}
\fill (0, 3) circle(0.1);
\draw[ultra thick] (3,4) -- (6, 4) -- (7, 5) -- (9,5) -- (12, 8) -- (14,8) --(15,9);
\draw[ultra thick, dotted] (0, 3) -- (1, 3)-- (2, 4) -- (3,4) ;
\draw[ultra thick] (0,1) -- (3,4);
\fill (15,9) circle(0.1);
\fill (0,1) circle(0.1);
\fill[gray, opacity=0.5] ( -0.1, 11) -- ( -0.1, 1) -- (0.1, 1) -- (0.1, 11) ;

\draw (16,0) node[anchor=north]{$t$};
\draw (0, 10) node[anchor=east]{$n$};
\draw(15,9) node[anchor=south]{$Z(t, n)$};
\draw (0,1) node[anchor=east]{$(0,1)$};

\end{tikzpicture}
\caption{The thick line represents a possible polymer path in the point-to-point Beta polymer model. The dotted thick part represents a modification of the polymer path that is admissible if one considers the half-line to point polymer (see the paragraph  \ref{subsubsec:halflinepolymer}). 
The partition function for the half-line to point model $\tilde{Z}(s,k)$ at the point $(s,k)$ shown in gray equals $1$.}
\label{fig:betapolymer}
\end{figure}
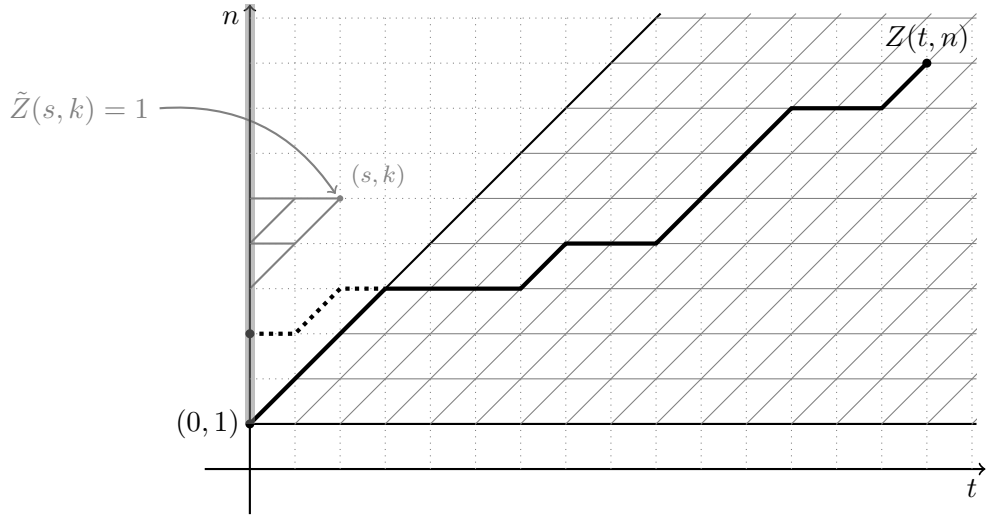
\begin{remark}
One recovers  at the  $\nu\to \infty$ limit the strict-weak lattice polymer described in \cite{o2014tracy, corwin2014strict}. As $\nu$ goes to infinity,  
$$\nu \cdot Beta(\mu, \nu-\mu) \Rightarrow Gamma(\mu),$$ 
and $1-Beta(\mu, \nu-\mu)\Rightarrow 1$. There are $t-n+1$ horizontal edges in any admissible lattice path from $(0,1)$ to $(t,n)$, and thus  
\begin{equation*}
\bar{Z}(t,n) :=\lim_{\nu\to\infty} \nu^{t-n+1} Z(t,n)
\end{equation*}
is the partition function of the strict-weak polymer. Indeed, in the strict-weak polymer, the horizontal edges have weights $Gamma(\mu)$ whereas upright paths have weight $1$.
\end{remark}
\subsubsection{Half-line to point Beta polymer}
\label{subsubsec:halflinepolymer}
Another (equivalent) possible interpretation of the same quantity $Z(t,n)$ is the partition function of an ensemble of polymer paths starting from the ``half-line'' $\big\lbrace (0,n) : n>0\big\rbrace $.
Fix $t\geqslant 0$ and $n>0$. One considers paths starting from any point $(0,m)$ for $0<m\leqslant n$ and ending at $(t,n)$. As for the point-to-point Beta polymer, paths are allowed to make right and diagonal steps. 
The weight of any path is the product of the weights of each edge along the path, and the weight $\tilde{w}_e$ of the edge $e$ is now defined by 
$$ \tilde{w}_e = \begin{cases} B_{ij} &\mbox{if } e\text{ is the horizontal edge }(i-1,j)\to(i,j), \\
1-B_{i,j}&\mbox{if } e\text{ is the diagonal edge }(i-1,j-1)\to(i,j).
\end{cases}$$
Let us denote by $\tilde{Z}(t,n)$ the partition function in the half-line to point model. It is characterized by the recurrence 
$$ \tilde{Z}(t, n)  = \tilde{Z}(t-1,n) B_{t,n} + \tilde{Z}(t-1,n-1)(1- B_{t,n}) $$
for all $t,n>0$ and the initial condition $Z(0,n)=1$ for $n>0$. 
With the above definition of weights, we can see by induction that for any $t\geqslant 0$ and $n>t$, $\tilde{Z}(t, n)=1$. For example, in Figure \ref{fig:betapolymer}, the possible paths leading to $(s,k)$ are shown in gray. On the figure, one has
$$ \tilde{Z}(s,k)= \tilde{Z}(2,6)  = B_{1,6}B_{2,6} + (1-B_{1,6})B_{2,6} + B_{1,5}(1-B_{2,6}) + (1-B_{1,5})(1-B_{2,6}) =1.$$
Consequently, the partition functions of the half-line-to-point and the point-to-point model coincide for $t+1\geqslant n$. In the following, we drop the tilde above $Z$, even when considering the half-line-to point model, since the models are equivalent.

By deforming the lattice so that admissible paths are up/right, and reverting the orientation of the path, one sees that the Beta polymer and the Beta-RWRE are closely related models, in the sense of Proposition \ref{prop:equivRWREpolymer}. This proposition is proved in Section \ref{subsec:equivRWREpolymer}.
\begin{proposition}
Consider the Beta-RWRE with parameters $\alpha, \beta >0$ and the Beta polymer with parameters $\mu=\alpha$ and $\nu=\alpha+\beta$. 
For any fixed $t, n\in \Z_{\geqslant 0}$ such that $t+1\geqslant n$, then we have the equality in law
$$ Z(t, n)  = P\big(t, t-2n+2\big) .$$
Moreover, conditioning on the environment of the Beta polymer corresponds to conditioning on the environment of the Beta RWRE. 
\label{prop:equivRWREpolymer}
\end{proposition}

\subsection{Bernoulli-Exponential directed first passage percolation}
Let us introduce the ``zero-temperature'' counterpart of the Beta RWRE. 
\begin{definition}
\label{def:fpp}
Let $(E_e)$ be a family of independent exponential random variables indexed by the horizontal and vertical edges $e$ in the lattice $\mathbb{Z}^2$, such that $E_e$ is distributed according to the exponential law with  parameter $a$ (i.e. with mean $1/a$) if $e$ is a vertical edge and $E_e$ is distributed according to the exponential law with parameter $b$ if $e$ is a horizontal edge. Let $(\xi_{i,j})$ be a family of independent Bernoulli random variables with parameter $ b/(a+b)$. For an edge $e$ of the lattice $\Z^2$, we define the the passage time $t_e$ by
\begin{equation}
t_e= \begin{cases}  
\xi_{i,j}E_e \text{ if }e\text{ is the vertical edge }(i,j) \to (i, j+1),\\
(1-\xi_{i,j})E_e\text{ if }e\text{ is the horizontal edge }(i,j)\to(i+1, j).
\end{cases}
\label{eq:defte}
\end{equation}
The first passage-time $T(n,m)$ in the Bernoulli-Exponential first passage percolation model is given by 
$$ T(n, m)  =  \min_{\pi :(0, 0)\to D_{n, m}}\  \sum_{e\in\pi} \  t_e ,$$
where the minimum is taken over all up/right paths $\pi$ from $(0,0)$ to $D_{n, m}$, which is the set of points
$$ D_{n, m} = \Big\lbrace (i, n+m-i) : 0\leqslant i\leqslant n\Big\rbrace .$$ 
\end{definition}
\begin{figure}
\begin{tikzpicture}[scale=0.6]
\fill (0,0) circle (0.1);
\draw[->, thick] (0,0) -- (15.5, 0);
\draw[->, thick] (0,0) -- (0,15.5);
\foreach \k in {1, 2, ..., 14}
	{\draw[dashed] (\k, 0) -- (\k, 15-\k);} 
\foreach \k in {1, 2, ..., 14}
	{\draw[dashed] (0,\k) -- ( 15-\k, \k);} 
\draw[ultra thick, gray] (0, 15) -- (10, 5);
\draw[dotted, thick, gray] (10,5) -- (15,0);
\draw[ultra thick] (0,0) -- (0,1) -- (2,1) -- (2, 4) -- (4,4) -- (4,5) -- (8,5) -- (8,6) -- (9,6);
\fill (9,6) circle (0.1);
\fill[gray] (10,5) circle (0.1);
\fill[gray, opacity=0.2] (0,5) -- (10,5) -- (0,15);
\draw[gray, dashed, ultra thick] (0,5) -- (10,5);
\draw (10,0.1) -- (10, -0.1) node[anchor=north]{$n$};
\draw (0.1, 5) -- ( -0.1, 5) node[anchor=east]{$m$};
\draw[gray] (6, 10) node{$D_{n,m}$};
\draw[gray] (-2, 7) node{$\tilde{D}_{n,m}$};
\draw[gray, ->, thick] (-1, 7)  to[bend left] (2.5, 5.1);
\draw (-0.5, -0.5) node{$(0,0)$};
\end{tikzpicture}
\caption{An admissible path for the Bernoulli-Exponential FPP model is shown on the figure. $T(n,m)$ is the passage time between $(0,0)$ and $D_{n,m}$ (thick gray line). Note that the first passage time to $D_{n,m}$ is also the first passage time to $\tilde{D}_{n,m}$ depicted in dotted gray on the figure (cf Remark \ref{rem:equivalentline}).}
\label{fig:FPP}
\end{figure}
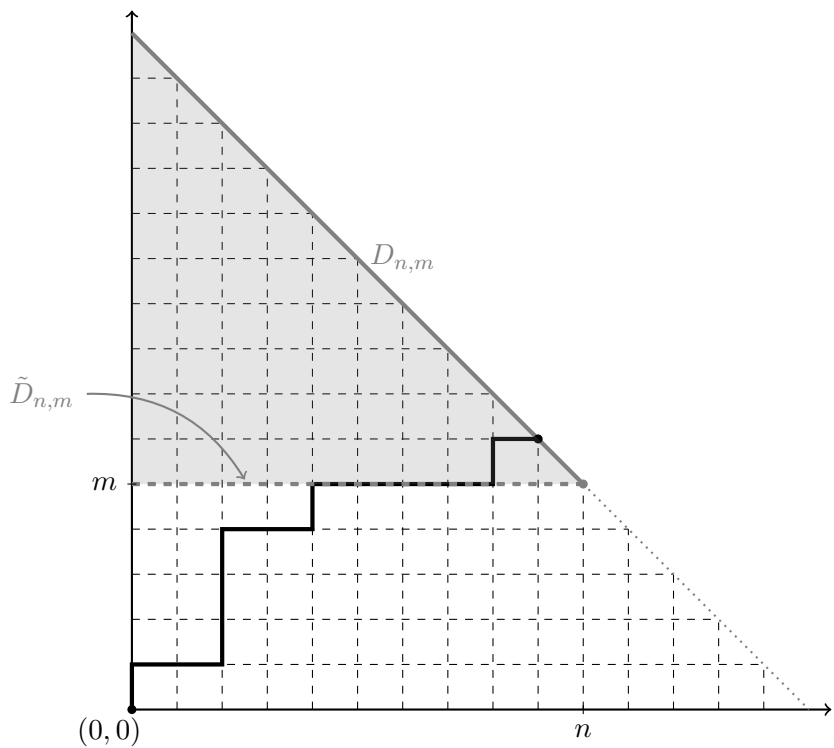
Although the quantity that we are fully able to study is $T(n,m)$, that is a point to half-line passage time, is is also natural to introduce the point-to-point passage time $T^{pp}(n,m)$ defined by 
$$ T^{pp}(n, m)  =  \min_{\pi :(0, 0)\to (n,m)}\  \sum_{e\in\pi} \  t_e, $$
where the maximum is taken over paths between the points $(0,0)$ and $(n,m)$. 
We define the percolation cluster $C(t)$ by
$$ C(t) = \Big\lbrace (n,m) : T^{pp}(n,m) \leqslant t\Big\rbrace.$$
It can be constructed in a dynamic way (see Figure \ref{fig:FPP4shades}). At each time $t$, $C(t)$  is the union of points visited by (portions of) several directed up/right random walks in the quarter plane $\Z_{\geqslant 0}^2$. The evolution is as follows:
\begin{itemize}
\item At time $0$, the percolation cluster contains the points of the path of a directed random walk starting from $(0,0)$.

 Indeed, since for any $i,j$, $\xi_{i,j}$ is a Bernoulli random variable in $\lbrace 0,1\rbrace$, either the passage time from $(i,j)$ to $(i+1,j)$ is zero, or the passage time from $(i,j)$ to $(i,j+1)$ is zero. This implies that there exists a unique infinite up-right path starting from $(0,0)$ with zero passage-time. This path is distributed as a directed random walk. 
\item At time $t$, from each point on the boundary of the percolation cluster where a random walk can branch, we add to the percolation cluster after an exponentially distributed waiting time, the path of that random walk. Paths starting with a vertical (resp. horizontal) edge are added at rate $a$ (resp. $b$). This random walk almost surely crosses the percolation cluster somewhere, and we add to the percolation cluster only the points of the walk path up to the first hitting point. 

Indeed, any edge $e=(x,y)$ from a point $x$ inside $C(t)$ to a point $y$ outside $C(t)$, has a positive passage time. Hence, one adds the point $y$ to the percolation cluster after an exponentially distributed waiting time $t_e$. Once the point $y$ is added, one immediately adds to $C(t)$ all the points that one can reach from $y$ with zero passage time. These points form a portion of random walk that will almost surely coalesce with the initial random walk path $C(0)$. 
\end{itemize}

\begin{remark}
Denote by $\tilde{D}_{n,m}$ the set of points $\big\lbrace (i, m)\  : 0\leqslant i\leqslant n\big\rbrace$ (see Figure \ref{fig:FPP}). Any path going from $(0,0)$ to  $D_{n,m}$ has to go through a point of  $\tilde{D}_{n,m}$. Moreover, the first passage time from any point of $\tilde{D}_{n,m}$  to the set $D_{n,m}$ is zero.  Hence the first passage time from $(0,0)$ to $\tilde{D}_{n,m}$ is also $T(n,m)$.
\label{rem:equivalentline}
\end{remark}
\begin{remark}
When $b$ tends to infinity, $E_e$ tends to $0$ for all horizontal edges, %
 and one recovers the first passage percolation model introduced in \cite{o1999directed}, which is the zero temperature limit of the strict-weak lattice polymer as explained in \cite{o2014tracy, corwin2014strict}.
\end{remark}

Let us show how the Bernoulli-Exponential first passage percolation model is a limit of the Beta RWRE. 
\begin{proposition}
Let $\alpha_{\epsilon}=\epsilon a$ and $\beta_{\epsilon} = \epsilon b$. Let $P_{\epsilon}(t,x)$ be the probability distribution function of the Beta-RWRE with parameters $\alpha_{\epsilon}$ and $\beta_{\epsilon}$ and $T(n, m)$ the first-passage time in the Bernoulli-Exponential FPP model with parameters $a,b$. Then, for all $ n, m\geqslant 0$, $-\epsilon \log(P_{\epsilon}(n+m,m-n))$ weakly converges as $\epsilon$ goes to zero to $T(n,m)$, the first passage time from $(0,0)$ to $D_{n,m}$ in the Bernoulli-Exponential FPP model.
\label{prop:RWREtoFPP}
\end{proposition}
Proposition \ref{prop:RWREtoFPP} is proved in Section \ref{sec:limitBetatoFPP}.

\begin{figure}
\includegraphics[scale=0.8]{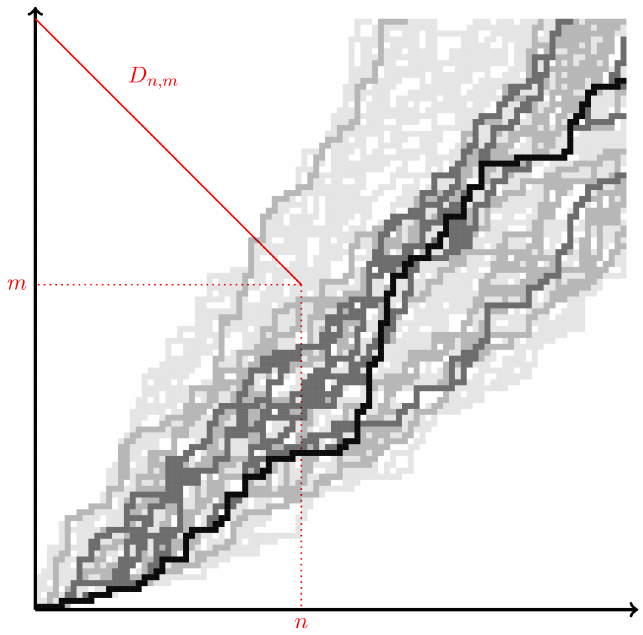}
\caption{Percolation cluster for the Bernoulli-Exponential model with parameters $a=b=1$ in a grid of size $100\times 100$. The different shades of gray correspond to different times: the black line corresponds to the percolation cluster at time $0$ and the other shades of gray corresponds to times $0.2$, $0.5$ and $1.2$. This implies that for $n$ and $m$ chosen as on the figure, 
$ 0.2\leqslant T(n,m)\leqslant 0.5$. }
\label{fig:FPP4shades}
\end{figure}

\subsection{Exact formulas}

Our first result is an exact formula for the mixed moments of the polymer partition function $\mathbb{E}\big[Z(t,n_1) \cdots Z(t,n_k) \big]$. In light of Proposition \ref{proprecursionlimit}, this result can be seen as a limit when $q$ goes to $1$ of the formula from Theorem 1.8 in \cite{corwin2014q}. Even so, we prove this in an independent way in Section \ref{sec:directproof} via a rigorous polymer replica trick methods (See Proposition \ref{prop:momentformula}). 
\begin{proposition}
\label{prop:momentformulaintro}
For $n_1 \geqslant n_2\geqslant \dots \geqslant n_k\geqslant 1$, one has the following moment formula,
\begin{multline}
\mathbb{E}\Big[Z(t,n_1) \cdots Z(t,n_k) \Big] =
 \frac{1}{(2i\pi)^k} \int\dots\int \prod_{1\leqslant A<B\leqslant k} \frac{z_A-z_B}{z_A-z_B-1}\prod_{j=1}^k \left( \frac{\nu+z_j}{z_j}\right)^{n_j} \left(\frac{\mu + z_j}{\nu+z_j}\right)^t\frac{\mathrm{d}z_j}{\nu+z_j},
\label{eq:momentformulapolymer}
\end{multline}
where the contour for $z_k$ is a small circle around the origin, and the contour for $z_j$ contains the contour for $z_{j+1} + 1$ for all $j=1, \dots , k-1$, as well as the origin, but all contours exclude $-\nu$. 
\end{proposition}
The previous proposition provides a formula for the moments of the partition function $Z(t,n)$. Using tools developed in the study of Macdonald processes \cite{borodin2014macdonald} (See also \cite{dotsenko2010replica, calabrese2010free}), one is able to take the moment generating series, which yields a Fredholm determinant representation for the Laplace transform of $Z(t,n)$. We refer to \cite[Section 3.2.2]{borodin2014macdonald} for background about Fredholm determinants.  
\begin{theorem}
For $u\in\mathbb{C} \setminus \mathbb{R}_{>0}$, fix $n,t\geqslant 0$ with $n\leqslant t+1$ and $\nu >\mu>0 $. Then one has
\begin{equation*}
\mathbb{E}\left[ e^{uZ(t,n)} \right] = \det(I-K^{\mathrm{BP}}_u)_{\mathbb{L}^2(C_0)}
\end{equation*}
where $C_0$ is a small positively oriented circle containing $0$ but not $-\nu$ nor $-1$, and $K^{\mathrm{BP}}_u : \mathbb{L}^2(C_0)\rightarrow \mathbb{L}^2(C_0)$ is defined by its integral kernel
\begin{equation*}
K^{\mathrm{BP}}_u(v,v') = \frac{1}{2i\pi} \int_{1/2-i\infty}^{1/2+i\infty} \frac{\pi}{\sin(\pi s)} (-u)^s\frac{g^{\mathrm{BP}}(v)}{g^{\mathrm{BP}}(v+s)} \frac{\mathrm{d}s}{s+v - v'}
\end{equation*} 
where 
\begin{equation}
g^{\mathrm{BP}}(v) = \left(\frac{\Gamma(v)}{\Gamma(\nu+v)} \right)^n \left( \frac{\Gamma(\nu+v)}{\Gamma(\mu+v)}\right)^t \Gamma(\nu+ v).
\label{eq:defgbp}
\end{equation}
\label{thmLaplacepolymer}
\end{theorem}
In light of the relation between the Beta RWRE and the Beta polymer given in Proposition \ref{prop:equivRWREpolymer}, we have a similar Fredholm determinant representation for the Laplace transform of $P(t,x)$. 
\begin{theorem}\label{thmLaplaceRWRE}
For $u\in\mathbb{C} \setminus \mathbb{R}_{>0}$, fix $t\in \Z_{\geqslant 0}$, $x\in \lbrace -t, \dots, t\rbrace$ with the same parity, and $\alpha, \beta>0 $. Then one has
\begin{equation}
\mathbb{E}\left[ e^{u P(t,x)} \right] = \det(I-K^{\mathrm{RW}}_u)_{\mathbb{L}^2(C_0)}
\label{eq:fredholmRWRE}
\end{equation}
where $C_0$ is a small positively oriented circle containing $0$ but not $-\alpha-\beta$ nor $-1$, and $K^{\mathrm{RW}}_u : \mathbb{L}^2(C_0)\rightarrow \mathbb{L}^2(C_0)$ is defined by its integral kernel 
\begin{equation*}
K^{\mathrm{RW}}_u(v,v') = \frac{1}{2i\pi} \int_{1/2-i\infty}^{1/2+i\infty} \frac{\pi}{\sin(\pi s)} (-u)^s\frac{g^{\mathrm{RW}}(v)}{g^{\mathrm{RW}}(v+s)} \frac{\mathrm{d}s}{s+v - v'}
\end{equation*} 
where 
\begin{equation*}
g^{\mathrm{RW}}(v) = \left(\frac{\Gamma(v)}{\Gamma(\alpha+v)} \right)^{(t-x)/2} \left( \frac{\Gamma(\alpha+\beta+v)}{\Gamma(\alpha+v)}\right)^{(t+x)/2} \Gamma( v).
\end{equation*}
\end{theorem}

\subsection{Limit theorem for the random walk}

A quenched large deviation principle is proved in \cite[Section 4]{rassoul2013quenched} for a wide class of random walks in random environment that includes the Beta-RWRE model. More precisely, the setting of \cite{rassoul2013quenched} applies to the random walk $\mathbf{X}_t=(t, X_t)$ (see Remark \ref{rem:RWinZ2}). The condition that one has to check is that the logarithm of the probability of each possible step has nice properties with respect to the environment (The random variables must belong to the class $\mathcal{L}$ defined in \cite[Definition 2.1]{rassoul2013quenched}). Using the fact that if $B$ is a $Beta(\alpha, \beta)$ random variable, $\log(B)$ and $\log(1-B)$  have integer moments of any order, \cite[Lemma A.4]{rassoul2013quenched} ensures that the condition is satisfied. 
The limit
$$\lambda(z) := \lim_{t\to\infty} \frac{1}{t} \log\left(\mathsf{E}\left[ e^{zX_t}\right]\right)$$
exists $\mathbb{P}$-almost surely. Let $I$ be the Legendre transform of $\lambda$. Then, we have \cite[Section 4]{rassoul2013quenched} that for $x>(\alpha-\beta)/(\alpha+\beta)$,
\begin{equation}
 \lim_{t\to\infty} \frac{1}{t} \log\Big(\mathsf{P}(X_t>x t)\Big) = -I(x) \ \ \mathbb{P} \text{ a.s.}
 \label{eq:existenceLDP}
\end{equation}
\begin{remark}
 In the language of polymers, the limit (\ref{eq:existenceLDP}) states the existence of the quenched free energy. Theorem 4.3 in \cite{rassoul2014quenched} states that for such random walks in random environment, we have that 
$$\lim_{t\to\infty} \frac{1}{t} \log\Big(\mathsf{P}(X_t=\lfloor x t\rfloor)\Big)  =  \lim_{t\to\infty} \frac{1}{t} \log\Big(\mathsf{P}(X_t>x t)\Big)= -I(x).$$
In other terms, the point-to-point free energy and the point-to-half-line free energies are equal. 
\end{remark}
In \cite[Theorem 3.1]{rassoul2013quenched}, a formula is given for $I$ in terms of a variational problem over a space of measures. We provide a closed formula in the present case. It would be interesting to see how the variational problem recovers the formulas that we now present.

For the Beta-RWRE, critical point Fredholm determinant asymptotics shows that the function $I$ is implicitly defined by 
\begin{equation}
x(\theta) = \frac{\Psi_1(\theta+\alpha+\beta) +\Psi_1(\theta )- 2 \Psi_1(\theta + \alpha)}{\Psi_1(\theta) - \Psi_1(\theta + \alpha+\beta)}
\label{eq:defxintro}
\end{equation}
and 
\begin{equation}
I(x(\theta)) = \frac{\Psi_1(\theta+\alpha+\beta) - \Psi_1(\theta + \alpha)}{\Psi_1(\theta) - \Psi_1(\theta + \alpha+\beta)} \Big(\Psi(\theta + \alpha+\beta)- \Psi(\theta)   \Big)+ \Psi(\theta + \alpha+\beta)- \Psi(\theta+\alpha),  
\label{eq:defIintro}
\end{equation}
where $\Psi$ is the digamma function ($\Psi(z)= \Gamma'(z)/\Gamma(z)$) and $\Psi_1$ is the trigamma function ($\Psi_1(z)=\Psi'(z)$). The parameter $\theta$ does not seem natural at a first sight. It is convenient to use it as it will turn out to be the position of the critical point in the asymptotic analysis. When $\theta$ ranges from $0$ to $+\infty$, $x(\theta)$ ranges from $1$ to $(\alpha-\beta)/(\alpha+\beta)$. This covers all the interesting range of large deviation events since $(\alpha-\beta)/(\alpha+\beta)$ is the expected drift of the random walk, and we know that $\mathsf{P}(X_t>x t)=0$ for $x>1$. 

 Moreover, we define $\sigma(\theta)>0$ such that 
\begin{equation}
2\sigma(\theta)^3 = \Psi_2(\theta+\alpha) - \Psi_2(\alpha+\beta+\theta) + \frac{\Psi_1(\alpha+\theta) - \Psi_1(\alpha+\beta+\theta)}{\Psi_1(\theta) - \Psi_1(\alpha+\beta+\theta)}\left( \Psi_2(\alpha+\beta+\theta) - \Psi_2(\theta)\right).
\label{eq:defsigmaintro}
\end{equation}
In the case $\alpha=\beta=1$, that is when the $B_{x,t}$ variables are distributed uniformly on $(0,1)$, the expressions for $x(\theta)$ and $I(x(\theta))$ simplify. We find that 
$$ x(\theta) = \frac{1+2\theta}{\theta^2+(\theta+1)^2}$$
and
$$ I(x(\theta)) = \frac{1}{\theta^2+(\theta+1)^2} ,$$
so that the rate function $I$ is simply the function $I:x \mapsto 1-\sqrt{1-x^2}$. 

The following theorem gives a second order correction to the large deviation principle satisfied by the position of the walker at time $t$. 
\begin{theorem}
\label{thm:RWasymptoticsintro}
\label{thmTWBetaintro}
For $0<\theta<1/2$ and $\alpha=\beta=1$, we have that 
\begin{equation}
\lim_{t\to\infty} \PP\left( \frac{\log\Big(P\big(t, x(\theta)t\big)\Big) + I\big(x(\theta)\big)t}{t^{1/3}\sigma\big(x(\theta)\big)} \leqslant y \right)  = F_{\rm GUE}(y).
\end{equation}
\end{theorem}
\begin{remark}
As we explain in Section \ref{sec:RWREasymptotics}, we expect Theorem \ref{thm:RWasymptoticsintro} to hold more generally for arbitrary parameters $\alpha, \beta>0$ and $\theta>0$. The assumption $\alpha=\beta$ is made for simplifying the computations, whereas the assumption $\theta<1/2$ is present because certain deformations of contours are justified only for $\theta<\min\lbrace 1/2, \alpha+\beta\rbrace$. 
The condition $\theta>0$ is natural, it corresponds to looking at $x(\theta)<1$. We know that for $x(\theta)>1$, then $P(t, x(\theta)t)=0$. 

In the case $\alpha=\beta=1$, the condition $\theta<1/2$ corresponds to $x(\theta)>4/5$. 
\label{rem:restrictiveconditiontheta}
\end{remark}
\begin{remark}
The Tracy-Widom limit theorem from Theroem \ref{thm:RWasymptoticsintro} should be understood as an analogue of limit theorems for the free energy fluctuations of exactly-solvable random directed polymers. Similar results are proved in \cite{amir2011probability, borodin2012free} for the continuum polymer, in \cite{borodin2014macdonald, borodin2012free} for the O'Connell-Yor semi-discrete polymer, in \cite{borodin2013log} for the log-gamma polymer, and in \cite{o2014tracy, corwin2014strict} for the strict-weak-lattice polymer. 

In light of KPZ universality for directed polymers, we expect the conclusion of Theorem \ref{thm:RWasymptoticsintro} to be more general with respect to weight distribution, but this is only the first RWRE to verify this. 
\end{remark}

In Section \ref{sec:RWREasymptotics}, we also provide an interesting corollary of Theorem \ref{thmTWBetaintro}. Corollary \ref{cor:extreme} states that if one considers an exponential number of Beta RWRE drawn in the same environment, then the maximum of the endpoints satisfies a Tracy-Widom limit theorem. It turns out that even if the rescaled endpoint of a random walk converges in distribution to a Gaussian random variable for large $t$, the limit theorem that we get is quite different from the one verified by Gaussian random variables having the same dependence structure.

\subsection{Localization of the paths}
\label{subsec:localizationintro}

The localization properties of random walks in random environment are quite different from localization properties of random directed polymers in $1+1$ dimensions. For instance, in the log-gamma polymer model, the endpoint of a polymer path of size $n$ fluctuates on the scale $n^{2/3}$ \cite{seppalainen2012scaling}, and localizes in a region of size $\mathcal{O}(1)$ when one conditions on the environment\cite{comets2014loalization}. For random walks in random environment, it is clear by the central limit theorem that the endpoint of a path of size $n$ fluctuates on the scale $\sqrt{n}$. 

Remarkably, the central limit theorem also holds if one conditions on the environment. A general quenched central limit theorem is proved in \cite{rassoul2005almost} for space-time i.i.d. random walks in $\Z^d$. The only hypotheses are that the environment is not deterministic, and that the expectation over the environment of the variance of an elementary increment is finite. These two conditions are clearly satisfied by the Beta-RWRE model. In the particular case of one-dimensional random walks, and when transition probabilities have mean $1/2$, the result was also proved in \cite{berard2004almost}. However, most of the other papers proving a quenched central limit theorem for similar RW models assume a strict ellipticity condition, which is not satisfied by the Beta-RWRE. See also  \cite{rassoul2009almost, bouchet2014quenched} for similar results about random walks in random environment under weaker conditions.

In any case, if we let the environment vary, the fluctuations of the endpoints at time $t$ in the Beta RWRE live on the $\sqrt{t}$ scale. For the Beta-RWRE, Proposition \ref{prop:calculprobaoverlap} shows that the expected proportion of overlap between two random walks drawn independently in a common environment is of order $\sqrt{t}$ up to time $t$. The $\sqrt{t}$ order of magnitude has already been proved in \cite[Lemma 2]{rassoul2005almost} based on results from \cite{ferrari1998fluctuations}, and our Proposition \ref{prop:calculprobaoverlap} provides the precise equivalent. 

Let us give an intuitive argument explaining the difference of behaviour between polymers and random walks. Assume that the environment of the random walk (resp. the polymer) has been drawn, and consider a random walk starting from the point $0$ (resp. a point-to-point polymer starting from $0$). The quenched probability that the random walk performs a first step upward depends only on the environment at the point $0$ (i.e. the random variable $B_{0,0}$ in the case of the Beta RWRE). However, the probability  for the polymer path to start with a step upward depends on the global environment. For instance, if the weight on some edge is very high, this will influence the probability that the first step of the polymer path is upward or downward, so as to enable the polymer path to go through the edge with high weight. This explains why two independent paths in the same environment have more tendency to overlap in the polymer model. 

In \cite{georgiou2013ratios}, a random walk in dynamic random environment is associated to a random directed polymer in $1+1$ dimensions, under a condition called north-east induction on the edge-weights. For the log-gamma polymer, it turns out that the associated random walk has Beta distributed transition probabilities. However, the environment is correlated, so that this RWRE is very different from the Beta RWRE. The random walk considered in \cite{georgiou2013ratios} defines a measure on lattice paths which can be seen as a limit of point-to-point polymer measures. Hence, as pointed out in \cite[Remark 8.3]{georgiou2013ratios}, it has very different localization properties than random walks in space-time i.i.d random environment that we consider in the present paper. 

\subsection{Limit theorem at zero-temperature}

Turning to the zero-temperature limit, Theorem \ref{thmLaplaceRWRE} degenerates to the following for the Bernoulli-Exponential FPP model: 
\begin{theorem}\label{thmProb}
For $r\in\mathbb{R}_{>0}$, fix $n,m\geqslant 0$ and consider $T(n,m)$ the first passage time to the set $D_{n,m}$ in the Bernoulli-Exponential FPP model with parameters $a,b>0$.  Then, one has
\begin{equation*}
\mathbb{P}\Big( T(n,m) > r\Big) = \det(I-K^{\mathrm{FPP}}_r)_{\mathbb{L}^2(C'_0)}
\end{equation*}
where $C'_0$ is a small positively oriented circle containing $0$ but not $-a-b$, 
 and $K^{\mathrm{FPP}}_r : \mathbb{L}^2(C'_0)\rightarrow \mathbb{L}^2(C'_0)$ is defined by its integral kernel 
\begin{equation}
K^{\mathrm{FPP}}_r(u,u') = \frac{1}{2i\pi} \int_{1/2-i\infty}^{1/2+i\infty} \frac{e^{rs}}{s} 
\frac{g^{\mathrm{FPP}}(u)}{g^{\mathrm{FPP}}(u+s)} \frac{\mathrm{d}s}{s+u-u'},
\label{eq:defkernelFPPintro}
\end{equation}
where 
\begin{equation}
g^{\mathrm{FPP}}(u) = \left(\frac{a+u}{u}\right)^n \left(\frac{a+u}{a+b+u} \right)^m \frac{1}{u}.
\label{eq:defgfppintro}
\end{equation}
\end{theorem}
The integral in (\ref{eq:defkernelFPPintro}) is an improper oscillatory integral if one integrates on the vertical line $1/2+i\R$. One could justify a deformation of the integration contour (so that the tails go to $\infty e^{\pm i2\pi/3}$ for instance) in order to have an absolutely convergent integral, but it happens that the vertical contour is more practical for analyzing the asymptotic behaviour of $\det(I+K^{\mathrm{FPP}}_r)$ in Section \ref{sec:FPPasymptotics}. 

One has a Tracy-Widom limit theorem for the fluctuations of the first passage time $T(n,\kappa n)$ when $n$ goes to infinity, for some slope $\kappa>\frac{a}{b}$. Theorem \ref{thmTWFPPintro} is proved as Theorem \ref{thmTWFPP} in Section \ref{sec:FPPasymptotics}. 
\begin{theorem}
We have that for any $\theta>0$ and parameters $a,b>0$, 
$$ \lim_{n \to \infty}\mathbb{P}\left(\frac{T\big(n, \kappa(\theta)n\big) - \tau(\theta)n}{\rho(\theta)n^{1/3}}\geqslant -y \right) = F_{\mathrm{TW}}(y),$$
where $\kappa(\theta), \tau(\theta)$ and $\rho(\theta)$ are explicit constants (see Section \ref{sec:FPPasymptotics}) such that when $\theta$ ranges from $0$ to infinity, $\kappa(\theta)$ ranges from $+\infty$ to $a/b$. 
\label{thmTWFPPintro}
\end{theorem}
Notice that in Theorem \ref{thmTWFPPintro}, we do not have any restriction on the range of the parameters $a, b$ and $\theta$. 

Another direction of study for the Bernoulli-Exponential FPP model is to compute the asymptotic shape of the percolation cluster $C(t)$ for a fixed time $t$ (but looking very far from the origin). In Section \ref{subsec:nonrigorous} we explain, based on a degeneration of the results of Theorem \ref{thmTWFPPintro}, what should be the limit shape of the the convex envelope of the percolation cluster, and  guess the scale of the fluctuations. However, these arguments are based on a non-rigorous interchange of limits and we leave a rigorous proof for future consideration.

\subsection*{Acknowledgements}  G.B. would like to thank Vu-Lan Nguyen for interesting discussions. G.B. and I.C. thank Firas Rassoul-Agha and Timo Sepp{\"a}l{\"a}inen for useful comments on a first version of the paper. 
 
G.B. acknowledges support from the Laboratoire de Probabilit\'es et Mod\`eles Al\'eatoires, Universit\'e Paris-Diderot--Paris 7.   I.C. was partially supported by the NSF through DMS-1208998, the Clay Mathematics Institute through a Clay Research Fellowship, the Institute Henri Poincaré through the Poincaré Chair, and the Packard Foundation through a Packard Fellowship for Science and Engineering.

Since the publication of this paper as \cite{barraquand2017random}, we have noticed or been made aware of several minor mistakes. These mistakes do not affect significantly the main results and are all remedied in the present version. We are grateful to Mark Rychnovsky for noticing the mistake in the (former version of the) proof of Lemma \ref{lemma:steepdescentz} and a sign mistake in the statement of Theorem \ref{thmTWFPPintro} (this was noticed in \cite{barraquand2017tracy} which makes rigorous the claims from Section \ref{subsec:nonrigorous}). We also thank Thimothée Thiery and  Pierre Le Doussal for telling us about the sign mistake in Proposition \ref{prop:shrinkingcontours} that they have noticed in \cite[footnote 4 page 23]{thiery2016exact}.

\subsection*{Outline of the paper}
In Section \ref{sec:qhahntobeta}, we introduce the $q$-Hahn TASEP \cite{corwin2014q, povolotsky2013integrability} and show how some observables of the $q$-Hahn TASEP  converge to the partition function of the Beta polymer (and likewise endpoint distribution of the Beta RWRE). This enables us to give a first proof of the Fredholm determinant formulas in Theorems \ref{thmLaplacepolymer} and \ref{thmLaplaceRWRE}. In Section \ref{sec:directproof}, we give a direct proof of Theorems \ref{thmLaplacepolymer} and \ref{thmLaplaceRWRE} using an approach which can be seen as a rigorous instance of the replica method. In Section \ref{sec:limitBetatoFPP}, we show that the Beta RWRE converges to the Bernoulli-Exponential FPP, and prove the Fredholm determinant formula of Theorem \ref{thmProb}.  
In Section \ref{sec:RWREasymptotics} we perform an asymptotic analysis of the Fredholm determinant from Theorem \ref{thmLaplaceRWRE} to prove Theorem \ref{thmTWBetaintro}. We also discuss Corollary \ref{cor:extreme} which is  about the maximum of the endpoints of several Beta RWRE drawn in a common environment, and we relate this result to extreme value theory. In Section \ref{sec:FPPasymptotics}, we perform  an asymptotic analysis of the Bernoulli-Exponential FPP model to prove Theorem \ref{thmTWFPPintro}.

\section{From the $q$-Hahn TASEP to the Beta polymer}
\label{sec:qhahntobeta}

In this section, we explain how the Beta-RWRE and the Beta polymer arise as limits of the $q$-Hahn TASEP introduced in \cite{povolotsky2013integrability} (see also \cite{corwin2014q}). We first show that some observables of the $q$-Hahn TASEP converge to the partition function of the polymer model (Proposition \ref{proprecursionlimit}). This yields a first proof of Theorem \ref{thmLaplacepolymer}. Then we prove that the Beta RWRE and the Beta polymer model are equivalent models in the sense of Proposition \ref{prop:equivRWREpolymer}.

\subsection{The $q$-Hahn TASEP}
\label{subsec:defqHahn}
Let us recall the definition of the $q$-Hahn-TASEP: This is a discrete time interacting particle system on the one-dimensional integer lattice. Fix $0<q<1$ and $0\leqslant \bar{\nu} \leqslant \bar{\mu}<1$. Then the $N$-particle $q$-Hahn TASEP is a discrete time Markov chain $\vec{x}(t) = \lbrace x_n(t) \rbrace_{n=0}^{N} \in \mathbb{X}_N$ where the state space $\mathbb{X}_N$ is 
$$\mathbb{X}_N  = \Big\lbrace +\infty = x_0 > x_1 >\dots >x_N : \forall i, x_i\in \mathbb{Z} \Big\rbrace.$$
At time $t+1$, each coordinate $x_n(t)$ is updated independently and in parallel to $x_n(t+1) = x_n(t)+j_n$ where $0\leqslant j_n \leqslant x_{n-1}(t)-x_n(t)-1$ is drawn according to the $q$-Hahn probability distribution $\varphi_{q, \bar{\mu}, \bar{\nu}}(j_n\vert x_{n-1}(t)-x_n(t)-1) $. The $q$-Hahn probability distribution on $j\in\lbrace 0, 1, \dots ,m\rbrace$ is defined by the probabilities
\begin{equation}
\varphi_{q, \bar{\mu}, \bar{\nu}}(j\vert m) = \bar{\mu}^j \frac{(\bar{\nu}/\bar{\mu} ; q)_{j}(\bar{\mu} ; q)_{m-j}}{(\bar{\nu};q)_{m}} \frac{(q;q)_m}{(q;q)_{j}(q;q)_{m-j}},
\label{eq:defqhahnweights}
\end{equation}
where for $a\in\mathbb{C}$ and $n\in \mathbb{Z}_{\geqslant 0} \cup\lbrace +\infty\rbrace$, $(a;q)_n$ is the $q$-Pochhammer symbol
$$ (a ; q)_n = (1-a)(1-qa) \dots (1-q^{n-1}a).$$

\subsection{Convergence of the $q$-Hahn TASEP to the Beta polymer}
\label{subsec:convergenceqhahntoBeta}
An interesting interpretation of the $q$-Hahn distribution is provided in Section 4 of \cite{gnedin2009q}. The authors define a $q$-analogue of the P\'olya urn process: One considers two urns, initially empty, in which one sequentially adds balls. When the first urn contains $k$ balls, and the second urn contains $n-k$ balls, one adds a ball to the first urn with probability $ [\nu-\mu+n-k]_q/[\nu+n]_q $, where for any integer $m$, $[m]_q=(1-q^m)/(1-q)$ denotes the $q$-deformed integer, and we set $\bar{\mu} = q^{\mu}$ and $\bar{\nu} = q^{\nu}$. One adds a ball to the second urn with the complementary probability. Then $\varphi_{q, \bar{\mu}, \bar{\nu}}(j\vert m)$ is the probability that after $m$ steps, the first urn contains $j$ balls. When $q$ goes to $1$, one recovers the classical P\'olya urn process. 

For the classical P\'olya urn, it is known that after $n$ steps, the number of balls in the first urn is distributed according to the Beta-Binomial distribution. Further, the proportion of balls in the first urns converges in distribution to the Beta distribution when the number of added balls tends to infinity. Thus, it is natural to consider the $q$-Hahn distribution as a $q$-analogue of the Beta-Binomial distribution.  Further, we expect that if $X$ is a random variable drawn according to the $q$-Hahn distribution on $\lbrace 0, \dots, m\rbrace$ with parameters $(q, \bar{\mu}, \bar{\nu})$, the $q$-deformed proportion $[X]_q /[m]_q$  converges as $m$ goes to infinity to a $q$ analogue of the Beta distribution, which converges as $q$ goes to $1$ to the Beta distribution with parameters  $(\nu-\mu, \mu)$. 

Now, we show that the partition function of the Beta polymer is a limit of observables of the $q$-Hahn TASEP. Let $F^{\epsilon}(t,n)$ be the rescaled quantity
\begin{equation}
\label{eq:defFepsilon}
F^{\epsilon}(t,n) = -\epsilon (x_n(t)+n ), 
\end{equation} 
where $x_n(t)$ is the location of the $n^{th}$ particle in $q$-Hahn TASEP and we set $q=e^{-\epsilon}, \bar{\mu} = q^{ \mu}$ and $ \bar{\nu} = q^{ \nu}$.  

\begin{proposition}
\label{proprecursionlimit}
For $t\geqslant 0$ and $n\geqslant 1$ such that  $n\leqslant t+1$, the sequence of random variables $\left(F^{\epsilon}(t,n) \right)_{\epsilon}$ converges in distribution as $\epsilon \to 0$ to a limit  $F(t,n)$ and one has 
$$ e^{F(t,n)} = e^{F(t-1,n)} B_{t,n} + e^{F(t-1,n-1)}(1- B_{t,n})$$
where $B_{t,n}$ are i.i.d. Beta distributed random variables with parameters $(\mu, \nu-\mu)$. 
Additionally, we have the weak convergence of processes
\begin{equation*}
 \lbrace e^{F^{\epsilon}(t,n)}\rbrace_{t\geqslant 0, n\geqslant 1} \Rightarrow  \lbrace Z(t,n) \rbrace_{t\geqslant 0, n\geqslant 1}.
\end{equation*}
\end{proposition}
\begin{proof}
We first state a lemma useful for taking limits of $q$-Pochhammer symbols. 
\begin{lemma}
For $r,q\in(0,1) $ and $x>0$, we have that 
$$ \frac{(r ; q)_{\infty}}{(rq^x ; q)_{\infty}} \xrightarrow[q\to 1]{}  (1-r)^{x}.$$
\label{lemma:limitpochhammer}
\end{lemma}
\begin{proof}
We take the limit of $\log\left( \frac{(r ; q)_{\infty}}{(rq^x ; q)_{\infty}} \right)$. Since $r,q\in (0,1)$, one can use the series expansion of the logarithm around $1$. This yields
\begin{align*}
\log\left( \frac{(r ; q)_{\infty}}{(rq^x ; q)_{\infty}} \right) &= \sum_{i=0}^{\infty} \log\left( \frac{1-rq^i}{1-rq^{x+i}}\right)\\
&= \sum_{i=0}^{\infty} -\left(\sum_{j=1}^{\infty} \left( \frac{(rq^i)^j}{j} - \frac{(rq^{x+j})^j}{j}\right)\right)\\
&= \sum_{j=1}^{\infty}\frac{-r^j}{j}\left(\sum_{i=0}^{\infty} (q^j)^i - q^{xj} \sum_{i=0}^{\infty} (q^j)^i\right)\\
&= - \sum_{j=1}^{\infty}\frac{r^j}{j}\frac{1-q^{xj}}{1-q^j}\\
&\xrightarrow[q\to 1]{} x \log(1-r).
\end{align*}
The last convergence comes from term-wise convergence as $q\to 1$ along with absolute convergence of the sum for $r\in(0,1)$. 
\end{proof}

\begin{lemma}
The sequence of random variables $\exp(F^{\epsilon}(1,1))$ converges as $\epsilon \to 0$ to a Beta distributed random variable with parameters $(\mu, \nu-\mu)$.
\label{lemma:betalimit1}
\end{lemma}
\begin{proof}
By the definition of $F^{\epsilon}(t,n)$ given in Equation (\ref{eq:defFepsilon}),
 $$\exp(F^{\epsilon}(1,1))= r \Longleftrightarrow x_1(1)+1 = -\epsilon^{-1}\log r .$$
 For $r$ such that $-\epsilon^{-1}\log(r)\in \Z$, 
 \begin{align*}
 \mathbb{P}\left(x_1(1)+1 = -\epsilon^{-1}\log r \right) &= \varphi_{q, \bar{\mu}, \bar{\nu}}(-\epsilon^{-1}\log r \vert \infty) \\
 &= \bar{\mu}^{-\epsilon^{-1}\log r} \frac{( \bar{\nu}/\bar{\mu} ; q)_{-\epsilon^{-1}\log r}}{(q ; q)_{-\epsilon^{-1}\log r}} \frac{(\bar{\mu} ; q)_{\infty}}{(\bar{\nu} ;q)_{\infty}}.
 \end{align*}
We shall use the $q$-Gamma function defined by 
$$ \Gamma_q(z) = \frac{(q ; q)_{\infty}}{(q^x ; q)_{\infty}} (1-q)^{1-z}.$$
Note that for $j\in\mathbb{Z}_{\geqslant 0}$, 
$$ (q^x ; q)_j  = \frac{(q^x ; q)_{\infty}}{(q^{x+j} ; q)_{\infty}} = (1-q)^{j} \frac{\Gamma_q(x+j)}{\Gamma_q(x)}, $$
so that 
\begin{align*}
\mathbb{P}\left(X_1(1)+1 = -\epsilon^{-1}\log r \right) &= r^{\mu} \dfrac{\dfrac{(q^{\nu-\mu};q)_{\infty}}{(rq^{\nu-\mu};q)_{\infty}}}{\dfrac{(q;q)_{\infty}}{(rq;q)_{\infty}}} \frac{(1-q)^{\nu-\mu}\Gamma_q(\nu)}{\Gamma_q(\mu)} \\
&= (1-q) r^{\mu}  \frac{(rq ; q)_{\infty}}{(rq^{\nu-\mu} ; q)_{\infty}} \frac{\Gamma_q(\nu)}{\Gamma_q(\nu-\mu)\Gamma_q(\mu)}.
\end{align*}

As $\epsilon$ goes to zero, using Lemma \ref{lemma:limitpochhammer} and the fact that $\Gamma_q(x)\xrightarrow[q\to 1]{} \Gamma(x)$, 
$$  \frac{\Gamma_q(\nu)}{\Gamma_q(\nu-\mu)\Gamma_q(\mu)} \rightarrow  \frac{\Gamma(\nu)}{\Gamma(\nu-\mu)\Gamma(\mu)}; $$
$$  \frac{(rq ; q)_{\infty}}{(rq^{\nu-\mu} ; q)_{\infty}}  \rightarrow \left(1-r \right)^{\nu-\mu-1}.$$
Thus as $\epsilon$ goes to zero,
\begin{equation}
 \epsilon^{-1}\mathbb{P}\left(F^{\epsilon}(1,1) = \log r \right)\rightarrow r \times r^{\mu-1} (1-r)^{\nu-\mu-1} \frac{\Gamma(\nu)}{\Gamma(\nu-\mu)\Gamma(\mu)}.
 \label{eq:limitproba}
\end{equation}
Hence $F^{\epsilon}(1,1)$, which takes values in $a_{\epsilon} + \epsilon \mathbb{Z}$ where $a_\epsilon$ is an $\epsilon$-dependent shift, converges weakly to a continuous random variable $F$ whose density is given by $f(s) = \lim \epsilon^{-1} \mathbb{P}\left(F^{\epsilon}(1,1) = s_{\epsilon} \right)$ (where $s_{\epsilon}$ is the closest point to $s$ in $a_{\epsilon} + \epsilon \mathbb{Z}$). For more details, see the proof of \cite[Lemma 2.1]{corwin2014strict} which is very similar. Consequently, $\exp(F^{\epsilon}(1,1))$ converges weakly to the continuous random variable $\exp(F)$. Since the density of $F$ is given by the right-hand-side of (\ref{eq:limitproba}) with $s=\log r$, one concludes that the density of $\exp(F(1,1))$ is 
$$ r^{\mu-1} (1-r)^{\nu-\mu-1} \frac{\Gamma(\nu)}{\Gamma(\nu-\mu)\Gamma(\mu)} $$
which is the density of a $Beta(\mu, \nu-\mu)$ random variable. 

\end{proof}
\begin{lemma}
Conditionally on $e^{F^{\epsilon}(t-1,n)} = Z$ and $e^{F^{\epsilon}(t-1,n-1)} = Z'$, the sequence of random variables $\frac{\exp(F^{\epsilon}(t,n)) - Z'}{Z-Z'}$ converges as $\epsilon \to 0$ to a Beta distributed random variable with parameters $(\mu, \nu-\mu)$.
\label{lemma:betalimit2}
\end{lemma}
\begin{proof}
Conditioning on $e^{F^{\epsilon}(t-1,n)} = Z$ and $e^{F^{\epsilon}(t-1,n-1)} = Z'$ corresponds to conditioning on the fact that the gap $x_{n-1}(t-1) - x_{n}(t-1)-1$ is $\epsilon^{-1} \log(Z/Z')$. 
The probability that $e^{F^{\epsilon}(t,n)}/Z = s$, conditioned to $e^{F^{\epsilon}(t-1,n)} = Z$ and $e^{F^{\epsilon}(t-1,n-1)} = Z'$ is 
\begin{equation*}
\mathbb{P}\Big( x_n(t)-x_n(t-1)  = -\epsilon^{-1}\log(s)  \Big\vert x_{n-1}(t-1) - x_{n}(t-1)-1= -\epsilon^{-1} \log(r) \Big),
\end{equation*}
where we have set $r=Z'/Z$. By the definition of the $q$-Hahn TASEP, this probability is exactly $\varphi_{q, \bar{\mu}\vert \bar{\nu}}\big(-\epsilon^{-1}\log(s)\big\vert -\epsilon^{-1}\log(r)\big)$.

\begin{multline*}
\varphi_{q, \bar{\mu}, \bar{\nu}}\Big(-\epsilon^{-1}\log(s)\Big\vert -\epsilon^{-1}\log(r)\Big) = \\
 s^{\mu} \frac{(q^{\mu} ; q)_{-\epsilon\log(r/s)}}{(q^{\nu} ; q)_{-\epsilon\log(r)}} \frac{(q^{\nu-\mu} ; q)_{-\epsilon\log(s)}}{(q ; q)_{-\epsilon\log(s)}} \frac{(q ; q)_{-\epsilon\log(r)}}{(q ; q)_{-\epsilon\log(r/s)}}\\
= (1-q) s^\mu \frac{(\frac r s q^\nu ; q)_{\infty}}{(\frac r s q^\mu ; q)_{\infty}}\frac{(s q ; q)_{\infty}}{(s q^{\nu-\mu} ; q)_{\infty}}\frac{(r q ; q)_{\infty}}{(\frac r s q ; q)_{\infty}}\frac{\Gamma_q(\nu)}{\Gamma_q(\nu-\mu)\Gamma_q(\mu)}.
\end{multline*}
Using again Lemma (\ref{lemma:limitpochhammer}), one finds that as $\epsilon$ goes to zero, 
$$ \epsilon^{-1}\varphi_{q, \bar{\mu}, \bar{\nu}}\Big(-\epsilon^{-1}\log(s)\Big\vert -\epsilon^{-1}\log(r)\Big) \rightarrow s^{\mu}(1-s)^{\nu-\mu - 1} \frac{(1-r/s)^{\mu-1}}{(1-r)^{\nu-1}} \frac{\Gamma(\nu)}{\Gamma(\nu-\mu)\Gamma(\mu)}$$
which can be rewritten as 
\begin{multline}
 \epsilon^{-1}\varphi_{q, \bar{\mu}, \bar{\nu}}\Big(-\epsilon^{-1}\log(s)\Big\vert -\epsilon^{-1}\log(r)\Big) \rightarrow \\
 \frac{s}{1-r}\times \left(\frac{s-r}{1-r} \right)^{\mu-1}\left( \frac{1-s}{1-r}\right)^{\nu-\mu-1} \frac{\Gamma(\nu)}{\Gamma(\nu-\mu)\Gamma(\mu)}.
\end{multline}
By the same arguments as in the proof of Lemma \ref{lemma:betalimit1}, this shows that $\exp{F^{\epsilon}(t,n)}/Z$ converges to a continuous random variable in $(r,1)$ having density 
$$\frac{1}{1-r} \left(\frac{s-r}{1-r} \right)^{\mu-1}\left( \frac{1-s}{1-r}\right)^{\nu-\mu-1} \frac{\Gamma(\nu)}{\Gamma(\nu-\mu)\Gamma(\mu)}\mathrm{d}s.$$
This implies that  
$$\frac{\exp{F^{\epsilon}(t,n)}/Z - Z'/Z }{1-Z'/Z} \Longrightarrow Beta(\mu, \nu-\mu).$$
\end{proof}
It is clear that $\exp{F^{\epsilon}(0,1)}=1$ for any $\epsilon$. By applying several times Lemma \ref{lemma:betalimit1}, one sees that $\exp{F^{\epsilon}(t,1)}$ converges to a product of independent Beta random variables with parameters $(\mu, \nu-\mu)$. Finally, by recurrence on $t+n$ and using Lemma \ref{lemma:betalimit2}, $\exp{F^{\epsilon}(t,n)}$ converges in distribution to $\exp{F(t,n)}$ where the family of random variables  $\left(\exp{F(t,n)}\right)_{t,n}$ is defined by the recurrence formula of the statement of Proposition \ref{proprecursionlimit}. This, in turn,  implies the weak convergence of processes
\begin{equation}
 \lbrace e^{F^{\epsilon}(t,n)}\rbrace_{t\geqslant 0, n\geqslant 1} \Rightarrow  \lbrace Z(t,n) \rbrace_{t\geqslant 0, n\geqslant 1}.
 \label{eq:convergenceprocess}
\end{equation}
\end{proof}

One has the following Fredholm determinant representation for the $e_q$-Laplace transform of $X_n(t)$.
\begin{theorem}[Theorem 1.10 in \cite{corwin2014q}]
Consider $q$-Hahn TASEP started from step initial data $x_n(0)=-n \ \forall n \geqslant 1$. Then for all $\zeta\in \mathbb{C}\setminus \mathbb{R}_{>0}$,
\begin{equation}
\mathbb{E}\left[\frac{1}{(\zeta q^{x_n(t)+n} ; q)_{\infty}} \right] = \det(I+K^{\mathrm{qHahn}}_{\zeta})_{\mathbb{L}^2(C_1)} 
\label{eq:fredholmqHahn}
\end{equation}
where $C_1$ is a small positively oriented circle containing $1$ but not $1/\bar{\nu}, 1/q$ nor $0$, and $K^{\mathrm{qHahn}}_{\zeta} : \mathbb{L}^2(C_1)\rightarrow \mathbb{L}^2(C_1)$ is defined by its integral kernel 
$$K^{\mathrm{qHahn}}_{\zeta} (w,w') =\frac{1}{2i\pi} \int_{1/2+i\mathbb R} \frac{\pi}{\sin(\pi s) } (-\zeta)^s \frac{g^{\mathrm{qHahn}}(w)}{g^{\mathrm{qHahn}}(q^sw)} \frac{\rm{d}s}{q^s w -w'}$$
with $$ g(w)  = \left(\frac{(\bar{\nu} w ; q)_{\infty}}{(w ; q)_{\infty}} \right)^n \left( \frac{(\bar{\mu} w ; q)_{\infty}}{(\bar{\nu} w ; q)_{\infty}}\right)^t \frac{1}{(\bar{\nu} w ; q)_{\infty}}.$$
\end{theorem}

Let us scale the parameter $\zeta$ as 
$$\zeta = (1-q) u, $$
and scale the other parameters as previously: $q=e^{-\epsilon}, \bar{\mu} = q^{\mu}, \bar{\nu} = q^{\nu }$. Then we have 
$$\mathbb{E}\left[\frac{1}{(\zeta q^{x_n(t)+n} ; q)_{\infty}} \right] = \mathbb{E}\left[ e_q\left(ue^{F^{\epsilon}(t,n)} \right) \right] $$
where 
$$ e_q(x) = \frac{1}{\big((1-q)x ; q\big)_{\infty}} $$
is the $e_q$-exponential function. Since $e_q(x)\rightarrow e^x$ uniformly for $x$ in a compact set, we have, using the convergence of processes (\ref{eq:convergenceprocess}) and the fact that $e^{F^{\epsilon}(t,n)}$ are uniformly bounded by $1$, that 
\begin{equation}
 \lim_{\epsilon \to 0}\mathbb{E}\left[\frac{1}{(\zeta q^{x_n(t)+n} ; q)_{\infty}} \right] = \mathbb{E}\left[\exp(ue^{F(t,n)}) \right]. 
 \label{eq:star}
\end{equation}
Hence, in order to prove Theorem \ref{thmLaplacepolymer}, one has to take the limit when $\epsilon$ goes to zero of the Fredholm determinant in the right-hand-side of (\ref{eq:fredholmqHahn}). This is achieved in Proposition \ref{prop:limitFredholmqhahn}. 
\begin{proposition}
\begin{equation*}
 \lim_{\epsilon \to 0}\mathbb{E}\left[\frac{1}{(\zeta q^{x_n(t)+n} ; q)_{\infty}} \right] =\det(I-K^{\mathrm{BP}}_u)_{\mathbb{L}^2(C_0)}
\end{equation*}
where $C_0$ is a small positively oriented circle containing $0$ but not $-\nu$ nor $-1$, and $K^{\mathrm{BP}}_u : \mathbb{L}^2(C_0)\rightarrow \mathbb{L}^2(C_0)$ is defined by its integral kernel 
\begin{equation}
K^{\mathrm{BP}}_u(v,v') = \frac{1}{2i\pi} \int_{1/2-i\infty}^{1/2+i\infty} \frac{\pi}{\sin(\pi s)} (-u)^s\frac{g^{\mathrm{BP}}(v)}{g^{\mathrm{BP}}(v+s)} \frac{\mathrm{d}s}{s+v - v'}
\label{eq:formulakernelBeta}
\end{equation} 
where 
\begin{equation*}
g^{\mathrm{BP}}(v) = \left(\frac{\Gamma(v)}{\Gamma(\nu+v)} \right)^n \left( \frac{\Gamma(\nu+v)}{\Gamma(\mu+v)}\right)^t \Gamma(\nu+ v).
\end{equation*}
\label{prop:limitFredholmqhahn}
\end{proposition}
\begin{proof}
Let us first show that the pointwise limit of the kernel of the Fredholm determinant (\ref{eq:fredholmqHahn}) of the $q$-Hahn TASEP agrees with  (\ref{eq:formulakernelBeta}). Make the change of variables 
$$ w=q^v,\ \  w'=q^{v'}.$$
The function $g^{\mathrm{qHahn}}$ used inside the integrand of the kernel $K^{\mathrm{qHahn}}_{\zeta}$ becomes
$$ g^{\mathrm{qHahn}}(q^v) = \left(\frac{(q^{\nu+v} ; q)_{\infty}}{(q^v; q)_{\infty}} \right)^n \left( \frac{(q^{\mu+v} ; q)_{\infty}}{(q^{\nu+v} ; q)_{\infty}}\right)^t \frac{1}{(q^{\nu+v} ; q)_{\infty}}.$$
We again use the $q$-Gamma function and the formula
$$(q^z ; q)_{\infty} = \frac{(q ; q)_{\infty}}{\Gamma_q(z) } (1-q)^{1-z}.$$
In terms of $q$-Gamma function
$$ g^{\mathrm{qHahn}}(q^v) = \left(\frac{(1-q)^{v}\Gamma_q(v)}{(1-q)^{\nu+v}\Gamma_q(\nu+v)} \right)^n \left( \frac{(1-q)^{\nu+v}\Gamma_q(\nu+v)}{(1-q)^{\mu+v}\Gamma_q(\mu+v)}\right)^t  \frac{(1-q)^{\nu+v}\Gamma_q(\nu+v)}{(1-q)(q;q)_{\infty}}.$$
Thus, 
$$\frac{g^{\mathrm{qHahn}}(q^v)}{g^{\mathrm{qHahn}}(q^{v+s})}  =  \left(\frac{\Gamma_q(\nu+v+s)\Gamma_q(v)}{\Gamma_q(v+s)\Gamma_q(\nu+v)} \right)^n \left( \frac{\Gamma_q(\mu+v+s)\Gamma_q(\nu+v)}{\Gamma_q(\nu+v+s)\Gamma_q(\mu+v)}\right)^t  \frac{\Gamma_q(\nu+v)}{(1-q)^s\Gamma_q(\nu+v+s)}.$$
and hence,
$$ (1-q)^s \frac{g^{\mathrm{qHahn}}(q^v)}{g^{\mathrm{qHahn}}(q^{v+s})} \xrightarrow[\epsilon \to 0]{} \frac{g^{\mathrm{BP}}(v)}{g^{\mathrm{BP}}(v+s)}.$$
The extra factor $(1-q)^s$ in $g(q^v)/g(q^{v+s})$ cancels with the one coming from $\zeta$. 
Moreover, 
$$ \frac{-\epsilon}{q^{s+v}-q^{v'}} \xrightarrow[\epsilon \to 0]{} \frac{1}{s+v-v'},$$
where the  factor $-\epsilon$ comes from the Jacobian of the change of variables $\mathrm{d}w = -\epsilon q^v \mathrm{d}v$. This demonstrates pointwise convergence  of the integrand in the integral defining $K^{\mathrm{qHahn}}_{\zeta}$ to that defining $K^{\mathrm{BP}}_u$. 

In order to show that $-\epsilon K^{\mathrm{qHahn}}_{\zeta}(q^{v}, q^{v'})$ converges to $K^{\mathrm{BP}}_u(v,v')$, one needs an integrable bound.  We will to show that  for a fixed $v$, the quantity 
\begin{equation}
 \frac{\pi}{\sin(\pi s)}  \frac{\Gamma_q(\mu+v+s)^t}{\Gamma_q(\nu+v+s)^{t-n}\Gamma_q(v+s)^n}\frac{1}{\Gamma_q(\nu+v+s)}
 \label{eq:tointegrate}
\end{equation}
is uniformly integrable in $s$ as $q$ varies. 
We need a few estimates to show this. For $z=x+iy$ with fixed $x$, we have from \cite[Chapter 1, 1.18 (2)]{erdelyi1953higher},
\begin{equation}
\big\vert \Gamma(x+iy) \big\vert e^{\vert y \vert \pi/2} \vert y\vert^{1/2-x}  \xrightarrow[\vert y \vert \to \infty]{} e^{-x}\sqrt{2\pi}.
\label{eq:estimategamma}
\end{equation}
We also need the estimates for the $q$-Gamma function in the two next lemmas.
\begin{lemma}
For any fixed $a,b>0$,  there exists a constant $C_2>0$ such that for any $y\in \R$ and $q\in (\frac 1 2,1)$, 
$$ \left| \frac{\Gamma_q(a+i y)}{\Gamma_q(b+iy)}\right|\leqslant C_2 \Big(\big\vert  y\big\vert^{\vert b-a\vert+1}+1\Big). $$
\label{lem:estimateqGamma1}
\end{lemma}
\begin{proof}
By symmetry, it is enough to prove the result for $y>0$. We have 
\begin{equation*}
\left| \frac{\Gamma_q(a+i y)}{\Gamma_q(b+iy)}\right| = (1-q)^{b-a} \bigg| \prod_{n \geqslant 0}\left(\frac{1-q^{b+iy+n}}{1-q^{a+iy+n}} \right) \bigg|.
\end{equation*}
If $a<b$, then 
$$ \left| \frac{\Gamma_q(a+i y)}{\Gamma_q(b+iy)}\right| \leqslant (1-q)^{b-a} \bigg| \prod_{n \geqslant 0}\left(\frac{1-q^{b+n}}{1-q^{a+n}} \right) \bigg| =\frac{\Gamma_q(a)}{\Gamma_q(b)} \leqslant 1. $$
If $a>b$, 
We write 
$$\left| \frac{\Gamma_q(a+i y)}{\Gamma_q(b+iy)}\right| =  \left| \frac{\Gamma_q(a+i y)}{\Gamma_q\big(b+\lceil a-b\rceil+iy\big)} \frac{\Gamma_q\big(b+\lceil a-b\rceil+i y\big)}{\Gamma_q(b+iy)}\right|.$$
Since $b+\lceil a-b \rceil>a$, we have from the first part of the proof that 
$$\left|\frac{\Gamma_q(a+i y)}{\Gamma_q\big(b+\lceil a-b\rceil+iy\big)}\right|\leqslant 1.$$
Moreover, since the q-Gamma function satisfies the functional equation 
$$ \Gamma_q(z+1) = [z]_q \Gamma_q(z) ,$$
where $[z]_q:=\frac{1-q^z}{1-q}$ is the $q$-deformed complex number, we have that 
$$\left| \frac{\Gamma_q(a+i y)}{\Gamma_q(b+iy)}\right|\leqslant \prod_{j=0}^{\lceil a-b\rceil-1}\Big| [b+j+iy]_q\Big|.$$
It can be checked that for $x,y\in\R$, we have the identity
$$ \big|[x+iy]_q \big|^2 = \big|[x]_q \big|^2  + q^x\big| [iy]_q\big|^2.$$
For $q\in (1/2, 1)$, $\vert \log(q)/(1-q)\vert^2 \leqslant 2$, and hence
$$ \big| [iy]_q\big|^2 = 2\frac{\Big(1-\cos\big( y\log(q)\big)\Big)}{(1-q)^2} \leqslant \frac{y^2\log(q)^2}{(1-q)^2} \leqslant 2 y^2.$$
This implies that there exist a constant $C_2>0 $ independent of $q$ such that 
$$\left|\prod_{j=0}^{\lceil a-b\rceil-1} [b+j+iy]_q  \right| \leqslant C_2 (y^{(a-b+1)}+1),$$
which concludes the proof. 
\end{proof}

\begin{lemma}
For any $y\in\R$  and $q\in(0,1)$, 
$$ \big|\Gamma(1+iy) \big| \leqslant \big|\Gamma_q(1+iy) \big|.$$
\label{lem:estimateqGamma2}
\end{lemma}
\begin{proof}
We have 
$$ \Gamma_q(1+iy)  = (1-q)^{-iy} \prod_{n=1}^{\infty} \frac{1-q^n}{1-q^{n+iy}},$$
so that 
$$ \big\vert\Gamma_q(1+iy)\big\vert  = \prod_{n=1}^{\infty} \frac{1-q^n}{\sqrt{1+q^{2n}-2q^n \cos\big(y \log(q)\big)}}.$$
We also have that 
$$ \Gamma(1+iy) =  \prod_{n=1}^{\infty} \frac{n}{n+iy} \left(\frac{n+1}{n} \right)^{iy}, $$
so that 
$$   \big\vert \Gamma(1+iy)\big\vert =  \prod_{n=1}^{\infty} \left|\frac{n}{n+iy}\right| = \prod_{n=1}^{\infty} \frac{n}{\sqrt{n^2+y^2}}.$$
Hence, it is enough to show that for all $n\geqslant 1$, 
$$ \frac{n}{\sqrt{n^2+y^2}} \leqslant  \frac{1-q^n}{\sqrt{1+q^{2n}-2q^n \cos\big(y \log(q)\big)}}.$$
Setting $Y=y/n$ and $Q=q^n$, it is equivalent to 
$$ \frac{1}{1+Y^2} \leqslant \frac{(1-Q)^2}{1+Q^2 -2Q\cos\big(Y\log(Q)\big)},$$
which is equivalent to 
$$ Y^2 (1-Q)^2 \geqslant 2Q\Big(1-\cos\big(Y\log(Q)\big)\Big), $$
which is true for any $Q\in(0,1)$ and $Y\in \R$.
\end{proof}
Finally, we can write using Lemma \ref{lem:estimateqGamma2} that  for $s\in 1/2+i\R$,
\begin{multline*}
\bigg\vert \Gamma(s)\Gamma(1-s) \frac{1}{\Gamma_q(\nu+v+s)}\bigg\vert =\bigg\vert \frac{\Gamma(s)\Gamma(1-s)}{\Gamma(1/2+s)} \frac{\Gamma(1/2+s)}{\Gamma_q(1/2+s)} \frac{\Gamma_q(1/2+s)}{\Gamma_q(\nu+v+s)}\bigg\vert \\ \leqslant \bigg\vert\frac{\Gamma(s)\Gamma(1-s)}{\Gamma(1/2+s)} \frac{\Gamma_q(1/2+s)}{\Gamma_q(\nu+v+s)}\bigg\vert. 
\end{multline*}
Hence, using Lemma \ref{lem:estimateqGamma1}, there exist constants $C,c>0$ such that (\ref{eq:tointegrate}) is uniformly bounded by 
$$ Ce^{-\pi/2 \vert \Imag[s]\vert} \big\vert \Imag[s]\big\vert^c.$$
Thus, (\ref{eq:tointegrate}) is uniformly integrable for $s$ along $1/2+i\R$, as $q$ varies near $1$. 
Consequently the integrand of $-\epsilon K^{\mathrm{qHahn}}_{\zeta}(q^v,q^{v'})$ is uniformly integrable. By dominated convergence, it implies that we have pointwise convergence of the kernel $-\epsilon K^{\mathrm{qHahn}}_{\zeta}(q^{v}, q^{v'})$ to the kernel $K^{\mathrm{BP}}_u(v,v')$. 

However, it is a priori not sufficient. In order to prove the convergence of the Fredholm determinant, we use again dominated convergence. First notice that since the Fredholm determinant contour is finite, one can prove as in \cite[Lemma 3.2]{corwin2014strict} 
that $-\epsilon K^{\mathrm{qHahn}}_{\zeta}(q^v, q^{v'})$ is uniformly bounded for $v,v'$ in the contour $C_0$ and $q$ near $1$. 
Moreover, each term in the Fredholm determinant expansion
$$ \det(I+K^{\mathrm{qHahn}}_{\zeta}) = 1 + \sum_{n=1}^{\infty} \frac{1}{n!} \int\dots\int \det(K^{\mathrm{qHahn}}_{\zeta}(w_i, w_j))_{i,j=1}^{n} \mathrm{d}w_1\dots \mathrm{d}w_n, $$
can be bounded using Hadamard's bound, so that the sum absolutely converges. Combining this with the above established pointwise convergence of the kernels allows us to  conclude the proof of Proposition \ref{prop:limitFredholmqhahn}. 
\end{proof}

\begin{proof}[Proof of Theorems \ref{thmLaplacepolymer} and \ref{thmLaplaceRWRE}]
 Proposition \ref{prop:limitFredholmqhahn} combined with (\ref{eq:star}) yields the Fredholm determinant formula for the Laplace transform of $Z(t, n)$ given in Theorem \ref{thmLaplacepolymer}. In order to deduce Theorem \ref{thmLaplaceRWRE}, we use the equivalence between the Beta polymer and the Beta-RWRE from  Proposition \ref{prop:equivRWREpolymer}, proved in Section \ref{subsec:equivRWREpolymer}.
\end{proof}

\subsection{Equivalence Beta-RWRE and Beta polymer}
\label{subsec:equivRWREpolymer}
\begin{figure}
\begin{tikzpicture}[scale=0.6]
\clip (-0.8, -5.5) rectangle (16.5, 5.5);
\draw[gray, dotted] (-1, -10) grid (17, 10);
\fill (0,0) circle(0.1);
\draw (-0.3, 0.3) node{$0$};
\begin{scope}[rotate=45, scale= sqrt(2)]
\draw[gray] (0, -11) grid (11, 0);
\end{scope}
\draw[thick] (15, -0.1)  -- (15,0.1) node[anchor=south]{$t$};
\draw[ultra thick] (0,0) -- (1, -1)-- (3, 1) -- (6, -2) -- (8, 0) -- (9, -1) -- (13, 3)-- (14, 2) -- (15, 3);
\fill (15, 3) circle(0.1);
\draw (15, 3) node[anchor=west]{$X_t $};
\draw[->, thick] (0,-5.5) -- (0,5.5);
\draw[->, thick] (-1,0) -- (16.5,0);

\draw (-0.1, -1) node[anchor=east]{$x$} -- (0.1, -1);
\draw[dashed, gray , thick] (0.1, -1) -- (15, -1);
\fill[gray, opacity=0.5] ( 14.9, 7) -- ( 14.9, -1) -- (15.1, -1) -- (15.1, 7) ;
\fill[gray, opacity=0.5] (15, -1) circle(0.1);
\end{tikzpicture}
\caption{A possible path for the Beta-RWRE is shown. It corresponds to the half-line to point polymer path in Figure \ref{fig:betapolymer}. $P(t,x)$ is the (quenched) probability that the random walk ends at time $t$ in the gray region.}
\label{fig:betaRWRE2}
\end{figure}
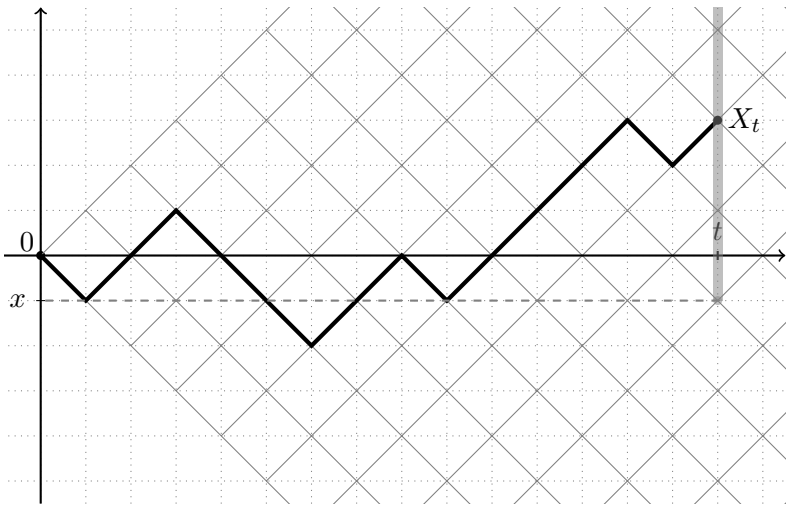
We show that the Beta RWRE and the Beta polymer are equivalent models in the sense that if the parameters $\alpha, \beta$ of the random walk and the parameters $\mu, \nu$ of the polymer are such that $\mu=\alpha$ and $\nu=\alpha+\beta$, we have the equality in law
$$ Z(t, n)  = P(t, t-2n+2).$$

The equality in law is true for fixed $t$ and $n$. However, as families of random variables, $\left(Z(t,n) \right)$ and  $\left(P(t,t-2n+2) \right)$  for $t+1\geqslant n\geqslant 1$ have different laws.

\begin{proof}[Proof of Proposition \ref{prop:equivRWREpolymer}]
Let us first notice that since $\mu=\alpha$ and $\nu=\alpha+\beta$, the i.i.d. collection of Beta random variables defining the environment for the Beta polymer, and the i.i.d. collection of r.v. defining the environment of the Beta RWRE, have the same law. 

Also, as it was already pointed-out in Section \ref{subsubsec:halflinepolymer}, the point-to-point Beta polymer is equivalent to a half line to point  Beta polymer.

Let $t$ and $n$ having the same parity. The random variable $P(t, t-2n+2)$ is the probability for the Beta RWRE to arrive above (or exactly at) $t-2n+2$. This is also the probability for the Beta RWRE to make at most $n-1$ downward steps up to time $t$. Let us imagine that we deform the underlying lattice of the Beta polymer so that Beta polymer paths are actually up-right path, and we also consider the path from $(t,n)$ to its initial point. Then the polymer path is the trajectory of a random walk, and one can interpret the weight of this polymer path as the quenched probability of the corresponding random walk trajectory (compare the polymer path depicted in Figure \ref{fig:betapolymer} with the RWRE path depicted in Figure \ref{fig:betaRWRE2}, using the correspondence shown in Figure \ref{figdeformationlattice}). Moreover the event that the random walk performs at most $n-1$ downward steps is equivalent to the fact that the polymer path starts with positive $n$-coordinate. These events correspond to the fact that the path intersects the thick gray half-lines in Figures \ref{fig:betapolymer} and \ref{fig:betaRWRE2}. 

\begin{figure}
\begin{tikzpicture}[scale=2]
\draw[gray] (-0.5, 0) -- (1.5, 0);
\draw[gray] (-0.5, 1) -- (1.5, 1);
\draw[dotted, gray] (0, -0.5) -- (0,1.5);
\draw[dotted, gray] (1, -0.5) -- (1,1.5);
\draw[gray] (-0.5, -0.5) -- (1.5, 1.5);
\draw[gray] (-0.5, 0.5) -- (0.5, 1.5);
\draw[gray] (0.5, -0.5) -- (1.5, 0.5);
\draw[thick, ->] (0,1) -- (0.95,1);
\draw[thick, ->] (0,0) -- (0.96,0.96);
\fill (1,1) circle(0.05);
\draw (1,1) node[anchor=south west]{$(t,n)$};
\draw (0.5, 1.15) node{$B_{t,n} $};
\draw (0.6, 0.2) node{\footnotesize{$1-B_{t,n} $}};

\draw  (2.5,0.5) node{\Large{$\Longleftrightarrow $}};

\begin{scope}[xshift=4cm, yshift=0.5cm]
\draw[dotted, gray] (-0.3, 0) -- (1.3, 0);
\draw[dotted, gray] (-0.3, 1) -- (1.3, 1);
\draw[dotted, gray] (-0.3, -1) -- (1.3, -1);
\draw[dotted, gray] (0, -1.3) -- (0,1.3);
\draw[dotted, gray] (1, -1.3) -- (1,1.3);
\draw[gray] (-0.3,-0.3) -- (1.3, 1.3);
\draw[gray] (-0.3,0.3) -- (1.3, -1.3);

\draw[thick, ->] (0.03,0.03) -- (1,1);
\draw[thick, ->] (0.03,-0.03) -- (1,-1);
\fill (0,0) circle(0.05);

\draw (0,0) node[anchor=east]{$(x,t)$};
\draw (0.3, 0.6) node{$ B_{x, t}$};
\draw (0.3, -0.7) node{$ 1-B_{x, t}$};
\end{scope}
\end{tikzpicture}
\caption{Illustration of the deformation of the underlying lattice for the Beta polymer. The left picture corresponds to the Beta polymer whereas the right picture corresponds to the RWRE. Black arrows represents possible steps for the polymer path (resp. the RWRE) with their associated weight (resp.  probability). }
\label{figdeformationlattice}
\end{figure}
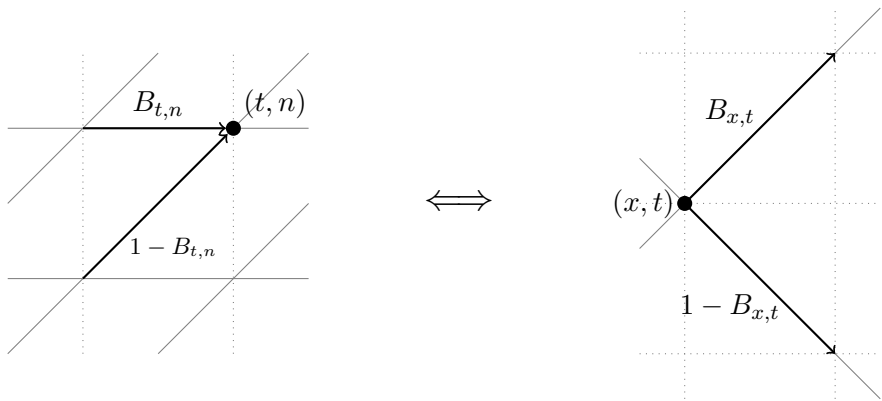

Finally, for any fixed $t, n\in \Z_{\geqslant 0}$ such that $t+1\geqslant n$, if we set $x=t-2n+2$, then $P(t,x)$ and $Z(t,n)$ have the same probability law. Moreover, conditioning on the environment of the Beta polymer corresponds to conditioning on the probability of each step for the Beta RWRE. 
\end{proof}

\section{Rigorous replica method for the Beta polymer}
\label{sec:directproof}
\subsection{Moment formulas}
Let $\Weyl{k}$ be the Weyl chamber 
$$ \Weyl{k} = \big\lbrace \vec{n}\in \Z^k : n_1\geqslant n_2\geqslant \dots \geqslant n_k \big\rbrace.$$
For $\vec{n}\in \Weyl{k}$, let us define 
\begin{equation}
u(t, \vec{n})  = \EE\big[Z(t, n_1)\dots Z(t, n_k)\big], 
\label{eq:defu}
\end{equation} 
with the convention that $Z(t,n)=)$ for $n<1$.
The recurrence relation (\ref{eq:recurrencerelation}) implies a recurrence relation for $u(t, \vec{n})$. We are going to solve this recurrence to find a closed formula for $u(t, \vec{n})$, using a variant of the Bethe ansatz. It is the  analogue of Section 5 in \cite{corwin2014strict}. Besides the strict weak polymer \cite{corwin2014strict}, such ``replica method'' calculations have been performed to study moments of the partition function for the continuum polymer \cite{dotsenko2010replica,calabrese2010free,borodin2014macdonald}, the semi-discrete polymer \cite{borodin2012duality, borodin2014macdonald}, and the log-gamma polymer \cite{borodin2014macdonald, thiery2014log}.  However, in those models, the moment problems are ill-posed and one cannot rigorously recover the distribution from them. In the present case, since the $Z(t,n) \in [0,1]$, the moments do determine the distribution as explained in Section  \ref{subsec:secondproof}.

Using the recurrence relation (\ref{eq:recurrencerelation}), 
\begin{equation}
u(t+1,\vec{n}) = \EE\left[ \prod_{i=1}^k \Big( (1-B_{t+1,n_i}) Z(t,n_i) + B_{t+1, n_i} Z(t, n_i-1)\Big)\right].
\end{equation}

Let us first simplify this expression when $k=c$ and  $\vec{n} = (n, \dots, n)$ is a vector of length $c$ with all components equal. In this case, setting $B=B_{t+1, n}$ to simplify the notations, we have
\begin{align*}
u(t+1,\vec{n}) =& \sum_{j=0}^c \binom{c}{j} \EE\left[ (1-B)^j B^{c-j} Z(t, n-1)^j Z(t, n)^{c-j} \right]\\
=&\sum_{j=0}^c \binom{c}{j} \EE\left[ (1-B)^j B^{c-j}\right]u(t, n, \dots, n,\underbrace{n-1, \dots , n-1}_{j}).
\end{align*}
The recurrence relation can be further simplified using the next Lemma. 
\begin{lemma}
Let $B$ a random variable following the $Beta(\mu, \nu-\mu)$ distribution.
Then for integers $0\leqslant j\leqslant c$, 
$$ \EE\big[(1-B)^jB^{c-j}\big]  = \frac{(\nu-\mu)_j(\mu)_{c-j}}{(\nu)_c}.$$
where $(a)_k$ is the Pochhammer symbol $ (a)_k= a (a+1) \dots (a+k-1)$ and $(a)_0 =1 $. 
\label{lemma:betamoments}
\end{lemma}
\begin{proof}
By the definition of the Beta law, we have 
\begin{align*}
\EE\big[(1-B)^jB^{c-j}\big] &= \frac{\Gamma(\nu)}{\Gamma(\mu)\Gamma(\nu-\mu)}\int_{0}^1 (1-x)^j x^{c-j} x^{\mu-1} (1-x)^{\nu-\mu-1},\\
&= \frac{\Gamma(\nu)}{\Gamma(\mu)\Gamma(\nu-\mu)}\frac{\Gamma(\mu+c-j)\Gamma(\nu-\mu+ j)}{\Gamma(\nu+c)},\\
&=\frac{(\nu-\mu)_j(\mu)_{c-j}}{(\nu)_c}.
\end{align*}
\end{proof}
In order to write the general case, we need a little more notation. For $\vec{n}\in \Weyl{k}$, we denote by $c_1, c_2, \dots c_\ell$ the sizes of clusters of equal components in $\vec{n}$. More precisely, $c_1, c_2, \dots c_{\ell}$ are positive integers such that $\sum c_i = k$ and 
$$ n_1 = \dots = n_{c_1}>n_{c_1+1} = \dots n_{c_1+c_2} >\dots >n_{c_1+\dots + c_{k-1}+1} = \dots = n_k.$$
Define also the operator $\tau^{(i)}$ acting on a function $f:\Weyl{k}\rightarrow \R$ by  
$$ \tau^{(i)} f(\vec{n}) = f(n_1, \ldots , n_i-1,\ldots, n_k).$$
Using the Lemma \ref{lemma:betamoments}, we have that 
\begin{equation}
u(t+1,\vec{n}) = \sum_{j_1=0}^{c_1}\dots \sum_{j_{\ell}=0}^{c_{\ell}} \left(\prod_{i=1}^{\ell}\binom{c_i}{j_i}  \frac{(\nu-\mu)_{j_i}(\mu)_{c_i-j_i}}{(\nu)_{c_i}} \prod_{r=0}^{j_i-1} \tau^{(c_1 + \dots + c_{i} - r)} \right) u(t,\vec{n}).
\label{eq:trueevolution}
\end{equation}
In words, for each $\ell$-tuple $j_1, \dots, j_{\ell}$ such that $0\leqslant j_i \leqslant c_i$, we decrease the $j_i$ last coordinates of the cluster $i$ in $\vec{n}$, for each cluster, and multiply by  
$$\prod_{i=1}^{\ell}\binom{c_i}{j_i}  \frac{(\nu-\mu)_{j_i}(\mu)_{c_i-j_i}}{(\nu)_{c_i}}.$$
\begin{lemma}
Let $X,Y$ generate an associative algebra such that 
$$ YX = \frac{1}{1+\nu}XX + \frac{\nu-1}{1+\nu} XY  + \frac{1}{1+\nu} YY.$$
Then we have the following non-commutative binomial identity:
$$ \big(pX+(1-p)Y \big)^n  = \sum_{j=0}^n \binom{n}{j} \frac{(\nu-\mu)_j(\mu)_{n-j}}{(\nu)_n}X^j Y^{n-j},$$
where $p=\frac{\nu-\mu}{\nu}$. 
\label{lemma:binomial}
\end{lemma}
\begin{proof}
It is shown in \cite[Theorem 1]{povolotsky2013integrability} that if X and Y satisfy the quadratic homogeneous relation 
$$ YX= \alpha XX + \beta XY + \gamma YY, $$
with 
$$ \alpha= \frac{\bar{\nu}(1-q)}{1-q\bar{\nu}}, \ \ \beta = \frac{q-\bar{\nu}}{1-q\bar{\nu}}, \ \ \gamma = \frac{1-q}{1-q\bar{\nu}}, $$
and 
$$ \bar{\mu} = \bar{p}+\bar{\nu}(1-\bar{p}),$$
then 
$$ \big(\bar{p}X+(1-\bar{p})Y\big)^n = \sum_{k=0}^{n}\varphi_{q, \bar{\mu}, \bar{\nu}}(j\vert n) X^kY^{n-k},$$
where $\varphi_{q, \bar{\mu}, \bar{\nu}}(j\vert n)$ are the $q$-Hahn weights defined in (\ref{eq:defqhahnweights}). 
Our lemma is the $q\to 1$ degeneration of this result. 
\end{proof}
Let $ \mathcal{L}^{\mathrm{cluster}}_{c}$ denote the operator 
$$ \mathcal{L}^{\mathrm{cluster}}_c =  \sum_{j=0}^{c} \binom{c}{j}  \frac{(\nu-\mu)_{j}(\mu)_{c-j}}{(\nu)_{c}} \prod_{r=0}^{j-1} \tau^{(c - r)}$$
which appears in the R.H.S. of (\ref{eq:trueevolution}), and $\mathcal{L}^{free}_{c}$ the operator 
$$  \mathcal{L}^{free}_{c} = \prod_{i=1}^{c} \nabla_i,$$
where $\nabla_i = p\tau^{(i)} + (1-p)$. It is worth noticing that for $c=1$, $\mathcal{L}^{\mathrm{cluster}}_{c} = \mathcal{L}^{free}_{c}$. 

For a function $f: \Z^c\to\C$, we formally  identify monomials $X_1 X_2 \dots X_c$ where $X_i\in\lbrace X,Y\rbrace$ with terms $f(\vec{n})$ where for all $1\leqslant i \leqslant c$, $n_{c-i}=n-1$ if $X_i=X$ and $n_{c-i}=n$ if $X_i=Y$. Using this identification, the binomial formula from Lemma \ref{lemma:binomial} says that the operators $\mathcal{L}^{free}_{c}$ and $\mathcal{L}^{\mathrm{cluster}}_{c}$ act identically on functions $f$ satisfying the  condition
\begin{equation}
\forall 1\leqslant i <c\left(\frac{1}{1+\nu} \tau^{(i)}\tau^{(i+1)} + \frac{\nu-1}{1+\nu}\tau^{(i+1)} +\frac{1}{1+\nu} -\tau^{(i)}\right) f(n, \ldots ,n) = 0.
\end{equation}

One notices that the operator involved in (\ref{eq:trueevolution}) acts independently by $\mathcal{L}^{\mathrm{cluster}}_{c}$ on each cluster of equal components. It follows that 
if a function $u : \Z_{\geqslant 0}\times \Z^k \to \C$ satisfies the \emph{boundary condition}
\begin{equation}
\left(\frac{1}{1+\nu} \tau^{(i)}\tau^{(i+1)} + \frac{\nu-1}{1+\nu}\tau^{(i+1)} +\frac{1}{1+\nu} -\tau^{(i)}\right) u(t, \vec{n}) = 0, 
\end{equation}
for all $\vec{n}$ such that $n_i=n_{i+1}$ for some $1\leqslant i < k$, and satisfies the \emph{free evolution equation}
\begin{equation}
 u(t+1, \vec{n}) = \left(\prod_{i=1}^{k} \nabla_i \right)u(t, \vec{n}),
\end{equation}
for all $\vec{n}\in \Z^k$, then the restriction of $u(t,\vec{n})$ to $\Weyl{k}$ satisfies the \emph{true evolution equation}  (\ref{eq:trueevolution}).

\begin{remark}
The coefficients $\binom c j \frac{(\nu-\mu)_j(\mu)_{c-j}}{(\nu)_c}$ that appear in the true evolution equation  (\ref{eq:trueevolution}) are probabilities of the Beta-binomial distribution with parameters $c, \mu, \nu-\mu$. Hence, the true evolution equation could be interpreted as  the ``evolution equation'' for a series of urns where each urn evolves according to the P\'olya urn scheme. Such dynamics could be interpreted as the $q\to 1$ degeneration of the $q$-Hahn Boson, which is dual to the $q$-Hahn TASEP \cite{corwin2014q}. 
\end{remark}

\begin{proposition} For $n_1 \geqslant n_2\geqslant \dots \geqslant n_k\geqslant 1$, one has the following moment formula,
\begin{multline}
\mathbb{E}\Big[Z(t,n_1) \cdots Z(t,n_k) \Big] = 
\frac{1}{(2i\pi)^k} \int\dots\int \prod_{1\leqslant A<B\leqslant k} \frac{z_A-z_B}{z_A-z_B-1}  \prod_{j=1}^k \left( \frac{\nu+z_j}{z_j}\right)^{n_j} \left(\frac{\mu + z_j}{\nu+z_j}\right)^t\frac{\mathrm{d}z_j}{\nu+z_j}.
\label{eq:momentformula}
\end{multline}
where the contour for $z_k$ is a small circle around the origin, and the contour for $z_j$ contains the contour for $z_{j+1}$ shifted by $+1$ for all $j=1, \dots , k-1$, as well as the origin, but all contours exclude $-\nu$. 
\label{prop:momentformula}
\end{proposition}
\begin{proof}
We show that the right-hand-side of (\ref{eq:momentformula}) satisfies the free evolution equation, the boundary condition and the initial condition for $u(0,\vec{n})$ for $\vec{n}\in \Weyl{k}$ (the initial condition outside $\Weyl{k}$ is inconsequential). The above discussion shows that the restriction to $\vec{n}\in \Weyl{k}$ then solves the true evolution equation (\ref{eq:trueevolution}).
By the definition of the function $u$ in (\ref{eq:defu}) and the initial condition for the half-line to point polymer, $u(0,\vec{n}) = \prod_{i=1}^k \mathds{1}_{n_i\geqslant 1}=\mathds{1}_{n_k\geqslant 1}$ (the second equality holds because the $n_i$'s are ordered).  Let us consider the right-hand-side of (\ref{eq:momentformula}) when $t=0$.    If $n_k\leqslant 0$, there is no pole in zero, so one can shrink the $z_k$ contour to zero, and consequently $u(0,\vec{n})=0$. When $n_k>0$ (and consequently all $n_i$'s are positive), there is no pole at $-\nu$ for $t=0$, so that one can successively send to infinity the contours for the variables $z_k$, $z_{k-1}, \dots$. Since the residue at infinity is one for each variable, then $u(0,\vec{n})=1$. Hence, the initial condition is satisfied. 

In order to show that the boundary condition is satisfied, we assume that $n_i=n_{i+1}$ for some $i$.  Let us  apply the operator 
$$\left(\frac{1}{1+\nu} \tau^{(i)}\tau^{(i+1)} + \frac{\nu-1}{1+\nu}\tau^{(i+1)} +\frac{1}{1+\nu} -\tau^{(i)}\right)$$
inside the integrand. This brings into the integrand a factor 
\begin{equation*}
 \frac{1}{1+\nu} \frac{z_i}{\nu +z_i}\frac{z_{i+1}}{\nu +z_{i+1}} + \frac{\nu-1}{\nu+1} \frac{z_{i+1}}{\nu +z_{i+1}} + \frac{1}{1+\nu} - \frac{z_i}{\nu +  z_i} 
  = \frac{-\nu^2(z_i-z_{i+1}-1)}{(1+\nu)(\nu+ z_i)(\nu +z_{i+1})}.
 \end{equation*}
Since it cancels the pole for $z_i=z_{i+1}+1$, one can use the same contour for both variables, and since the integrand is now antisymmetric in the variables $(z_i, z_{i+1})$ the integral is zero as desired. 

In order to show that the free evolution equation is satisfied, it is enough to show that applying the operator $p\tau^{(i)} + (1-p)$ for $i$ from $1$ to $k$ inside the integrand brings an extra factor $\prod_{j=1}^{k}\frac{\mu+z_j}{\nu+z_j}$. This is clearly true since 
$$ \left( p\tau^{(i)} + (1-p)\right) \left( \frac{\nu+z_i}{z_i}\right)^{n_i} = \left( \frac{\nu+z_i}{z_i}\right)^{n_i} \frac{\mu+z_i}{\nu+z_i}. $$
\end{proof}

\begin{remark}
It is possible to prove a generalization of Proposition \ref{prop:momentformula} where the parameter $\mu$ depends on $t$. In this generalization, the weight of an edge starting from a point $(s,n)$ for any $n$ would have a weight $B$ or $1-B$ (depending on the direction of the edge), where $B$ is a random variable distributed according to the Beta distribution with parameters $(\mu_s, \nu-\mu_s)$. In the formula (\ref{eq:momentformula}), the factor $\left(\frac{\mu + z_j}{\nu+z_j}\right)^t$ would be replaced by 
$$ \prod_{s=0}^{t-1} \frac{\mu_s + z_j}{\nu+z_j}.$$
Such moment formulas with time inhomogeneous parameters have been proved for the discrete-time $q$-TASEP \cite{borodin2013discrete} and for the $q$-Hahn TASEP in \cite[Section 2.4]{corwin2014q} (See also the discussion in \cite[Section 5.7]{corwin2015stochastic} which deals with a generalization of the $q$-Hahn TASEP). In all these cases, this allows to prove Fredholm determinant formulas with time-dependent parameters, using the same method as in the homogeneous case.  It is not clear however if one can find moment formulas with a parameter inhomogeneity depending on $n$ (e.g. the parameter $\nu$ would depend on $n$). 
\end{remark}
Proposition \ref{prop:momentformula} provides an integral formula for the moments of $Z(t, n)$. In order to form the generating series, it is convenient to transform the formula so that all integrations are on  the same contour. 
\begin{proposition}
For all $n,t \geqslant 0$, we have 
\begin{multline}
\EE\left[ Z(t,n)^k\right] = k!\sum_{\lambda \vdash k} \frac{1}{m_1!m_2!\dots}\frac{1}{(2i\pi)^{\ell(\lambda)}} \int \dots \int  \det\left( \frac{1}{v_i +\lambda_i-v_j}\right)_{i,j=1}^{\ell(\lambda)}\\ \times \prod_{j=1}^{\ell(\lambda)} f(v_j)f(v_j+1)\dots f(v_j+\lambda_j-1) \mathrm{d}v_1 \dots \mathrm{d}v_{\ell (\lambda)},
\end{multline}
where 
$$ f(v) = \frac{g^{\mathrm{BP}}(v)}{g^{\mathrm{BP}}(v+1)} = \left(\frac{\nu+v}{v} \right)^n \left(\frac{\mu+v}{\nu+v}\right)^t \frac{1}{v+\nu}.$$
where $g^{\mathrm{BP}}$ is defined in (\ref{eq:defgbp}) and the integration contour is a small circle around $0$ excluding $-\nu$ and for a partition $\lambda \vdash k$ (i.e. $\sum_i \lambda_i =k$) we write $\lambda=1^{m_1} 2^{m_2}\dots$, meaning that $m_j$ is the number of indices $i$ such that $\lambda_i=j$ components;  and $\ell(\lambda)$ is the number of non-zero components $\ell(\lambda)=\sum_i m_i$.
\label{prop:shrinkingcontours}
\end{proposition}
\begin{proof}
This type of deduction, called the contour shift argument, has already occurred in the context of the $q$-Whittaker process in \cite[Section 3.2.1]{borodin2014macdonald}. 
See \cite{borodin2015spectral}, in particular Proposition 7.4, and references therein for more background on the contour shift argument. The present formulation corresponds to a degeneration when $q\to 1$ of Proposition 3.2.1 in \cite{borodin2014macdonald}. 

 One starts with the moment formula given by Proposition  \ref{prop:momentformula}: 
$$\EE\left[ Z(t,n)^k\right] = \frac{1}{(2i\pi)^k} \int\dots\int \prod_{A<B} \frac{z_A-z_B}{z_A-z_B-1}  \prod_{j=1}^k f(z_j)\mathrm{d}z_j.$$
We need to shrink all contours to a small circle around $0$. During the deformation of contours, one encounters all poles of the product $\prod_{A<B} \frac{z_A-z_B}{z_A-z_B-1}$. Thus, a direct proof would amount to carefully book-keeping the residues. Although one could adapt to the present setting the proof of \cite[Proposition 7.4]{borodin2015spectral}, we refer to Proposition 6.2.7 in \cite{borodin2014macdonald} which provides a very similar  formula. The only modification is that the function $f$ that we consider has a pole at $-\nu$, but this does not play any role in the deformation of contours. 

It is also worth remarking that applying Proposition 3.2.1 in \cite{borodin2014macdonald} to $q$-Hahn moment formula \cite[Theorem 1.8]{corwin2014q} and taking a suitable limit yields the statement of Proposition \ref{prop:shrinkingcontours}.
\end{proof}
\subsection{Second proof of Theorem \ref{thmLaplacepolymer}}
\label{subsec:secondproof}

Thanks to Proposition \ref{prop:shrinkingcontours}, the moments of $Z(t,n)$ have a suitable form for taking the generating series. Let us denote $\mu_k = \EE\left[ Z(t,n)^k\right]$. A degeneration when $q$ goes to $1$ of Proposition 3.2.8 in \cite{borodin2014macdonald} shows  that 
$$ \sum_{k\geqslant 0} \mu_k \frac{u^k}{k!} = \det(I+K)_{\mathbb{L}^2(\Z_{>0}\times C_0)}, $$
where 
$\det(I+K)$ is the formal Fredholm determinant expansion of the operator $K$ defined  by the integral kernel 
$$ K(n_1, v_1 ; n_2, v_2) =   \frac{u^{n_1} f(v_1) f(v_1+1) \dots f(v_1+n_1-1)}{v_1+n_1-v_2}, $$
and $C_0$ is a positively oriented circular contour around $0$ excluding $-\nu$.  
Since $f(v+n_1)$ is uniformly bounded for $v\in C_0$ and $n_1\geqslant 1$, and $v_1+n_1-v_2$ is uniformly bounded away from $0$ for $v_1, v_2 \in C_0, n_1\geqslant 1$, the identity holds also numerically.  
Since $\vert Z(t, n)\vert \leqslant 1$, then one can exchange summation and expectation so that for any $u\in \C$
$$\sum_{k\geqslant 0} \mu_k \frac{u^k}{k!}  = \EE\left[e^{uZ(t,n)} \right].$$
It is useful to notice that  
$$f(v_1) f(v_1+1) \dots f(v_1+n_1-1) = \frac{g^{\mathrm{BP}}(v_1)}{g^{\mathrm{BP}}(v_1+n_1)}.$$
Next, we want to rewrite  $\det(I+K)$ as the Fredholm determinant of an operator acting on a single contour. For that purpose we use the following Mellin-Barnes integral formula: 
\begin{lemma}
For $u\in \C\setminus \R_{>0}$ with $\vert u\vert <1$,
\begin{equation}
\sum_{n=1}^{\infty} u^n \frac{g^{\mathrm{BP}}(v)}{g^{\mathrm{BP}}(v+n)}\frac{1}{v+n-v'} = \frac{1}{2i\pi} \int_{1/2-i\infty}^{1/2+i\infty} \Gamma(-s)\Gamma(1+s) (-u)^s \frac{g^{\mathrm{BP}}(v)}{g^{\mathrm{BP}}(v+s)}\frac{\mathrm{d}s}{v+s-v'}, 
\label{eq:mellinbarnesidentity}
\end{equation}
where $z^s$ is defined with respect to a branch cut along $z\in\R_{\leqslant 0}$.
\label{lem:mellinbarnes}
\end{lemma}
\begin{proof}
The statement of the Lemma is very similar to  \cite[Lemma 3.2.13]{borodin2014macdonald}. 

Since $\Res{s=k} \left( \Gamma(-s)\Gamma(1+s)\right) = (-1)^{k+1}$, we have that 
\begin{equation}
\sum_{n=1}^{\infty} u^n \frac{g^{\mathrm{BP}}(v)}{g^{\mathrm{BP}}(v+n)}\frac{1}{v+n-v'} = \frac{1}{2i\pi} \int_{\mathcal{H}} \Gamma(-s)\Gamma(1+s) (-u)^s \frac{g^{\mathrm{BP}}(v)}{g^{\mathrm{BP}}(v+s)}\frac{\mathrm{d}s}{v+s-v'},
\label{eq:mellinbarnesidentity2}
\end{equation}
where  $\mathcal{H}$ is a negatively oriented integration contour enclosing all positive integers. For the identity to be valid, the L.H.S. of (\ref{eq:mellinbarnesidentity2}) must converge, and the contour must be approximated by a sequence of contours $\mathcal{H}_k$ enclosing the integers $1, \dots , k$ such that the integral along the symmetric difference $\mathcal{H}\setminus \mathcal{H}_k$ goes to zero. 

The following estimates show that one can chose the contour $\mathcal{H}_k$ as a rectangular contour connecting the points $1/2+i$, $k+1/2+i$, $k+1/2-i$ and $1/2-i$; and the contour $\mathcal{H}$ as the infinite contour from $\infty -i$ to $1/2-i$ to $1/2+i$ to $\infty+i$.

We first need an estimate for the Gamma function \cite[Chapter 1, 1.18 (2)]{erdelyi1953higher}: for any $\delta>0$
\begin{equation}
\Gamma(z) = \sqrt{2\pi} e^{-z} z^{z-1/2} (1+\mathcal{O}\left(1/z\right))\ \ \text{ as } \vert z\vert \to\infty, \ \  \vert \arg(z)\vert <\pi-\delta.
\label{eq:estimategammageneral}
\end{equation}
Then recall that 
$$ g^{\mathrm{BP}}(v+s) = \left(\frac{\Gamma(v+s)}{\Gamma(\nu+v+s)} \right)^n \left( \frac{\Gamma(\nu+v+s)}{\Gamma(\mu+v+s)}\right)^t \Gamma(\nu+ v+s). $$
Using (\ref{eq:estimategammageneral}), 
$$ g^{\mathrm{BP}}(v+s)  = \sqrt{2\pi} e^{-\nu-v-s}(\nu+v+s)^{\nu+v+s-1/2} \frac{(\nu+v+s)^{(\nu-\mu)t}}{(\nu+v+s)^{\nu n}} \left(1+\mathcal{O}\left(\frac 1 s\right) \right).$$
It implies that for $s$ going to $\infty e^{i\phi}$ with $\phi\in(-\pi/2, \pi/2)$,  $1/g^{\mathrm{BP}}(v+s)$  has exponential decay in $\vert s\vert$. Moreover, for $s$ going to $\infty e^{i\phi}$ with $\phi\in(-\pi/2, \pi/2)$ and  $\phi\neq 0$, 
$$ (-u) ^s \frac{\pi}{\sin(\pi s)} \frac{1}{v+s-v'}$$
is bounded. Further, using the exponential decay of $\frac{\pi}{\sin(\pi s)}$ along vertical lines,  one can freely deform the integration contour $\mathcal{H}$ in (\ref{eq:mellinbarnesidentity2}) to become the straight line from $1/2 -i\infty$ to $1/2+i\infty$. 
\end{proof}
This shows that for any $u\in \C\setminus \R_{>0}$ with $\vert u\vert <1$, one has that 
\begin{equation}
\EE\left[e^{uZ(t,n)} \right] = \det(I+K^{\mathrm{BP}}_u)_{\mathbb{L}^2(C_0)},
\label{eq:identityfredholm}
\end{equation}
where the kernel $K^{\mathrm{BP}}_u$ is defined in the statement of Theorem \ref{thmLaplacepolymer}. One extends the result to any $u\in \C\setminus \R_{>0}$ by analytic continuation. The right-hand-side in (\ref{eq:identityfredholm}) is analytic since we have already shown in the proof of Proposition  \ref{prop:limitFredholmqhahn} that the  Fredholm determinant expansion is absolutely summable and integrable. The left-hand-side is analytic since $\vert Z(t,n)\vert <1$.  

\section{Zero-temperature limit}
\label{sec:limitBetatoFPP}

\subsection{Proof of Proposition \ref{prop:RWREtoFPP}}
In this section, we prove that the Bernoulli-Exponential first passage percolation model is the zero-temperature limit of the Beta-RWRE. The zero temperature limit corresponds to sending the parameters $\alpha, \beta$ of the Beta RWRE to zero. 
\begin{proof} We first show how the transition probabilities for the Beta RWRE degenerate in the zero temperature limit. 
\begin{lemma}
Fix $a,b >0$. For $\epsilon >0$, let $B_{\epsilon}$ be a Beta distributed random variable with parameters $(\epsilon a, \epsilon b)$.  We have the convergence in distribution 
$$ \Big(-\epsilon \log(B_{\epsilon}), -\epsilon \log(1-B_{\epsilon}) \Big) \Longrightarrow \big(\xi E_a , (1-\xi)E_b\big) $$
as $\epsilon$ goes to zero, where $\xi$ is a Bernoulli random variable with parameter $b/(a+b)$ and $(E_a, E_b)$ are exponential random variables with parameters $a$ and $b$, independent of $\xi$. 
\label{lem:betaconvergence}
\end{lemma}
\begin{proof}
Let $f, g : \R\to\R$ be continuous bounded functions. 

\begin{multline}
 \EE\Big[ f\big(-\epsilon \log(B_{\epsilon})\big)g\big(-\epsilon \log(1-B_{\epsilon})\big)\Big]  = \\ 
 \int_0^{1} f\big(-\epsilon \log(x)\big) g\big(-\epsilon \log(1-x)\big) x^{\epsilon a -1} (1-x)^{\epsilon b-1}  \frac{\Gamma(\epsilon a+\epsilon b)}{\Gamma(\epsilon a) \Gamma(\epsilon b)} \mathrm{d}x.
 \label{eq:esperanceselonepsilon}
 \end{multline}
In order to compute the limit of (\ref{eq:esperanceselonepsilon}), we evaluate separately the contribution of the integral between $0$ and $1/2$, and between $1/2$ and $1$. By making the change of variable $z=-\epsilon \log(x)$, we have that 
\begin{multline}
 \int_0^{1/2} f\big(-\epsilon \log(x)\big) g\big(-\epsilon \log(1-x)\big) x^{\epsilon a -1} (1-x)^{\epsilon b-1}  \frac{\Gamma(\epsilon a+\epsilon b)}{\Gamma(\epsilon a) \Gamma(\epsilon b)}= \\
 \frac{\Gamma(\epsilon a+\epsilon b)}{\Gamma(\epsilon a) \Gamma(\epsilon b)} \int_{\epsilon \log(2)}^{\infty} f(z) g\big(-\epsilon \log(1-e^{-z/\epsilon})\big) e^{-az} e^{(\epsilon b-1)\log(1-e^{-z/\epsilon})} \mathrm{d}z.
 \label{eq:esperancecoupee}
 \end{multline}
Since 
$$  \frac{\Gamma(\epsilon a+\epsilon b)}{\Gamma(\epsilon a) \Gamma(\epsilon b)} \xrightarrow[\epsilon\to 0]{} \frac{ab}{a+b},$$
the limit of the right-hand-side in (\ref{eq:esperancecoupee}) is 
$$ \frac{b}{a+b} \int_{0}^{\infty} f(z)g(0) ae^{-az} \mathrm{d}z = \frac{b}{a+b} \EE[f(E_a)g(0)].$$ 
The contribution of the integral in (\ref{eq:esperanceselonepsilon}) between $1/2$ and $1$ is computed in the same way, and we find that 
\begin{align*}
\lim_{\epsilon\to 0}  \EE\Big[ f\big(-\epsilon \log(B_{\epsilon})\big)g\big(-\epsilon \log(1-B_{\epsilon})\big)\Big] &=   \frac{b}{a+b} \EE\big[f(E_a)g(0)\big] + \frac{a}{a+b} \EE\big[f(0)g(E_b)\big]\\
&= \EE\Big[ f(\xi E_a)g((1-\xi)E_b)\Big],
\end{align*}
which proves the claim.
\end{proof}
\begin{remark}
Whether $E_a$ and $E_b$ are independent or not does not have any importance. However, it is important that  $E_a$ and $E_b$ are independent of the Bernoulli random variable $\xi$. 
\end{remark}

Let $\alpha_{\epsilon}=\epsilon a, \beta_{\epsilon} = \epsilon b$ and $P_{\epsilon}(t,x)$ the (quenched) distribution function of the endpoint at time $t$ for the Beta random walk with parameters $\alpha_{\epsilon}$ and $\beta_{\epsilon}$. Let $T(n,m)$ be the first-passage time in the Bernoulli-Exponential model with parameters $a,b$. 

It is convenient to define the analogue of the set of weights $w_e$ of the Beta polymer in the context of the Beta RWRE. For an edge $e$ in $(\Z_{\geqslant 0})^2 $ we define $p_e$ by 
$$p_e= \begin{cases}
 B_{j-i, i+j} \text{ if }e\text{ is the vertical edge }(i,j)\to(i, j+1)\\
  1-B_{j-i, i+j} \text{ if }e\text{ is the horizontal edge }(i,j)\to(i+1,j);\\
\end{cases}  $$
where the variables $B_{\cdot , \cdot}$ define the environment of the random walk.
 Lemma \ref{lem:betaconvergence} implies that as $\epsilon$ goes to zero, we 
have the weak convergence 
$$ \min_{\pi : (0,0)\to D_{n,m}}\left\lbrace  \sum_{e\in\pi} -\epsilon \log(p_e)\right\rbrace \Rightarrow \min_{\pi: (0,1)\to D_{n,m}}\left\lbrace \sum_{e\in \pi} t_e\right\rbrace,$$
where the minimum is taken over up-right paths, and the passage times $t_e$ are defined in (\ref{eq:defte}).

Since the times $t_e$ in the FPP model are either zero or exponential, and there is at most one path with zero passage time, the minimum over paths of $\sum_{e\in\pi} t_e$ is attained for a unique path with probability one. We know by the principle of the largest term that as $\epsilon \to 0$,
$$ -\epsilon \log\big(P_{\epsilon}(n+m,m-n)\big) = -\epsilon \log\left( \sum_{\pi : (0,0) \to D_{n,m}} \exp\left(\sum_{e\in\pi} \log(p_e)\right) \right) $$
has the same limit as 
$$ \min_{\pi : (0,0)\to D_{n,m}} \left\lbrace \sum_{e\in\pi} -\epsilon \log(p_e) \right\rbrace. $$
Since the family of rescaled weights $\big(-\epsilon \log(p_e)\big)_e$ weakly converges to $(t_e)_e$, then 
$$  \min_{\pi : (0,0)\to D_{n,m}} \left\lbrace \sum_{e\in\pi} -\epsilon \log(p_e) \right\rbrace \Rightarrow \min_{\pi: (0,0)\to D_{n,m}}\left\lbrace \sum_{e\in \pi} t_e\right\rbrace.$$
Hence for any $n, m\geqslant 0$, $-\epsilon \log(P_{\epsilon}(t,n)$ weakly converges as $\epsilon$ goes to zero to $T(n ,m)$.
\end{proof}

\subsection{Proof of Theorem \ref{thmProb}}
Theorem \ref{thmProb} states that for $r\in\mathbb{R}_{>0}$, one has
\begin{equation*}
\mathbb{P}\big( T(n,m) > r\big) = \det(I-K^{\mathrm{FPP}}_r)_{\mathbb{L}^2(C'_0)}
\end{equation*}
where $C'_0$ is a small positively oriented circle containing $0$ but not $-\nu$, and $K^{\mathrm{FPP}}_r : \mathbb{L}^2(C'_0)\rightarrow \mathbb{L}^2(C'_0)$ is defined by its integral kernel 
\begin{equation}
K^{\mathrm{FPP}}_r(u,u') = \frac{1}{2i\pi} \int_{1/2-i\infty}^{1/2+i\infty} \frac{e^{rs}}{s} 
\frac{g^{\mathrm{FPP}}(u)}{g^{\mathrm{FPP}}(u+s)} \frac{\mathrm{d}s}{s+u-u'}
\label{eq:defkernelFPP}
\end{equation}
where 
\begin{equation}
g^{\mathrm{FPP}}(u) = \left(\frac{a+u}{u}\right)^n \left(\frac{a+u}{a+b+u} \right)^m \frac{1}{u}.
\label{eq:defgfpp}
\end{equation}
\begin{proof}
The proof splits into two pieces. We first show that under appropriate scalings, the Laplace transform $\mathbb{E}\big[ e^{u P_{\epsilon}(n+m, m-n)} \big]$ converges to  $\mathbb{P}\big(T(n,m)\geqslant r\big)$. Then we show that the Fredholm determinant $\det(I-K^{\mathrm{BP}}_u)$ from \ref{thmLaplaceRWRE} converges to $\det(I-K^{\mathrm{FPP}}_r)_{\mathbb{L}^2(C'_0)}$.  

\textbf{First step:} We have an exact formula for $\mathbb{E}\left[ e^{u P_{\epsilon}(n+m, m-n)} \right]$. Let us scale $u$ as $u=-\exp\left(\epsilon^{-1}r\right)$ so that 
$$ \mathbb{E}\left[ e^{u P_{\epsilon}(n+m,m-n)} \right] = \mathbb{E}\left[\exp\left(-e^{-\epsilon^{-1}(-\epsilon \log(P_{\epsilon}(n+m,m-n))-r)}\right) \right].$$
If $f_{\epsilon}(x):= \exp\left(-e^{-\epsilon^{-1}x}\right)$, then the sequence of functions $\lbrace f_{\epsilon} \rbrace$  maps $\mathbb{R}$ to $(0,1)$, is strictly increasing with a limit of $1$ at $+\infty$ and $0$ at $-\infty$, and for each $\delta>0$,  on $\mathbb{R}\setminus [-\delta, \delta]$ converges uniformly to $\mathds{1}_{x>0}$. We define the $r$-shift of $f_{\epsilon}$  as 
$f_{\epsilon}^r(x) = f_{\epsilon}(x-r)$. Then,
$$\mathbb{E}\Big[ e^{u P_{\epsilon}(n+m, m-n)} \Big] = \mathbb{E}\Big[ f_{\epsilon}^r(-\epsilon \log(P_{\epsilon}(n+m, m-n)))\Big].$$ 
Since the variable $T(n,m)$ has an atom in zero, we are not exactly in the situation of Lemma 4.1.38 in \cite{borodin2014macdonald}, but we can adapt the proof. 
Let $s<r<u$. By the properties of the functions $f_{\epsilon}$ mentioned above, we have that for any $ \eta>0$, there exists an $\epsilon_0$ such that for any $\epsilon <\epsilon_0$, 
\begin{multline*}
\mathbb{P}\Big( -\epsilon \log\big(P_{\epsilon}(n+m, m-n)\big) \geqslant u \Big)  \leqslant
 \mathbb{E}\bigg[ f_{\epsilon}^r\Big(-\epsilon \log\big(P_{\epsilon}(n+m, m-n)\big)\Big)\bigg]\leqslant\\
  \mathbb{P}\Big( -\epsilon \log\big(P_{\epsilon}(n+m, m-n)\big) \geqslant s \Big).
\end{multline*}
Since we have established the weak convergence of $-\epsilon \log\big(P_{\epsilon}(n+m, m-n)\big)$ one can take limits as $\epsilon$ goes to zero in the probabilities and we find that 
\begin{multline*}
  \mathbb{P}\big( T(n,m)\geqslant u \big)  \leqslant \liminf_{\epsilon\to 0} \mathbb{E}\bigg[ f_{\epsilon}^r\Big(-\epsilon \log\big(P_{\epsilon}(n+m, m-n)\big)\Big)\bigg] \\
  \leqslant \limsup_{\epsilon\to 0} \mathbb{E}\bigg[ f_{\epsilon}^r\Big(-\epsilon \log\big(P_{\epsilon}(n+m, m-n)\big)\Big)\bigg]\leqslant \mathbb{P}\big( T(n,m) \geqslant s \big). 
  \end{multline*}
Now we take $s$ and $u$ to $r$ and notice that $T(n,m)$ can be decomposed as an atom at zero and an absolutely continuous part. Thus, for any $r>0$, 
$$  \mathbb{P}\big( T(n,m)> r \big)  = \lim_{\epsilon \to  0} \mathbb{E}\bigg[ f_{\epsilon}^r\Big(-\epsilon \log\big(P_{\epsilon}(n+m, m-n)\big)\Big)\bigg].$$

\textbf{Second step: }We shall prove that the limit when $\epsilon$ goes to zero of $\mathbb{E}\left[ e^{u P_{\epsilon}(n+m, m-n)} \right]$ is $\det(I-K^{\mathrm{FPP}}_r)_{\mathbb{L}^2(C_0)}$ where $K_r^{\mathrm{FPP}}$ is defined as in Theorem \ref{thmProb}. For that we take the limit of the Fredholm determinant $K^{RW}$ from Theorem \ref{thmLaplaceRWRE}. 
Let us use the change of variables 
$$ v=\epsilon \tilde{v},\ \  v'=\epsilon \tilde{v}',\ \  s=\epsilon \tilde{s}.$$
Assuming that the limit of the Fredholm determinant is the Fredholm determinant of the limit, which we prove below, we have to take the limit of $\epsilon K^{RW}(\epsilon \tilde{v}, \epsilon\tilde{v}')$. The factor $\epsilon$ in front of $K^{RW}$ is a priori necessary, it comes from the Jacobian of the change of variables $v=\epsilon \tilde{v}$ and $v'=\epsilon \tilde{v}'$. For any $1> \epsilon >0$ the kernel $K^{RW}(v,v')$ can be written as an integral over $\frac 1 2 \epsilon + i\R$ instead of an integral over $\frac 1 2+i\R$, since we do not cross any singularity of the integrand during the contour deformation, and the integrand has exponential decay. Thus, one can write
\begin{equation}
\epsilon K^{RW}(\epsilon \tilde{v}, \epsilon\tilde{v}') = \frac{1}{2i\pi} \int_{1/2 -i\infty}^{1/2 +i\infty} \frac{\epsilon\pi}{\sin(\pi\epsilon \tilde{s})} (-u)^{\epsilon \tilde{s}} \frac{g^{RW}(\epsilon\tilde{v})}{g^{RW}(\epsilon\tilde{v}+\tilde{s})} \frac{\mathrm{d}\tilde{s}}{\tilde{s}+\tilde{v}-\tilde{v}'}.
\label{eq:limiteaprendre}
\end{equation}
 With $u=-\exp\left(\epsilon^{-1}r\right)$, we have that  $(-u)^{\epsilon \tilde{s}} = e^{\tilde{s}r}$. 
Moreover, since 
$$ \lim_{\epsilon\to 0} \epsilon \Gamma(\epsilon z) = \frac{1}{z},$$
we have that 
$$ \lim_{\epsilon\to 0} \frac{g^{RW}(\epsilon\tilde{v})}{g^{RW}(\epsilon\tilde{v}+\tilde{s})} =  \frac{g^{\mathrm{FPP}}(\tilde{v})}{g^{\mathrm{FPP}}(\tilde{v}+\tilde{s})},$$
where $g^{\mathrm{FPP}}$ is defined in (\ref{eq:defgfpp}), and 
$$\lim_{\epsilon\to 0}\frac{\epsilon\pi}{\sin{\epsilon \tilde{s}}}  = \frac{1}{\tilde{s}}.$$
Because the integrand in (\ref{eq:defkernelFPP}) is not absolutely integrable, one cannot apply dominated convergence directly. Instead, we will split the integral  (\ref{eq:limiteaprendre}) into two pieces: the integral over $s$ when $\Imag[\epsilon s ]<1/4$ and the integral over $s$ when  $\Imag[\epsilon s ] \geqslant 1/4$.
Let us begin with some estimates. Since the function $z\mapsto z/\sin(z)$ is holomorphic on a circle of radius $1/2$ around zero, there exists a constant 
$C>0$ such that for $s\in 1/2+i\R$ and $\epsilon>0$ such that $\vert \epsilon s \vert <1/2$, we have 
$$ \Big\vert \frac{\epsilon\pi}{\sin(\pi\epsilon \tilde{s})} -\frac{1}{s} \Big\vert <C\epsilon. $$ 
In order to lighten the notations, we denote 
$$G(\epsilon, \tilde{s}) = \frac{g^{RW}(\epsilon\tilde{v})}{g^{RW}(\epsilon\tilde{v}+\epsilon\tilde{s})} \frac{1}{\tilde{s}+\tilde{v}-\tilde{v}'} .$$ 
The variables $\tilde{v}$ and $\tilde{v}'$ are fixed for the moment. We know that $G(\epsilon, \tilde{s})$ is bounded for $\epsilon $ close to zero and $ \tilde{s}\in 1/2+i\R$. Moreover, there exists a constant $C'>0$ such that for $\vert\epsilon s \vert<1/2$, 
$$\Big\vert G(\epsilon , \tilde{s}) - \frac{g^{\mathrm{FPP}}(\tilde{v})}{g^{\mathrm{FPP}}(\tilde{v}+\tilde{s})} \frac{1}{\tilde{v}+\tilde{s}-\tilde{v}'} \Big\vert <C'\epsilon. $$
We have the decomposition
\begin{multline}
\frac{1}{2i\pi} \int_{\frac12 -\frac{i\epsilon^{-1}}{4}}^{\frac12 +\frac{i\epsilon^{-1}}{4}} \frac{\epsilon\pi}{\sin(\pi\epsilon \tilde{s})} e^{r\tilde{s}} G(\epsilon, \tilde{s})\mathrm{d}\tilde{s} =
\frac{1}{2i\pi} \int_{\frac12 -\frac{i\epsilon^{-1}}{4}}^{\frac12 +\frac{i\epsilon^{-1}}{4}} \left(\frac{\epsilon\pi}{\sin(\pi\epsilon \tilde{s})} - \frac{1}{\tilde{s}}\right) e^{r\tilde{s}} G(\epsilon, \tilde{s})\mathrm{d}\tilde{s}\\ + 
\frac{1}{2i\pi} \int_{\frac12 -\frac{i\epsilon^{-1}}{4}}^{\frac12 +\frac{i\epsilon^{-1}}{4}} \frac{e^{r\tilde{s}}}{\tilde{s}} \Big( G(\epsilon, \tilde{s})-G(0, \tilde{s})\Big)\mathrm{d}\tilde{s}  + 
\frac{1}{2i\pi} \int_{\frac12 -\frac{i\epsilon^{-1}}{4}}^{\frac12 +\frac{i\epsilon^{-1}}{4}} \frac{e^{r\tilde{s}}}{\tilde{s}}G(0, \tilde{s}) \mathrm{d}\tilde{s}.
\label{eq:decoupage}
\end{multline}
The first integral in the R.H.S of (\ref{eq:decoupage}) can be bounded by 
$$ C \epsilon \frac{1}{2\pi} \int_{\frac{1}{2}\epsilon -\frac{i}{4}}^{\frac{1}{2}\epsilon +\frac{i}{4}}  \vert \Gamma(1-s)\vert e^{r/2}  \vert G(\epsilon, s\epsilon^{-1})\vert\mathrm{d}s, $$
which is $\mathcal{O}(\epsilon)$. The second integral in the R.H.S of (\ref{eq:decoupage}) can be bounded by 
$$ C'\epsilon \frac{1}{2\pi}  \int_{\frac{1}{2}-\frac{i\epsilon^{-1}}{4}}^{\frac{1}{2} +\frac{i\epsilon^{-1}}{4}} \frac{e^{r/2}}{\vert \tilde{s}\vert} \mathrm{d}\tilde{s},$$
which is $\mathcal{O}(\epsilon \log(\epsilon^{-1}))$. 
The third integral in the R.H.S of (\ref{eq:decoupage}) converges to a limit as $\epsilon$ goes to zero, even if the integrand is not absolutely integrable. The limit is the improper integral 
$$ \frac{1}{2i\pi}  \int_{1/2-i\infty}^{1/2+i\infty} \frac{e^{r\tilde{s}}}{\tilde{s}} \frac{g^{\mathrm{FPP}}(\tilde{v})}{g^{\mathrm{FPP}}(\tilde{v}+\tilde{s})} \frac{\mathrm{d}\tilde{s}}{\tilde{v}+\tilde{s}-\tilde{v}'} = K^{\mathrm{FPP}}_r(\tilde{v}, \tilde{v}'). $$ 
It remains to show that we have made a negligible error when cutting the tails of the integral. We have 
\begin{multline}
\frac{1}{2i\pi} \int_{\frac12 +\frac{i\epsilon^{-1}}{4}}^{\frac12 +i\infty} \frac{\epsilon\pi}{\sin(\pi\epsilon \tilde{s})} e^{r\tilde{s}} G(\epsilon, \tilde{s})\mathrm{d}\tilde{s}  = \frac{1}{2i\pi} \int_{\frac12 \epsilon +\frac{i}{4}}^{\frac12 \epsilon +i\infty} \frac{\pi}{\sin(\pi s)} e^{rs \epsilon^{-1}} G(\epsilon, s\epsilon^{-1})\mathrm{d}s = \\
\frac{1}{2i\pi} \int_{\frac12 \epsilon +\frac{i}{4}}^{\frac12 \epsilon +i\infty} \frac{\pi}{\sin(\pi s)} e^{rs \epsilon^{-1}} \big(G(\epsilon, s\epsilon^{-1})-1\big)\mathrm{d}s
+
\frac{1}{2i\pi} \int_{\frac12 \epsilon +\frac{i}{4}}^{\frac12 \epsilon +i\infty} \frac{\pi}{\sin(\pi s)} e^{rs \epsilon^{-1}} \mathrm{d}s.
\label{eq:seconddecoupage}
\end{multline}
The first integral in the R.H.S of (\ref{eq:seconddecoupage}) goes to zero by dominated convergence, and the second integral in the R.H.S of (\ref{eq:seconddecoupage}) goes to zero by the Riemann-Lebesgue lemma. 
At this point we have shown that for any $\tilde{v}, \tilde{v}'\in C_0$, 
$$ \lim_{\epsilon \to 0} \epsilon K^{RW}(\epsilon \tilde{v}, \epsilon\tilde{v}') = K^{\mathrm{FPP}}_r(\tilde{v}, \tilde{v}').$$
Observe now that the kernel $K^{\mathrm{FPP}}_r(\tilde{v},\tilde{v}')$ is bounded as $\tilde{v}, \tilde{v}'$ vary along their contour. Using Hadamard's bound, one can bound the Fredholm series expansion of $K^{\mathrm{FPP}}_r$ by an absolutely convergent series of integrals, and conclude by dominated convergence that under the scalings above
$$ \det(I-K_u^{\mathrm{RW}})_{\mathbb{L}^2(C_0)} \xrightarrow[\epsilon\to 0]{}\det(I-K_r^{\mathrm{FPP}})_{\mathbb{L}^2(C_0)}.$$
\end{proof}

\section{Asymptotic analysis of the Beta RWRE}
\label{sec:RWREasymptotics}

Let us first define the Tracy-Widom distribution governing the fluctuations of extreme eigenvalues of Gaussian hermitian random matrices. We refer to \cite[Section 3.2.2]{borodin2014macdonald} for an introduction to Fredholm determinants. 
\begin{definition}
\label{def:TWdist}
The distribution function $F_{\rm GUE}(x)$ of the GUE Tracy-Widom distribution is defined  by $F_{\rm GUE}(x)=\det(I-K_{\rm Ai})_{\mathbb{L}^2(x,+\infty )}$ where $K_{\rm Ai}$ is the Airy kernel,
\begin{equation*}
K_{\rm Ai} (u, v) = \frac{1}{(2i\pi)^2} \int_{e^{-2i\pi/3}\infty}^{e^{2i\pi/3}\infty} \mathrm{d}w \int_{e^{-i\pi/3}\infty}^{e^{i\pi/3}\infty} \mathrm{d}z \frac{e^{z^3/3-zu}}{e^{w^3/3-wv}}\frac{1}{z-w},
\end{equation*}
where the contours for $z$ and  $w$ do not intersect. There is some freedom in the choice of contours. For instance, one can choose the contour for $z$ (resp. $w$) as constituted of two infinite rays departing $1$ (resp. $0$) in directions $\pi/3$ and $-\pi/3$ (resp. $2\pi/3$ and $-2\pi/3$). 
\end{definition}

\subsection{Fredholm determinant asymptotics}

We consider a Beta RWRE $(X_t)_{t\geqslant 0}$ with parameters $\alpha, \beta>0$. 
For a parameter $\theta>0$, we define the quantity
\begin{equation}
x(\theta) = \frac{\Psi_1(\theta+\alpha+\beta) +\Psi_1(\theta )- 2 \Psi_1(\theta + \alpha)}{\Psi_1(\theta) - \Psi_1(\theta + \alpha+\beta)}
\label{eq:defx}
\end{equation}
and the function $I:\big(\frac{\alpha-\beta}{\alpha+\beta}, 1\big) \to \R_{>0}$ such that 
\begin{equation}
I\big(x(\theta)\big) = \frac{\Psi_1(\theta+\alpha+\beta) - \Psi_1(\theta + \alpha)}{\Psi_1(\theta) - \Psi_1(\theta + \alpha+\beta)} \Big(\Psi(\theta + \alpha+\beta)- \Psi(\theta)   \Big)+ \Psi(\theta + \alpha+\beta)- \Psi(\theta+\alpha),  
\label{eq:defI}
\end{equation}
where $\Psi$ is the digamma function ($\Psi(z)= \Gamma'(z)/\Gamma(z)$) and $\Psi_1$ is the trigamma function ($\Psi_1(z)=\Psi'(z)$). Moreover, we define a real-valued $\sigma(\theta)>0$ such that 
\begin{equation}
2\sigma(\theta)^3 = \Psi_2(\theta+\alpha) - \Psi_2(\alpha+\beta+\theta) + \frac{\Psi_1(\alpha+\theta) - \Psi_1(\alpha+\beta+\theta)}{\Psi_1(\theta) - \Psi_1(\alpha+\beta+\theta)}\left( \Psi_2(\alpha+\beta+\theta) - \Psi_2(\theta)\right).
\label{eq:defsigma}
\end{equation}
The fact that we can choose $\sigma(\theta)>0$ is proved in Lemma \ref{lemma:thirdderivative}. We will see that a critical point Fredholm determinant asymptotic analysis shows that for all $\theta>0$ and $\alpha, \beta>0$,
\begin{equation}
\lim_{t\to\infty} \PP\left( \frac{\log\Big(P\big(t, x(\theta)t\big)\Big) + I\big(x(\theta)\big)t}{t^{1/3}\sigma(\theta)} \leqslant y \right)  = F_{\rm GUE}(y).
\end{equation}
However, due to increased technical challenges in the general parameter case, we presently prove rigorously only the case of Theorem \ref{thm:RWasymptotics}, which deals with $\alpha=\beta=1$ (i.e. when the $B_{x,t}$ variables are distributed uniformly on $(0,1)$). 

When $\alpha=\beta$ the expressions for $x(\theta)$ and $I(x(\theta))$ simplify. We find that 
$$ x(\theta) = \frac{1+2\theta}{\theta^2+(\theta+1)^2}$$
and
$$ I\big(x(\theta)\big) = \frac{1}{\theta^2+(\theta+1)^2} ,$$
so that the rate function $I$ is simply the function $I:x \mapsto 1-\sqrt{1-x^2}$. We also find that for $\alpha=\beta=1$,
\begin{equation}
   \sigma(\theta)^3 = \frac{1}{\theta + 3 \theta^2+ 4\theta^3 + 2\theta^4} =\frac{2\left(1-\sqrt{1-x^2} \right)^2}{\sqrt{1-x^2}} = \frac{2 I(x)^2}{1-I(x)}, 
   \label{eq:simpleexpressionsigma}
\end{equation}
 where $x=x(\theta)$. 
\begin{theorem}
\label{thm:RWasymptotics}
For $0<\theta<1/2$ and $\alpha=\beta=1$, we have that 
\begin{equation}
\lim_{t\to\infty} \PP\left( \frac{\log\Big(P\big(t, x(\theta)t\big)\Big) + I\big(x(\theta)\big)t}{t^{1/3}\sigma(\theta)} \leqslant y \right)  = F_{\rm GUE}(y).
\end{equation}
\end{theorem}
The rest of this section is devoted to the proof of Theorem \ref{thm:RWasymptotics}. Most arguments in the proof apply equally  for any parameters $\alpha, \beta$ except the deformation of contours which is valid for small $\theta$ and  Lemma \ref{lemma:steepdescentw} which is only valid for $\alpha=\beta =1$. We expect the general $\alpha, \beta, \theta$ to still hold but do not attempt to extend to that case. 

We first observe that we do not need to invert the Laplace transform of $P(t, x(\theta)t)$. 
Setting $u=-e^{t I(x(\theta)) - t^{1/3} \sigma(\theta) y}$, 
one has that 
\begin{equation}
\lim_{t\to\infty }\mathbb{E}\left[ e^{u P(t,x(\theta)t)} \right] = \lim_{t\to\infty} \mathbb{P}\left(\frac{\log\Big(P\big(t, x(\theta) t\big)\Big)+ I\big(x(\theta)\big)t}{t^{1/3}\sigma(\theta)} <y \right). 
\end{equation}
This convergence is justified by Lemma 4.1.39 in \cite{borodin2014macdonald}, provided that the limit is a continuous  probability distribution function, and we see later that this is the case. Hence, in order to prove Theorem \ref{thm:RWasymptotics}, one has to take the $t\to\infty$ limit of the Fredholm determinant  (\ref{eq:fredholmRWRE}) in the statement of Theorem \ref{thmLaplaceRWRE}.

The asymptotic analysis of this Fredholm determinant proceeds by steepest descent analysis, and is very close to the analysis presented in the recent  papers \cite{borodin2012free,ferrari2013tracy,barraquand2014phase,veto2014tracy,corwin2014strict,barraquand2015q}, that deal with similar kernels. Let us assume for the moment that the contour $C_0$ is a circle around $0$ with very small radius. One can make the change of variables $v+s=z$ in the kernel $K_u^{\mathrm{RW}}$ so that, with the value of $u$ that we choose,
\begin{equation*}
K^{\mathrm{RW}}_u(v,v') = \frac{1}{2i\pi} \int_{1/2-i\infty}^{1/2+i\infty} \frac{\pi}{\sin(\pi (z-w))} e^{(z-w)\big(tI(x(\theta))-t^{1/3}\sigma(\theta)y\big)}\frac{g^{\mathrm{RW}}(v)}{g^{\mathrm{RW}}(z)} \frac{\mathrm{d}z}{z - v'}, 
\end{equation*} 
and the contour for $z$ can be chosen as $1/2+i\R$. The kernel can be rewritten
\begin{equation}
K^{\mathrm{RW}}_u(v,v') = \frac{1}{2i\pi} \int_{1/2-i\infty}^{1/2+i\infty} \frac{\pi}{\sin(\pi (z-w))} \exp\left( t(h(z)-h(v)) -t^{1/3}\sigma(\theta)y (z-v)\right) \frac{\Gamma(v)}{\Gamma(z)}\frac{\mathrm{d}z}{z - v'},
\label{eq:firstkernel}
\end{equation} 
where 
$$ h(z) = I\big(x(\theta)\big) z  + \frac{1-x(\theta)}{2}\log\left(\frac{\Gamma(\alpha+z)}{\Gamma(z)} \right) +  \frac{1+x(\theta)}{2}\log\left(\frac{\Gamma(\alpha+z)}{\Gamma(\alpha+\beta+z)} \right).$$

 The function $h$ governs the asymptotic behaviour of the Fredholm determinant of $K_u^{\mathrm{RW}}$. The principle of the steepest-descent method is to deform the integration contour -- both the contour in the definition of $K_u^{\mathrm{RW}}$ and the $\mathbb{L}^2$ contour -- so that they go across a critical point of the function $h$. Then one needs to prove that only the integration around the critical point has a contribution in the limit, and one can approximate all terms by their Taylor approximation close to the critical point. 
 
 The first derivatives of $h$ are  
$$ h'(z) = I\big(x(\theta)\big)  + \Psi(\alpha+z) - \frac 1 2  \Psi(z) -\frac 1 2 \Psi(\alpha+\beta+z) +\frac{x(\theta)}{2}\Big(\Psi(z) - \Psi(\alpha+\beta + z) \Big), $$
and 
$$ h''(z) =  \Psi_1(\alpha+z) - \frac 1 2  \Psi_1(z) -\frac 1 2 \Psi_1(\alpha+\beta+z) +\frac{x(\theta)}{2}\Big(\Psi_1(z) - \Psi_1(\alpha+\beta + z) \Big).$$
One readily sees that the expressions for $x(\theta)$ and $I(x(\theta))$ in (\ref{eq:defx}) and (\ref{eq:defI}) are precisely chosen so that $h'(\theta)=h''(\theta)=0$. Let us give an expression of $h'$ in terms of $\theta$:
\begin{multline}
h'(z) = \Psi(z+\alpha) - \Psi(\alpha+\beta+z) + \frac{\Psi_1(\alpha+\theta) - \Psi_1(\alpha+\beta+\theta)}{\Psi_1(\theta) - \Psi_1(\alpha+\beta+\theta)}\Big( \Psi(\alpha+\beta+z) - \Psi(z)\Big)\\
- \left(\Psi(\theta+\alpha) - \Psi(\alpha+\beta+\theta) + \frac{\Psi_1(\alpha+\theta) - \Psi_1(\alpha+\beta+\theta)}{\Psi_1(\theta) - \Psi_1(\alpha+\beta+\theta)}\Big( \Psi(\alpha+\beta+\theta) - \Psi(\theta)\Big) \right).
\label{eq:hprimetheta}
\end{multline}
Expressions are much simpler in the case $\alpha=\beta=1$. In that case we have 
\begin{align}
h'(z) &= \frac{1}{\theta+1}-\frac{1}{z+1} + \frac{1}{1+ \left(\frac{\theta+1}{\theta} \right)^2} \left(\frac{2z+1}{z(z+1)}- \frac{2\theta+1}{\theta(\theta+1)} \right),\nonumber \\
&= \frac{(\theta-z)^2}{z(1+z)(1+2\theta+2 \theta^2)}.
\label{eq:expressionshprimeuniform}
\end{align}

In order to understand the behaviour of $\Real[h]$ around the critical point $\theta$, we also need the sign of the third derivative of $h$. 
\begin{lemma}
For any $\alpha, \beta, \theta >0$, we have that $h'''(\theta)>0 $.
\label{lemma:thirdderivative}
\end{lemma}
Lemma \ref{lemma:thirdderivative} is proved in Section \ref{subsec:rigorousestimates}.

By the definition of $\sigma(\theta)$ in (\ref{eq:defsigma}),  $\sigma(\theta)=\left(\frac{h'''(\theta)}{2}\right)^{1/3}$. 
Then, using Taylor expansion, we have that for $z$ in a neighbourhood of $\theta$,
\begin{equation}
h(z) - h(\theta) \approx \frac{\big( \sigma(\theta)(z-\theta)\big)^3}{3}.
\label{eq:taylorexpansion}
\end{equation}

 We now deform the integration contour in (\ref{eq:firstkernel}) and the Fredholm determinant contour which was initially a small circle around $0$. Let $\mathcal{D}_{\theta}$ be the vertical line $\mathcal{D}_{\theta} = \lbrace \theta + iy : y\in \R \rbrace$, and $\mathcal{C}_{\theta}$ be the circle centred in $0$ with radius $\theta$. 
This deformation of contours does not change the Fredholm determinant $\det(I+K_u^{\mathrm{RW}})$ only if 
\begin{itemize}
\item All the poles of the sine inverse in (\ref{eq:firstkernel}) corresponding with $z-w\in \Z_{>0}$ stay on the right of $\mathcal{D}_{\theta}$.
\item We do not cross the pole of $h$ at $-\alpha-\beta$ when deforming the $\mathbb{L}^2$ contour. 
\end{itemize}
Hence, we will assume that $\theta<\min(\alpha+\beta, \frac{1}{2})$ so that the two above conditions are satisfied. 

\begin{lemma}
For any parameters $\alpha, \beta>0$, and $\theta >0$, the contour $\mathcal{D}_{\theta}$ is steep-descent for the function $\Real[h]$ in the sense that $y\mapsto \Real[h(\theta+iy)]$ is decreasing for $y$ positive and increasing for $y$ negative. 
\label{lemma:steepdescentz}
\end{lemma}
Lemma \ref{lemma:steepdescentz} is proved in Section \ref{subsec:rigorousestimates}. The step which prevents us from proving Theorem \ref{thm:RWasymptotics} for any parameters $\alpha, \beta>0$ is the steep-descent properties of the contour $\mathcal{C}_{\theta}$. 
\begin{lemma}
Assume $\alpha=\beta=1$. Then the contour $\mathcal{C}_{\theta}$ is steep descent for the function $-\Real[h]$, in the sense that $y\mapsto \Real[h(\theta e^{i\phi})]$ is increasing for $\phi\in (0, \pi)$ and decreasing for $\phi\in (-\pi, 0)$. 
\label{lemma:steepdescentw}
\end{lemma}
Lemma \ref{lemma:steepdescentw} is proved in Section \ref{subsec:rigorousestimates}. Proving Lemma \ref{lemma:steepdescentw} for arbitrary parameters  $\alpha, \beta$ turns out to be computationally difficult, and we do not pursue that here. 

In the rest of this section, although the proofs are quite general and do not depend on the value of parameters, we assume that $\alpha=\beta=1$ so that we can use Lemma \ref{lemma:steepdescentw}.  Let us show that the only part of the contours that contributes to the limit of the Fredholm determinant when $t$ tends to infinity is a neighbourhood of the critical point $\theta$.
\begin{proposition}Let $B(\theta, \epsilon)$ be the ball of radius $\epsilon$ centred at $\theta$. We note $\mathcal{C}_{\theta}^{\epsilon}$ (resp. $\mathcal{D}_{\theta}^{\epsilon}$) the part of the contour $\mathcal{C}_{\theta}$ (resp. $\mathcal{D}_{\theta}$) inside the ball $B(\theta, \epsilon)$. Then, for any $\epsilon>0$, 
$$ \lim_{t\to\infty} \det\big(I-K^{\mathrm{RW}}_u\big)_{\mathbb{L}^2(\mathcal{C}_{\theta})} = \lim_{t\to\infty} \det\big(I-K^{\mathrm{RW}}_{y, \epsilon}\big)_{\mathbb{L}^2(\mathcal{C}_{\theta}^{\epsilon})}$$
where $K^{\mathrm{RW}}_{y, \epsilon}$ is defined by the integral kernel
\begin{equation}
K^{\mathrm{RW}}_{y, \epsilon}(v,v') = \frac{1}{2i\pi} \int_{\mathcal{D}_{\theta}^{\epsilon}} \frac{\pi}{\sin(\pi (z-w))} \exp\left( t(h(z)-h(v)) -t^{1/3}\sigma(\theta)y (z-v)\right) \frac{\Gamma(v)}{\Gamma(z)}\frac{\mathrm{d}z}{z - v'}.
\label{eq:kerneltruncated}
\end{equation} 
\label{prop:localization}
\end{proposition}
\begin{proof}
By Lemmas \ref{lemma:steepdescentz} and \ref{lemma:steepdescentw}, there exists a constant $C>0$ such that if $v\in \mathcal{C}_{\theta}$ and $z\in \mathcal{D}_{\theta}\setminus \mathcal{D}_{\theta}^{\epsilon}$, then 
$$\re[h(z)-h(v)]<-C.$$
and consequently
$$\exp\left( t(h(z)-h(v)) -t^{1/3}\sigma(\theta)y (z-v)\right)\frac{\mathrm{d}z}{z - v'} \xrightarrow[t\to\infty]{} 0.$$
Since $\frac{\pi}{\sin(\pi (z-w))\Gamma(z)}$ has exponential decay in the imaginary part of $z$, the contribution of the integration over $\mathcal{D}_{\theta}\setminus \mathcal{D}_{\theta}^{\epsilon}$ is negligible (by dominated convergence). Thus, $K^{\mathrm{RW}}_{y}(v,v')$ and $K^{\mathrm{RW}}_{y, \epsilon}(v,v')$ have the same limit when $t$ goes to infinity. 

By Lemmas \ref{lemma:steepdescentz} and \ref{lemma:steepdescentw}, there exists another constant $C'>0$ such that if $v\in \mathcal{C}_{\theta}\setminus\mathcal{C}_{\theta}^{\epsilon}$ and $z\in \mathcal{D}_{\theta}$, then 
$$\re[h(z)-h(v)]<-C'.$$
Consider the Fredholm determinant expansion 
$$\det\big(I-K^{\mathrm{RW}}_{u}\big) = 1 + \sum_{n=1}^{\infty} \frac{(-1)^n}{n!} \int\dots\int \det\Big(K^{\mathrm{RW}}_{u}(w_i, w_j)\Big)_{i,j=1}^{n} \mathrm{d}w_1\dots \mathrm{d}w_n.$$
The $k^{th}$ term can be decomposed as the sum of the integration over $\left(\mathcal{C}_{\theta}^{\epsilon}\right)^k$ plus the integration over $\left(\mathcal{C}_{\theta}\right)^k\setminus\left(\mathcal{C}_{\theta}^{\epsilon}\right)^k$. The second contribution goes to zero since it will be possible to factorize $e^{-C't}$. Finally, the proposition is proved using again dominated convergence on the Fredholm series expansion, which is absolutely summable by Hadamard's bound. 
\end{proof}
Let us rescale the variables around $\theta$ by the change of variables 
$$z=\theta+ t^{-1/3}\tilde{z},\ \ v=\theta+ t^{-1/3}\tilde{v}, \ \  v'=\theta+ t^{-1/3}\tilde{v}'.$$
The Fredholm determinant of $K^{\mathrm{RW}}_{y, \epsilon}$ on the contour $\mathcal{C}_{\theta}^{\epsilon}$ equals the Fredholm determinant of the rescaled kernel 
$$ K^t_{y, \epsilon}(\tilde{v}, \tilde{v'})  = t^{-1/3}K^{\mathrm{RW}}_{y, \epsilon}\Big(\theta+ t^{-1/3}\tilde{v}, \theta+ t^{-1/3}\tilde{v}'\Big)$$
acting on the contour $\mathcal{C}_{\theta}^{t^{1/3}\epsilon}$. 

It is more convenient to change again the contours. For $L\in \R_{>0}$, define the contour 
\begin{equation}
 \mathcal{C}^L := \left\lbrace\vert y  \vert e^{i(\pi-\phi) \cdot \sgn(y)} : y\in [0,L]\right\rbrace,
 \label{eq:defnewcontour}
\end{equation}
where $\phi$ is some angle $\phi\in (\pi/6, \pi/2)$ to be chosen later. We also set
\begin{equation}
\mathcal{C} := \left\lbrace\vert y  \vert e^{i(\pi-\phi) \cdot \sgn(y)} : y \geqslant 0\right\rbrace.
\label{eq:defcontourC}
\end{equation}

The contour $\mathcal{C}_{\theta}^{\epsilon}$ is an arc of circle and crosses $\theta$ vertically. For $\epsilon$ small enough, one can replace the contour $\mathcal{C}_{\theta}^{\epsilon}$ by $\mathcal{C}^L$ without changing the Fredholm determinants. The values of $L$ and $\phi$ has to be chosen so that the endpoints of the contours coincide. 

We define the rescaled contour for the variable $\tilde{z}$ by
$$  \mathcal{D}^L := \left\lbrace i y : y\in [-L,L]\right\rbrace,$$
and we set $\mathcal{D}:=i\R$.
\begin{proposition}
We have that 
$$ \lim_{t\to\infty} \det(I-K^{\mathrm{BP}}_{y, \epsilon})_{\mathbb{L}^2(\mathcal{C}_{\theta}^{\epsilon})}= \det(I+K_y)_{\mathbb{L}^2(\mathcal{C})},$$
where $K_y$ is defined by its integral kernel 
$$ K_y(w,w') = \frac{1}{2i\pi} \int_{\infty e^{-i\pi/3}}^{\infty e^{i\pi/3}} \frac{\mathrm{d}z}{(z-w')(w-z)} \frac{e^{z^3/3-yz} }{e^{w^3/3-yw}}$$
where the contour for $z$ is a wedge-shaped contour constituted of two rays going to infinity in the directions $e^{-i\pi/3}$ and $e^{i\pi/3}$, such that it does not intersect $\mathcal{C}$. 
\label{prop:limitfredholm}
\end{proposition}
The proof of Proposition \ref{prop:limitfredholm} follows the lines of \cite[Proposition 6.4]{ferrari2013tracy} (see also \cite[Proposition 6.13]{barraquand2015q}).
\begin{proof}
We take the limit of the rescaled kernel $\det(I-K^t_{y, \epsilon}(\tilde{v}, \tilde{v}'))$. Let us first examine the pointwise convergence. Under the scalings above
\begin{align*}
\frac{t^{-1/3}\pi}{\sin(\pi (z-v))}&\xrightarrow[t\to\infty]{} \frac{1}{\tilde{z}-\tilde{v}},\\
\frac{\mathrm{d}z}{z-v'} &\xrightarrow[t\to\infty]{}  \frac{\mathrm{d}\tilde{z}}{\tilde{z}-\tilde{v}'},\\
\frac{\Gamma(v)}{\Gamma(z)} &\xrightarrow[t\to\infty]{}1,\\
 t(h(z)-h(v)) &\xrightarrow[t\to\infty]{} \frac{\sigma(\theta)^3}{3}(\tilde{z}^3 - \tilde{v}^3).
\end{align*}
Now we justify that one can take the pointwise limit. We take $\mathcal{D}^{\epsilon t^{1/3}}$ as the integration contour for the $\tilde{z}$ variable.  Since $\tilde{z}$ is pure imaginary, $\exp(\tilde{z}^3/3 - \tilde{z} y \sigma(\theta))$ has modulus one. Moreover for fixed $\tilde{v}$ and $\tilde{v}'$, we can find a constant $C'''>0$ such that 
$$  \frac{t^{-1/3}\pi}{\sin(\pi (z-v))} \frac{\mathrm{d}z}{z-v'} <\frac{C'''}{(\Imag(\tilde{z})^2)}.$$ 
This means that the integrand of $K^t_{y, \epsilon}(\tilde{v}, \tilde{v'})$ has quadratic decay, which is enough to apply dominated convergence. It results 
that 
$$ \lim_{t\to \infty} K^t_{y, \epsilon}(\tilde{v}, \tilde{v'}) = \frac{1}{2i\pi} \int_{\mathcal{D}^{\infty}} \frac{e^{\tilde{z}^3\sigma(\theta)^3/3- \tilde{z}y\sigma(\theta) }}{e^{\tilde{v}^3\sigma(\theta)^3/3- \tilde{v}y\sigma(\theta) }}\frac{1}{\tilde{z}-\tilde{v}}\frac{\mathrm{d}\tilde{z}}{\tilde{z}-\tilde{v}'}.$$

Now we need to prove that one can exchange the limit with the Fredholm determinant. By Taylor expansion, there exists a constant $C>0$ such that for $\vert v-\theta\vert <\epsilon$,
\begin{equation}
\Big\vert  t \cdot h(v) - \frac{\sigma(\theta)^3}{3}(\tilde{v})^3\Big\vert <Ct\vert v-\theta\vert^4.
\label{eq:thirdorderTaylor}
\end{equation}

Our aim now is to show that 
we may find constants $C', C''>0$ such that for $\tilde v, \tilde v'$ along their contour, 
\begin{equation}
\vert K^t_{y, \epsilon}(\tilde v, \tilde v') \vert <C'' \exp(-C' \vert \tilde v\vert ^3).
\label{eq:originalbound}
\end{equation}
However, there was a lack of rigor in the justifications given in previous versions of the present paper, including the published version. Indeed, as pointed out to us by Sergei Korotkikh, the argument  of the complex variable  $\tilde v$ is $\pi/2+O(\epsilon)$, so that the real part of $\frac{\sigma(\theta)^3}{3}\tilde v^3$ may not decay faster than the bound in the R.H.S. of \eqref{eq:thirdorderTaylor},  as $ \vert \tilde v\vert$ increases along the contour.

Instead, we need to consider a higher order Taylor approximation, as it was kindly suggested to us by Sergei Korotkikh (see also \cite{hidden2021korotkikh} for an alternative approach). Using \eqref{eq:expressionshprimeuniform}, we may compute the fourth derivative of the function $h$ and find that
\begin{equation}
h^{(4)}(\theta) = - \frac{6(1+2\theta)}{\theta^2(1+\theta)^2(1+2\theta+2\theta^2)}.
\label{eq:fourthderivative}
\end{equation}
By Taylor expansion, there exist a constant $C>0$ such that for $\vert v-\theta\vert <\epsilon$ with $\epsilon$ chosen small enough,
\begin{equation}
\left\vert th(v) -\frac{\sigma(\theta)^3}{3}\tilde v^3 -\frac{t^{-1/3}h^{(4)}(\theta)}{4!}\tilde v^4 \right\vert  < C t\vert v-\theta \vert^5 < C \epsilon^2 \vert \tilde v \vert ^3.
\label{eq:fourthorderTaylor}
\end{equation}
According to the choice of contours made in \eqref{eq:defnewcontour}, the variable $\tilde v$ belongs to a contour formed by two segments leaving $0$ with angle $\pm \phi(\epsilon)$ where $\phi(\epsilon)=\frac{\pi}{2}+ \frac{\epsilon}{2\theta} +o(\epsilon)$, see Figure \ref{fig:contours}.
\begin{figure}
	\begin{center}
		\begin{tikzpicture}[scale=2]
		\draw[thick, gray, ->] (-0.5,0) -- (1.6,0);
		\draw[thick, gray, ->] (0,-1.2) -- (0,1.2);
		\draw (1,0) arc(0:120:1);
		\draw (1,0) arc(0:-120:1);
		\draw[dashed] (0,0) -- (30:1);
		\draw[ultra thick] (1,0) -- (30:1);
		\draw[dashed] (0,0) -- (-30:1);
		\draw[ultra thick] (1,0) -- (-30:1);
		\draw (1,0) node[anchor=north west]{$\theta$};
		\draw[->] (.45,0) arc(0:30:.45);
		\draw (.3,.06) node{$\alpha$};
		\draw[<->, black!70] (.1,.05) + (30:1) -- (1.1,0.05);
		\draw[black!70] (1.15,0.3) node{$\epsilon$};
		\draw (1.05,-0.5) node{$\mathcal C_{\theta}^{\epsilon}$};
		\draw[fill=white] (0,0.6) node[opacity=1]{$\boxed{\alpha=\frac{\epsilon}{\theta}+ o(\epsilon)}$};
		\begin{scope}[xshift=4cm]
		\draw[thick, gray, ->] (-1,0) -- (1,0);
		\draw[thick, gray, ->] (0,-1) -- (0,1);
		\draw[ultra thick] (0,0) -- (105:1);
		\draw[ultra thick] (0,0) -- (-105:1);
		\draw[<->, black!70] (-0.1,-0.05) + (105:1) -- (-0.1,0.03);
		\draw[black!70] (-0.45,0.4) node{$\epsilon t^{1/3}$};
		\draw[->]  (0.3,0) arc(0:105:0.3);
		\draw (1,0.3) node{$\phi(\epsilon)= \frac{\pi}{2}+\frac{\epsilon}{2\theta}+o(\epsilon)$};
		\end{scope}
		\end{tikzpicture}
	\end{center}		
	\caption{Left: The contour $\mathcal C_{\theta}^{\epsilon}$ for variables $v,v'$ is shown (thick black segments). Since the length of each segment is $\epsilon$, the length of the corresponding arc it intercepts is $\epsilon+o(\epsilon)$. Taking into account that the circle has radius $\theta$, it  implies that the angle $\alpha$ shown in the figure is $\alpha=\frac{\epsilon}{\theta}+ o(\epsilon)$. Right: The contour (thick black segments) for variables $\tilde v, \tilde v'$ obtained after the change of variables $v=\theta+t^{-1/3}\tilde v$. The angle of the segments is the same as on the left, that is $\phi(\epsilon)= \frac{\pi}{2}+\frac{\epsilon}{2\theta}+o(\epsilon)$.}
	\label{fig:contours}
\end{figure}

Thus, using that $t^{-1/3} \vert \tilde v\vert <\epsilon$ and $h^{(4)}(\theta)<0$,
\begin{align}
\mathrm{Re}\left[\frac{-\sigma(\theta)^3}{3}\tilde v^3 - \frac{t^{-1/3}h^{(4)}(\theta)}{4!}\tilde v^4\right] &= -\sin\left(\frac{3\epsilon}{2\theta} +o(\epsilon)\right)\frac{\sigma(\theta)^3}{3}\vert \tilde v\vert ^3 - \cos\left(\frac{2\epsilon}{\theta} +o(\epsilon)\right) \frac{t^{-1/3}h^{(4)}(\theta)}{4!}\vert \tilde v\vert^4 \nonumber \\
&< - \epsilon \vert \tilde v\vert ^3 \left(\frac{\sigma(\theta)^3}{2\theta}+ \frac{h^{(4)}(\theta)}{4!} \right)+ o(\epsilon)  
\label{eq:boundonrealpart}
\end{align}
Using \eqref{eq:fourthderivative} and the expression of $\sigma(\theta)$ in \eqref{eq:defsigma}, we find that
$$\frac{\sigma(\theta)^3}{2\theta}+ \frac{h^{(4)}(\theta)}{4!}= \frac{1}{4\theta^2(1+\theta)^2(1+2\theta+2\theta^2)} >0,  $$
Combining \eqref{eq:fourthorderTaylor} and \eqref{eq:boundonrealpart}, there exist a constant $c>0$ such that
\begin{equation*}
\mathrm{Re}\left[-th(v) \right] <-c \epsilon \vert \tilde v\vert^3 + C \epsilon^2 \vert \tilde v \vert ^3.
\end{equation*}
Hence, choosing $\epsilon$ small enough, we may find constants $C', C''>0$ such that
\begin{equation*}
\vert K^t_{y, \epsilon}(\tilde v, \tilde v') \vert <C'' \exp(-C' \epsilon \vert \tilde v\vert ^3).
\end{equation*}
Thus,  the integrand of the rescaled kernel decays exponentially and we can apply dominated convergence. 
Further, the Fredholm determinant expansion of $K^t$ is integrable and summable (using Hadamard's bound), and dominated convergence implies that 
the limit of $ \det(I+K^{\mathrm{BP}}_{y, \epsilon})_{\mathbb{L}^2(\mathcal{C}_{\theta}^{\epsilon})}$ is the Fredholm determinant of an operator $\tilde{K}_y$ acting on $\mathcal{C}$ defined by the integral kernel 
$$ \tilde{K}_y(\tilde{v}, \tilde{v}') =\frac{1}{2i\pi} \int_{\mathcal{D}^{\infty}} \frac{e^{\tilde{z}^3\sigma(\theta)^3/3- \tilde{z}y\sigma(\theta) }}{e^{\tilde{v}^3\sigma(\theta)^3/3- \tilde{v}y\sigma(\theta) }}\frac{1}{\tilde{z}-\tilde{v}}\frac{\mathrm{d}\tilde{z}}{\tilde{z}-\tilde{v}'}.$$
Since the integrand of $\tilde{K}_y$ has quadratic decay on the tails of the contour $\mathcal{D}^{\infty}$ one can freely deform the contours so that it goes from $\infty e^{-i\pi/3}$ to  $\infty e^{i\pi/3}$ without intersecting $\mathcal{C}^{\infty}$. 
Finally, by doing another change of variables to eliminate the dependency in $\sigma(\theta)$ in the integrand, one recovers the Fredholm determinant of $K_y$ as claimed. 
\end{proof}
Using the $\det(I-AB)=\det(I-BA)$ trick, one can reformulate the Fredholm determinant of $K_y$ as the Fredholm  determinant of an operator on $\mathbb{L}^2(y, \infty)$ (see e.g. \cite[Lemma 8.6]{borodin2012free}). It turns out that 
$$\det(I+K_y)_{\mathbb{L}^2(\mathcal{C})} =\det(I-K_{\rm Ai})_{\mathbb{L}^2(x,+\infty )}, $$
and this concludes the proof of Theorem \ref{thm:RWasymptotics}.

\subsection{Precise estimates and steep-descent properties}
\label{subsec:rigorousestimates}
The following series representations will be useful:
\begin{equation}
\Psi(z)-\Psi(w) = \sum_{n=0}^{\infty} \frac{z-w}{(n+z)(n+w)},
\label{eq:seriesrepresentationdigamma}
\end{equation}
is valid for $z$ and $w$ away from the negative integers. We also use
\begin{equation}
\Psi_1(z)- \Psi_1(w) = \sum_{n=0}^{\infty}\left[\frac{1}{(n+z)^2} - \frac{1}{(n+w)^2}\right].
\label{eq:seriesrepresentationtrigamma}
\end{equation}

\begin{proof}[Proof of Lemma \ref{lemma:thirdderivative}]
Given the expression (\ref{eq:hprimetheta}) for the first derivative of $h$, we have 
\begin{equation}
h'''(\theta) = \Psi_2(\theta+\alpha) - \Psi_2(\alpha+\beta+\theta) + \frac{\Psi_1(\alpha+\theta) - \Psi_1(\alpha+\beta+\theta)}{\Psi_1(\theta) - \Psi_1(\alpha+\beta+\theta)}\Big( \Psi_2(\alpha+\beta+\theta) - \Psi_2(\theta)\Big), 
\label{eq:hthird}
\end{equation}
 where $\Psi_2$ is the second polygamma function ($\Psi_2(z)=\frac{\mathrm{d}}{\mathrm{d}z}\Psi_1(z))$. Hence $h'''(\theta)>0$ is equivalent to 
 \begin{multline*}
\Big(\Psi_2(\theta+\alpha+\beta) - \Psi_2(\theta+\alpha)\Big)\Big(\Psi_1(\theta+\alpha+\beta) -\Psi_1(\theta) \Big)\\
 - \Big(\Psi_1(\theta+\alpha+\beta) - \Psi_1(\theta+\alpha)\Big)\Big(\Psi_2(\theta+\alpha+\beta) -\Psi_2(\theta) \Big) >0,
 \end{multline*}
 which is equivalent to 
\begin{equation}
\frac{\Psi_2(\theta+\alpha+\beta) - \Psi_2(\theta+\alpha)}{\Psi_1(\theta+\alpha+\beta) - \Psi_1(\theta+\alpha)}>\frac{\Psi_2(\theta+\alpha+\beta) -\Psi_2(\theta)}{\Psi_1(\theta+\alpha+\beta) -\Psi_1(\theta) }.
\label{eq:compare}
\end{equation}
The function trigamma $\Psi_1$ is positive and decreasing on $\R_{>0}$. The function $\Psi_2$ is negative and increasing. One recognizes in (\ref{eq:compare}) difference quotients for the function $\Psi_2 \circ  \Psi_1^{-1}$. Thus, it is enough to prove that $ \Psi_2 \circ  \Psi_1^{-1}$ is strictly concave.  The derivative of $ \Psi_2 \circ  \Psi_1^{-1}$ is $ \Psi_3\circ \Psi_1^{-1} / \Psi_2\circ\Psi_1^{-1}$. Since $\Psi_1$ is decreasing, it is enough to show that $\Psi_3/\Psi_2$ is increasing, which, by taking the derivative, is equivalent to $\Psi_4 \Psi_2 >\Psi_3 \Psi_3$. 

For all $n\geqslant 1$, one has the integral representation 
\begin{equation}
\Psi_n(x) =  - \int_0^{\infty} \frac{(-t)^n e^{-xt}}{1-e^{-t}}  \mathrm{d}t.
\label{eq:representationPolygamma}
\end{equation}
Thus for $x>0$, $\Psi_4(x) \Psi_2(x) >\Psi_3(x) \Psi_3(x)$ is equivalent to 
\begin{equation*}
\int_0^{\infty} \int_0^{\infty}\frac{e^{-xt-xu}}{(1-e^{-t})(1-e^{-u})}t^3 u^3 < \int_0^{\infty} \int_0^{\infty} \frac{e^{-xt-xu}}{(1-e^{-t})(1-e^{-u})}t^2 u^4.
\end{equation*}
By symmetrizing the right-hand-side, the inequality is equivalent to 
\begin{equation*}
\int_0^{\infty} \int_0^{\infty} \frac{e^{-xt-xu} t^2 u^2 }{(1-e^{-t})(1-e^{-u})}tu < \int_0^{\infty} \int_0^{\infty} \frac{e^{-xt-xu}  t^2 u^2}{(1-e^{-t})(1-e^{-u})}\frac{t^2+u^2}{2},
\end{equation*}
which is true for all $x>0$. 
\end{proof}

\begin{proof}[Proof of Lemma \ref{lemma:steepdescentz}]
Note that the published version of this proof in \cite{barraquand2017random} contained a mistake. The present version remedies this issue, closely following  the proof of \cite[Lemma 2.6]{barraquand2020large}, which is a statement very similar as Lemma \ref{lemma:steepdescentz} (it corresponds to the $\alpha, \beta \to 0$ limit).

By symmetry, it is enough to treat only the case $y>0$. Hence we show that if $y>0$, then $\Imag[h'(\theta+iy)]>0.$
Using (\ref{eq:hprimetheta}), $\Imag[h'(\theta+iy)]>0$ is equivalent to 
\begin{multline}
\Big( \Psi_1(\theta) - \Psi_1(\alpha+\beta+\theta)\Big) \Imag\Big[\Psi(\alpha+ \theta+iy)-\Psi(\alpha+\beta+\theta+iy)\Big]\\
+ \Big( \Psi_1(\alpha +\theta) - \Psi_1(\alpha+\beta+\theta)\Big) \Imag\Big[\Psi(\alpha+\beta + \theta+iy)-\Psi(\theta+iy)\Big]>0. 
\label{eq:inequality1steepdesent}
\end{multline}
Using the series representations (\ref{eq:seriesrepresentationdigamma}), 
Equation (\ref{eq:inequality1steepdesent}) is equivalent to 
\begin{multline}
\Big( \Psi_1(\theta) - \Psi_1(\alpha+\beta+\theta)\Big)\Imag \sum_{m=0}^{\infty} \frac{-\beta}{(m+\theta+\alpha+ iy)(m+\theta+\alpha+\beta+iy)} \\
+  \Big( \Psi_1(\alpha +\theta) - \Psi_1(\alpha+\beta+\theta)\Big)\Imag \sum_{m=0}^{\infty} \frac{\alpha+\beta}{(m+\theta+ iy)(m+\theta+\alpha+\beta+iy)}>0,
\label{eq:inequalityimag}
\end{multline}
We have that
$$\Imag\left[ \frac{-\beta}{(m+\theta+\alpha+ iy)(m+\theta+\alpha+\beta+iy)}\right] =\frac{1}{(m+\theta+\alpha)^2+y^2} - \frac{1}{(m+\theta+\alpha+\beta)^2+y^2}$$
and 
$$  \Imag\left[\frac{-(\alpha+\beta)}{(m+\theta+ iy)(m+\theta+\alpha+\beta+iy)}\right] = \frac{1}{(m+\theta)^2+y^2} - \frac{1}{(m+\theta+\alpha+\beta)^2+y^2}.$$
It yields that (\ref{eq:inequalityimag}) can be rewritten as
\begin{multline}
\Big( \Psi_1(\theta) - \Psi_1(\alpha+\beta+\theta)\Big)\Big(\Phi(\theta+\alpha)- \Phi(\theta+\alpha+\beta)\Big) >\\
\Big( \Psi_1(\theta)+\alpha - \Psi_1(\alpha+\beta+\theta)\Big)\Big(\Phi(\theta)- \Phi(\theta+\alpha+\beta)\Big),
\label{eq:inequalityPsiPhi1}
\end{multline}
where 
$$\Phi(x) = \sum_{n\geqslant 0} \frac{1}{(n+x)^2+y^2} .$$
The inequality (\ref{eq:inequalityPsiPhi1}) is equivalent to 
\begin{equation}
 \frac{ \Psi_1(\theta) - \Psi_1(\theta+\alpha)}{\Phi(\theta) - \Phi(\theta+\alpha)}> \frac{ \Psi_1(\theta+\alpha) - \Psi_1(\theta+\alpha+\beta)}{\Phi(\theta+\alpha) - \Phi(\theta+\alpha+\beta)}
 \label{eq:inequalityratiosPsiPhi}
\end{equation}
Using Cauchy's mean value theorem, there exist $\theta_1\in(\theta, \theta+\alpha)$ and $\theta_2\in(\theta+\alpha, \theta+\alpha+\beta)$ such that (\ref{eq:inequalityratiosPsiPhi}) is equivalent to 
$$ \frac{\Psi_2(\theta_1)}{\Phi'(\theta_1)} >\frac{\Psi_2(\theta_2)}{\Phi'(\theta_2)}.$$
We need to show that the function $\Psi_2(\theta)/\Phi'(\theta)$ is decreasing. Taking the derivative, this amount to showing that for all $\theta>0$,
$$ \Psi_3(\theta)\Phi'(\theta) - \Psi_2(\theta)\Phi''(\theta) <0.$$
Using, the series representations for the digamma and $\Phi$ functions, this is equivalent to showing that
\begin{equation}
\sum_{n,m=0}^{\infty} T_{n,m} >0 \text{ where } T_{n,m} = a(n+\theta)B(m+\theta) - A(n+\theta) b(m+\theta),
\label{eq:toproveseries}
\end{equation}
with
$$a(x) = \frac{1}{x^4}, \;\; b(x) =\frac{1-\frac{y^2}{3x^2}}{x^4\left(1+\frac{y^2}{x^2}\right)^3}, \;\;
A(x) = \frac{1}{x^3}, \;\;B(x) = \frac{1}{x^3\left(1+\frac{y^2}{x^2}\right)^2}.$$
Note that $$ T_{n,m}> \tilde T_{n,m}:= a(n+\theta)B(m+\theta) - A(n+\theta) \tilde b(m+\theta),$$
where $\tilde b(x) = \frac{1}{x^4\left(1+\frac{y^2}{x^2}\right)^3}\geqslant b(x)$.

In order to prove \eqref{eq:toproveseries}, we will show that $\tilde T_{n,m}+\tilde T_{m,n}\geqslant 0$, and for that  we will show that
\begin{enumerate}
	\item[(1)] For all $0\leqslant n\leqslant m$, $\tilde T_{n,m}\geqslant 0$.
	\item[(2)] For all $0\leqslant n\leqslant m$, either $\tilde T_{n,m}/\tilde T_{m,n}\geqslant 0$ or $\vert \tilde T_{n,m}/\tilde T_{m,n}\vert\geqslant 1$.
\end{enumerate}
To prove (1), observe that using the shorthand notation $n_{\theta}= n+\theta, m_{\theta} =m+\theta$, we have
$$\tilde T_{n,m} = \frac{m_{\theta} \left(m_{\theta}(m_{\theta}-n_{\theta})+y^2\right)}{n_{\theta}^4 \left(m_{\theta}^2+y^2\right)^3}$$
which is clearly non-negative for all $0\leqslant n\leqslant m$.
Now we turn to proving (2). We have
\begin{equation*}
\frac{\tilde T_{n,m}}{\tilde T_{m,n}} = \frac{n_{\theta}}{m_\theta}\frac{\left(1+\frac{y^2}{n_{\theta}^2}\right)^3}{\left(1+\frac{y^2}{m_{\theta}^2}\right)^3}\frac{m_{\theta}(m_{\theta}-n_{\theta})+y^2}{n_{\theta}(n_{\theta}-m_{\theta})+ y^2}
\end{equation*}
For  $ n\leqslant m$, the numerator is always positive. Regarding the denominator, there are two cases to consider. Either it is nonnegative, which implies $T_{n,m}/T_{m,n}>  0$, or the denominator is negative. In the latter case, we have that $y^2<n_{\theta}(m_{\theta}-n_{\theta})$ and
\begin{equation*}
\left\vert\frac{\tilde T_{n,m}}{\tilde T_{m,n}} \right\vert  =  \frac{\left(1+\frac{y^2}{n_{\theta}^2}\right)^3}{\left(1+\frac{y^2}{m_{\theta}^2}\right)^3}   \frac{n_{\theta}}{m_\theta} \frac{m_{\theta}(m_{\theta}-n_{\theta})+y^2}{n_{\theta}(m_{\theta}-n_{\theta})- y^2}.
\end{equation*}
For  $ n\leqslant m$, we clearly have  that $\frac{\left(1+\frac{y^2}{n_{\theta}^2}\right)^3}{\left(1+\frac{y^2}{m_{\theta}^2}\right)^3}\geqslant 1$ and
$$  \frac{n_{\theta}}{m_\theta} \frac{m_{\theta}(m_{\theta}-n_{\theta})+y^2}{n_{\theta}(m_{\theta}-n_{\theta})- y^2} \geqslant   \frac{n_{\theta}}{m_\theta} \frac{m_{\theta}(m_{\theta}-n_{\theta})}{n_{\theta}(m_{\theta}-n_{\theta})} = 1,$$
so that $\left\vert\tilde T_{n,m}/\tilde T_{m,n} \right\vert\geqslant 1$. Therefore we have proved (2) and this concludes the proof of Lemma \ref{lemma:steepdescentz}.

\end{proof}

\begin{proof}[Proof of Lemma \ref{lemma:steepdescentw}]
We have that 
$$ \frac{\mathrm{d}}{\mathrm{d}\phi} \Real\Big[h(\theta e^{i\phi})\Big] = \Real\Big[i\theta e^{i\phi} h'(\theta e^{i\phi})\Big] . $$
Using formula (\ref{eq:expressionshprimeuniform}), we have 
$$ h'(\theta e^{i\phi}) = \frac{\theta (1-e^{i\phi})^2}{e^{i\phi}(\theta e^{i\phi}+1)\big((\theta+1)^2+\theta^2) \big)}.$$
We have to show that for any $\phi\in (0, \pi)$, $\Real\big[i\theta e^{i\phi} h'(\theta e^{i\phi})\big]>0$. 
We can forget the factor $\theta/\left((\theta+1)^2+\theta^2) \right)$ which is positive. 
Thus, we have to show that 
$$ \Imag\left[\frac{(1-e^{i\phi})^2}{(\theta e^{i\phi} +1)}\right]<0.$$
One can see that the inequality is equivalent to 
$$ 2 \sin(\phi)(\cos(\phi)-1) <0,$$
which is always true for $\phi\in (0,\pi)$.
\end{proof}

\subsection{Relation to extreme value theory}

Let us now state a corollary of Theorem \ref{thm:RWasymptotics}. Let $(X^{(1)}_t)_{t\in \Z_{\geqslant 0}}, \dots , (X^{(N)}_t)_{t\in \Z_{\geqslant 0}}$ be $N$ independent random walks drawn in the same Beta environment (Definition \ref{def:BetaRWRE}). We denote by $\mathcal{P}$ and $\mathcal{E}$ 
   the measure and expectation associated with the probability space which is the product of the environment probability space and the $N$ random walks probability space (for $f$ a function of the environment and the $N$ random walk paths, we have $\mathcal{E}\left[ f\right] = \EE \left[ \mathsf{E}^{\otimes N}[f] \right]$ and $\mathcal{P}(A) = \mathcal{E}[\mathds{1}_{A}]$). 
\begin{corollary} Assume $\alpha=\beta=1$.
 We set $N=\lfloor e^{ct}\rfloor$ for some $c\in\left(\frac 2 5,1\right)$, and $x_0 = I^{-1}(c) = \sqrt{1-(1-c)^2}$.  Then we have 
 \begin{equation}
\lim_{t\to\infty} \mathcal{P}\left(  \frac{ \max_{i=1, \dots , \lfloor e^{ct}\rfloor}\left\lbrace X^{(i)}_t \right\rbrace - t x_0}{t^{1/3} d}  \leqslant y  \right) = F_{\rm GUE}(y),
\end{equation}
where 
$$ d= \frac{2^{1/3}c^{2/3}(1-c)^{2/3}}{\sqrt{1-(1-c)^2}}.$$
\label{cor:extreme}
\end{corollary} 
\begin{remark}
The condition $c>2/5$ is equivalent to $x_0>4/5$. It is also equivalent to the condition that $\theta<1/2$ in Theorem \ref{thm:RWasymptotics}. Hence, it is most probably purely technical. 
\label{rem:restrictiononc}
\end{remark}
\begin{remark}
We expect that Corollary \ref{cor:extreme} holds more generally for arbitrary parameters $\alpha, \beta>0$. One would have the following result: 

Let $N=\lfloor e^{ct}\rfloor$ such that there exists  $x_0 >\frac{\alpha-\beta}{\alpha+\beta}$  and $\theta_0>0$ with  $x(\theta_0)=x_0$ and $I(x(\theta_0)) = c$. Then
 \begin{equation}
\lim_{t\to\infty} \mathcal{P}\left(  \frac{ \max_{i=1, \dots , \lfloor e^{ct}\rfloor}\left\lbrace X^{(i)}_t \right\rbrace - t x_0}{t^{1/3}\sigma(x_0)/I'(x_0)}  \leqslant y  \right) = F_{\rm GUE}(y),
\end{equation}
where $I'(x)= \frac{\mathrm{d}}{\mathrm{d} x}I(x)$. 
\end{remark}
\begin{remark}
The range of the parameter $c$ in Corollary \ref{cor:extreme} is a priori 
$c\in (0,1)$. The reason why the upper bound is precisely $1$ is because we are in the $\alpha=\beta =1$ case. In general, the upper bound is $I(1)$, which is always finite. It is natural that $c$ is bounded. Indeed, we know that for all $i$, $X_t^{(i)}\leqslant t$ (because the random walk performs $\pm 1$ steps), and for $c$ very large there exists some $i$ such that $X_t^{(i)}=t$ with high probability. Hence, one expects that for  $c$ large enough, the maximum  $\max_{i=1, \dots , \lfloor e^{ct}\rfloor}\left\lbrace X^{(i)}_t \right\rbrace$ is exactly $t$ with a probability going to $1$ as $t$ goes to infinity, and there cannot be random fluctuations in that case. 

If one considers $N$ simple symmetric random walks (corresponding to the annealed model), the threshold is $\log(2)$ (i.e. for $c>\log(2)$, $(1-(1/2)^t)^N\to 0$ and for $c<\log(2)$, $(1-(1/2)^t)^N\to 1$). One can calculate the large deviations rate function $I^{a}$ for the simple random walk\footnote{By Cr\'amer's Theorem, it is the Legendre transform of $z\mapsto \log\left(\frac{e^{-z}+e^{z}}{2}\right) $. One finds 
$$ I^a(x)= \begin{cases} \frac 1 2 \big( (1+x)\log(1+x)+ (1-x)\log(1-x)\big)\text{ for }x\in [-1,1],\\
+\infty \text{ else.} \end{cases} $$} and check that $I^a(1)=\log(2)$. 
\end{remark}

\begin{proof}[Proof of Corollary \ref{cor:extreme}]
This proof relies on Theorem \ref{thm:RWasymptotics} which deals only with $\alpha=\beta=1$. However, this type of deduction would also hold in the general parameter case, and we write the proof using general form expressions. 
From Theorem \ref{thm:RWasymptotics}, we have that writing 
\begin{equation}
\log(\mathsf{P}(X_t>xt)) = -I(x) t + t^{1/3} \sigma(x) \chi_t, 
\label{eq:logproba}
\end{equation}
then $\chi_t$ weakly converges to the Tracy-Widom GUE distribution, provided $x$ can be written $x=x(\theta)$ with $0<\theta<1/2$. 
For any realization of the environment, we have on the one hand 
\begin{equation*}
\mathsf{P}\left( \max_{i=1, \dots, \lfloor e^{ct}\rfloor}\left\lbrace X^{(i)}_t\right\rbrace \leqslant x t  \right) = \Big(1- \mathsf{P}(X_t>xt)\Big)^{\lfloor e^{ct}\rfloor } = \exp\Big( \lfloor e^{c t}\rfloor \log\big(1- \mathsf{P}(X_t>xt)\big)\Big). 
\end{equation*}
On the other hand, 
setting $x = x_0 + \frac{t^{-2/3}\sigma(x_0) y}{I'(x_0)}$, we have that 
\begin{equation}
\mathcal{P}\left( \max_{i=1, \dots, \lfloor e^{ct}\rfloor}\left\lbrace X^{(i)}_t\right\rbrace \leqslant x t \right)=\mathcal{P}\left(  \frac{ \max_{i=1, \dots , \lfloor e^{ct}\rfloor}\left\lbrace X^{(i)}_t \right\rbrace - t x_0}{t^{1/3}\sigma(x_0)/I'(x_0)}  \leqslant y  \right).
\end{equation}
By Taylor expansion, we have as $t$ goes to infinity
$$ I(x) = I(x_0) + t^{-2/3}\sigma(x_0) y + \mathcal{O}(t^{-4/3}),$$
and 
$$  \sigma(x)= \sigma(x_0) + t^{-2/3}\frac{\sigma'(x_0) \sigma(x_0) y}{I'(x_0)} + \mathcal{O}(t^{-4/3}).$$
Hence, the R.H.S. of (\ref{eq:logproba}) is approximated by 
\begin{equation}
-I(x)t+ t^{1/3}\sigma(x)\chi_t  = -I(x_0) t +  t^{1/3}\sigma(x_0)(\chi_t -y) + \mathcal{O}(t^{-1/3}) + \mathcal{O}(t^{-1/3}\chi_t).
\label{eq:logprobaapprox}
\end{equation}

Choosing $x_0$ such that $I(x_0)= c$, we have 
\begin{align*}
\mathcal{P}\left( \max_{i=1, \dots, \lfloor e^{ct}\rfloor}\left\lbrace X^{(i)}_t\right\rbrace \leqslant x t \right)
 &= \EE \exp\Big(\lfloor e^{ct }\rfloor  \log\big(1- \mathsf{P}(X_t>xt)\big) \Big) \\
&=  \EE \exp\Big(-\lfloor e^{ct }\rfloor P(t, xt) + \mathcal{O}\left(e^{ct} P(t, xt)^2\right) \Big) \\
&= \EE \exp\Big(e^{t^{1/3}\sigma(x_0)(\chi_t -y) + \mathcal{O}(t^{-1/3}(1+\chi_t))} +\mathcal{O}\left( P(t, xt)\right) +\mathcal{O}\left(e^{ct} P(t, xt)^2\right) \Big) 
\end{align*}
The second equality relies on Taylor expansion of the logarithm around $1$. The third equality is the consequence  (\ref{eq:logproba}) and (\ref{eq:logprobaapprox}). Since $\chi_t$ converges in distribution, $t^{-1/3}(1+\chi_t))$ converges in probability to zero by Slutsky's theorem. Hence, the term $\mathcal{O}(t^{-1/3}(1+\chi_t))$ inside the exponential converges in probability to zero. Recalling that $I(x_0)=c$, we have 
$$ P(t, xt)^2 = \exp\left( 2\log\big(P(t, xt)\big) \right)  =\exp\left( 2\Big( -c t +\mathcal{O}( t^{1/3} \chi_t) \Big)  \right)=\exp\Big( -2c t +2t^{2/3}\mathcal{O}( t^{-1/3} \chi_t) \Big),$$
and since $\mathcal{O}(t^{-1/3}(1+\chi_t))$ converges to zero in probability, we have that $P(t, xt)^2$ is smaller that $e^{-\frac{3}{2}ct}$ with probability going to $1$ as $t$ goes to infinity, so that the term  $\mathcal{O}\left(e^{ct} P(t, xt)^2\right)$ can be neglected. One can bound similarly $\mathcal{O}(P(t, xt))$ by $e^{-\frac{1}{2}ct}$ with probability going to $1$. 
 Thus, 
\begin{align*}
 \lim_{t\to\infty} \mathcal{P}\left(  \frac{ \max_{i=1, \dots , \lfloor e^{ct}\rfloor}\left\lbrace X^{(i)}_t \right\rbrace - t x_0}{t^{1/3}\sigma(x_0)/I'(x_0)}  \leqslant y  \right) &= \lim_{t\to\infty} \mathcal{P}\left( \max_{i=1, \dots, \lfloor e^{ct}\rfloor}\left\lbrace X^{(i)}_t\right\rbrace \leqslant x t \right)  \\
 &=\lim_{t\to\infty}\PP(\chi_t\leqslant y) \\
 &= F_ {\rm GUE}(y). 
\end{align*}
In the case $\alpha=\beta=1$, we have seen that $I(x) = 1-\sqrt{1-x^2}$ so that  $x_0 = \sqrt{1-(1-c)^2}$. Moreover, using (\ref{eq:simpleexpressionsigma}), 
\begin{equation}
\sigma(x_0) = \left(\frac{2c^2}{1-c}\right)^{1/3}
\label{eq:sigma0}
\end{equation}
and
\begin{equation}
I'(x_0) = \frac{x_0}{\sqrt{1-x_0^2}} = \frac{\sqrt{1-(1-c)^2}}{1-c}.
\label{eq:Iprime}
\end{equation}
Combining \eqref{eq:sigma0} and \eqref{eq:Iprime} yields 
$$\sigma(x_0)/I'(x_0) = d = \frac{2^{1/3}c^{2/3}(1-c)^{2/3}}{\sqrt{1-(1-c)^2}},$$
as in the statement of Corollary \ref{cor:extreme}. Finally, we have that $I(x(1/2)) = 2/5$ and $x((1/2)) = 4/5$, so that the hypothesis of Corollary \ref{cor:extreme} match with that of Theorem \ref{thm:RWasymptotics}.  
\end{proof}

In order to put Corollary \ref{cor:extreme} in the perspective of extreme value statistics, recall that if  $(G_i)_i$ for $i=1, \dots, \lfloor e^{ct}\rfloor$ is a sequence of independent Gaussian centred random variables of variance $1$, then we have (\cite[Section 2.3.2]{galambos1987asymptotic}) the weak convergence 
$$ \sqrt{2ct}\max_{i=1, \dots , e^{ct}} \left\lbrace G_i\right\rbrace  - 2 ct + \frac{1}{2}\log(t) +  \frac{1}{2}\log(4\pi c)\Longrightarrow \mathcal{G}, $$
where $\mathcal{G}$ is a Gumbel random variable with cumulative distribution function $\exp(-e^{-x})$.

 For the Beta-RWRE with  general $\alpha, \beta >0$ parameters, the variables $X^{(i)}_t$ have mean $\frac{\alpha -\beta}{\alpha+\beta}t$ with variance $ \mathcal{O}(t)$ (see Proposition \ref{prop:covariancestructure} (1) and (2)). Let us note $$R_t^{(i)}:=\frac{X_t^{(i)} - \frac{\alpha-\beta}{\alpha+\beta}t}{\sqrt{t}}.$$ 
We know that $R_t^{(i)}$ converges weakly to the Gaussian distribution by the central limit theorem.  Moreover, conditionally on the environment, $R_t^{(i)}$ converges weakly to the Gaussian distribution (It is proved in \cite{rassoul2005almost},
 see the discussion in Section \ref{subsec:localizationintro}). 
 However, if we let the environment vary, the variables $R^{(i)}_t$ are not independent since the random walks all share the same environment.  

The next proposition characterizes the covariance structure of the family $(X^{(i)}_t)_{i\geqslant 1}$. We state the result for any parameters $\alpha, \beta>0$.
\begin{proposition}
\begin{enumerate}
\item For all  $i\geqslant 1$, we have $\mathcal{E}\left[X^{(i)}_t \right]= t\frac{\alpha-\beta}{\alpha+\beta}$.
\item For all $i\geqslant 1$, we have $\mathcal{E}\left[\left(X^{(i)}_t\right)^2 \right]= \left(\frac{\alpha-\beta}{\alpha+\beta} \right)^2t^2  + \frac{4\alpha\beta}{(\alpha+\beta)^2} t $.
\item For all $i\neq j \geqslant 1$, we have  
\begin{equation}
\mathcal{E}\left[X^{(i)}_t X^{(j)}_t \right]=  \left(t\frac{\alpha-\beta}{\alpha+\beta}\right)^2 +  \frac{4\alpha \beta \sum_{s=0}^{t-1}\mathcal{P}(X^{(i)}_s= X^{(j)}_s)}{(\alpha+\beta)^2 (\alpha+\beta+1)}.
\label{eq:relationcovarianceoverlap}
\end{equation}
\item For two random variables $X$ and $Y$ measurable with respect to $\mathcal{P}$, we denote their correlation coefficient as
$$\rho(X,Y) = \frac{\mathcal{E}[XY]}{\sqrt{\mathcal{E}[X^2]\mathcal{E}[Y^2]}}.$$ 
For all $i\neq j \geqslant 1$, the correlation coefficient $\rho(X^{(i)}_t, X^{(j)}_t )$ equals $1/(\alpha+\beta+1)$ times the $\mathcal{E}$-expected proportion of overlap between the walks $X^{(i)}_t$ and  $X^{(j)}_t$, up to time $t$.  
\end{enumerate}
\label{prop:covariancestructure}
\end{proposition}
\begin{proof}
The points (1) and (2) are trivial since $X_t$ is actually a simple random walk if we do not condition on the environment. In any case, let us explain each case explicitly.
\begin{enumerate}
\item Let us write $\Delta_t = X_{t+1}-X_{t}$. Then $X_t=\sum_{i=0}^{t-1} \Delta_i$. $\Delta_i$ is a random variable that takes the value $1$ with probability $\EE[B]$ and the value $-1$ with probability $\EE[1-B]$ for some $Beta(\alpha, \beta)$ random variable $B$. We find that $\mathcal{E}\left[\Delta_t \right] = \frac{\alpha-\beta}{\alpha+\beta}$, and 
$$ \mathcal{E}\left[ X_t\right] = \sum_{i=1}^t \mathcal{E}\left[\Delta_t \right] = t\frac{\alpha-\beta}{\alpha+\beta}.$$
\item We have 
$$ \mathcal{E}\left[ (X_t)^2\right]  =\mathcal{E}\left[ \sum_{i=1}^t \Delta_i\sum_{j=1}^t \Delta_j\right] .$$
For $i\neq j$, $\mathcal{E}\left[ \Delta_i \Delta_j\right] = \mathcal{E}\left[ \Delta_i \right]\mathcal{E}\left[ \Delta_j\right]$, and since $\Delta_i$ equals plus or minus one,  $\mathcal{E}\left[ (\Delta_i)^2 \right]=1$. Hence, 
$$  \mathcal{E}\left[ (X_t)^2\right] = t(t-1)\left(\frac{\alpha-\beta}{\alpha+\beta}\right)^2 + t = \left(t\frac{\alpha-\beta}{\alpha+\beta}\right)^2 +t\frac{4\alpha\beta}{(\alpha+\beta)^2}.$$
\item Let us write $\Delta_t^{(i)} = X^{(i)}_{t+1} -X^{(i)}_t$ and $\Delta_t^{(j)} = X^{(j)}_{t+1} -X^{(j)}_t$. We have 
$$ \mathcal{E}\left[X^{(i)}_t X^{(j)}_t\right] = \mathcal{E}\left[\sum_{n=0}^{t-1}\Delta_n^{(i)}\sum_{m=0}^{t-1}\Delta_m^{(j)} \right].$$
For $n\neq m$, since the increments and the environments corresponding to different times are independent,
$$\mathcal{E}\left[\Delta_n^{(i)}\Delta_m^{(j)} \right] = \mathcal{E}\left[\Delta_n^{(i)}\right]\mathcal{E}\left[\Delta_m^{(j)} \right] = \left(\frac{\alpha-\beta}{\alpha+\beta}\right)^2.$$
However, $\mathcal{E}\left[\Delta_n^{(i)}\Delta_n^{(j)} \right]$ depends on whether $ X^{(i)}_n=X^{(j)}_n$ or not. More precisely, 
$$ \mathcal{E}\left[\Delta_n^{(i)}\Delta_n^{(j)}\Big\vert X^{(i)}_n \neq X^{(j)}_n \right]  = \mathcal{E}\left[\Delta_n^{(i)}\right]\mathcal{E}\left[\Delta_m^{(j)} \right] = \left(\frac{\alpha-\beta}{\alpha+\beta}\right)^2,$$
and 
$$ \mathcal{E}\left[\Delta_n^{(i)}\Delta_n^{(j)}\Big\vert X^{(i)}_n = X^{(j)}_n \right]  = \EE\left[ \mathsf{E}\left[\Delta_n^{(i)} \right]\mathsf{E}\left[\Delta_n^{(j)} \right]\Big\vert X^{(i)}_n = X^{(j)}_n\right] = \EE\left[ (2B-1)^2\right],$$
for some $Beta(\alpha, \beta)$ random variable $B$. This yields 
$$ \mathcal{E}\left[\Delta_n^{(i)}\Delta_n^{(j)} \right] = \mathcal{P}(X^{(i)}_n\neq X^{(j)}_n)\left(\frac{\alpha-\beta}{\alpha+\beta}\right)^2 + \mathcal{P}(X^{(i)}_n= X^{(j)}_n)\EE\left[ (2B-1)^2\right]. $$
Using $ \EE[B^2] = \frac{\alpha(\alpha+1)}{(\alpha+\beta)(\alpha+\beta+1)},$
we find that 
$$ \mathcal{E}\left[X^{(i)}_t X^{(j)}_t\right] = t^2\left( \frac{\alpha-\beta}{\alpha+\beta}\right)^2 + \left(\sum_{s=0}^{t-1}\mathcal{P}\left( X^{(i)}_s= X^{(j)}_s\right) \right) \frac{4\alpha\beta}{(\alpha+\beta)^2(\alpha+\beta+1)}.$$
\item The $\mathcal{E}$-expected proportion of overlap between the walks $X^{(i)}_t$ and  $X^{(j)}_t$ up to time $t$ is  $$\frac{1}{t}\mathcal{E}\left[\sum_{s=0}^{t-1}\mathds{1}_{X^{(i)}_s=X^{(j)}_s}\right]= \frac{1}{t}\sum_{s=0}^{t-1}\mathcal{P}(X^{(i)}_s= X^{(j)}_s).$$
 Hence, the point (4) is a direct consequence of (1), (2) and (3).
\end{enumerate}
\end{proof}

One can precisely describe the behaviour of $\sum_{s=0}^{t-1}\mathcal{P}(X^{(i)}_s= X^{(j)}_s)$. For simplicity, we restrict the study to the case where the random walks have no drift, that is $\alpha=\beta$.
\begin{proposition}
Consider $(X_t^{(1)})_{t\in \Z_{\geqslant 0}}$ and $(X_t^{(2)})_{t\in \Z_{\geqslant 0}}$ two Beta-RWRE drawn independently in the same environment with parameters $\alpha=\beta$. Then 
$$ \sqrt{t} \ \cdot\ \mathcal{P}\left(X_t^{(1)} =  X_t^{(2)}\right)  \xrightarrow[t\to\infty]{} \frac{2\alpha+1}{2\alpha} \frac{1}{\sqrt{\pi}}, $$
and consequently 
$$ \sqrt{t}\ \cdot\  \mathcal{E}\left[\frac{X^{(i)}_t}{\sqrt{t}} \frac{X^{(j)}_t }{\sqrt{t}}\right] \xrightarrow[t\to\infty]{} \frac{1}{\alpha\sqrt{\pi}}.$$
\label{prop:calculprobaoverlap}
\end{proposition}
\begin{proof}
First, notice that $\left(X^{(1)}_t - X^{(2)}_t\right)_{t\geqslant 0}$ is a random walk. Let $Y_t:= X^{(1)}_t - X^{(2)}_t$. The transitions probabilities depend on whether $Y_t=0$. If $Y_t=0$, then 
$$Y_{t+1}-Y_t = \begin{cases} +2 &\mbox{with probability }\EE\big[B(1-B)\big]\\
0 &\mbox{with probability }\EE\big[B^2 + (1-B)^2\big]\\  
-2 &\mbox{with probability }\EE\big[B(1-B)\big]\end{cases}$$
where $B$ is a $Beta(\alpha, \alpha)$ random variable. 
If $Y_t\neq 0$, then
$$Y_{t+1}-Y_t = \begin{cases} +2 &\mbox{with probability }1/4\\
0 &\mbox{with probability }1/2\\  
-2 &\mbox{with probability }1/4\end{cases}$$
In the following, we denote $r=\EE\big[B(1-B)\big] = \frac{\alpha}{4\alpha+2}$. We also denote $P_t:= \mathcal{P}\left(Y_t=0\right)$ which is the quantity that we want to approximate. 

We introduce an auxiliary random walk starting from $0$ and  having transitions 
$$\begin{cases} +2 &\mbox{with probability }1/4,\\
0 &\mbox{with probability }1/2,\\  
-2 &\mbox{with probability }1/4.\end{cases}$$
We denote by $Q_{t}$ the probability for the auxiliary random walk 
to arrive at zero at time $t$ and stay in the non-negative region between times $0$ and $t$. 

By conditioning on the first return in zero of the random walk $(Y_t)_t$, we claim that for $t\geqslant 2$,
\begin{equation}
 P_t = (1-2r)P_{t-1} + 2\sum_{i=2}^{t} r \frac{1}{4} Q_{i-2}P_{t-i}.
 \label{eq:recurrencePt}
\end{equation}
Let us explain more precisely equation (\ref{eq:recurrencePt}) (see Figure \ref{fig:recurrenceexplained}):
\begin{itemize}
\item The term $(1-2r)P_{t-1}$ corresponds to the case when the first return at zero occur at time $1$.
\item The factor $2$ in front of the sum in (\ref{eq:recurrencePt}) accounts for the fact that the walk can stay either in the positive, or in the negative region before the first return in zero, with equal probability.
\item  The factor $r$ is the probability that $Y_1=2$ (which is also the probability that $Y_1=-2$).
\item The factor $1/4$ is the probability of the last step before the first return at zero.  
\end{itemize}

\begin{figure}
\begin{center}
\begin{tikzpicture}[scale=1]
\draw[->, thick] (0,0) -- (10.5, 0);
\draw[->, thick] (0,-3) -- (0, 4.5);
\draw[gray, fill=gray!10] (0.4,0.9) -- (4.6,0.9) -- (4.6, 3.1) --(0.4,3.1) --(0.4,0.9);
\foreach \x in {-4,-2, ..., 8}
	{\draw[thick] (0.1, \x/2) -- (-0.1, \x/2) node[anchor=east]{$\x $};}
\draw (0,0)   --(0.5,1)   -- (1,2)  -- (1.5, 2)  -- (2,3)  -- (2.5, 3)  -- (2.5, 2)  -- (3, 1)  -- (3.5, 1)  -- (4, 2)  -- (4.5, 1)  -- (5, 0)  -- (5.5, -1)  -- (6, -1)  -- (6.5, -2)  -- (7, -1)  -- (7.5, 0)  -- (8, 0)  -- (8.5, 1)  -- (9,0)  -- (9.5, -1) -- (10, 0) ;
\foreach \Point in {(0,0) ,(0.5,1) ,(1,2)  , (1.5, 2) , (2,3) ,(2.5, 3) , (2.5, 2) , (3, 1) , (3.5, 1)  , (4, 2)  , (4.5, 1) , (5, 0) , (5.5, -1)  , (6, -1) ,(6.5, -2) , (7, -1) ,(7.5, 0) , (8, 0) ,(8.5, 1) , (9,0) , (9.5, -1) , (10, 0)}{\fill \Point circle(0.05);}

\draw[ultra thick] (0,0) -- (0.5, 1);
\draw[ultra thick] (4.5,1)  -- (5, 0);
\draw[thick, ->] (-2, -1) node[anchor=north]{\footnotesize{probability $r$}} to[bend left] (0.2,0.5);
\draw[thick, ->] (7, 2) node[anchor=south]{\footnotesize{probability $1/4$}} to[bend left] (4.9,0.5);
\draw[thick, ->] (3, -2) node[anchor=north east]{\footnotesize{first return at $0$}} to[bend right] (5,-0.1);
\draw[thick, ->] (8, 4) node[anchor=south west]{\footnotesize{probability $ Q_8$}} to (4.6,3.1);

\clip (0, -2.8) rectangle (10.3,4.3);
\draw[gray, dotted, step=0.5] (0, -5.5) grid (10.5, 5.5);
\end{tikzpicture}
\end{center}
\caption{A possible trajectory of the random walk $Y_t$ is decomposed to explain the recurrence (\ref{eq:recurrencePt}). The trajectory in the gray box has the same probability as that of the auxiliary random walk. }
\label{fig:recurrenceexplained}
\end{figure}
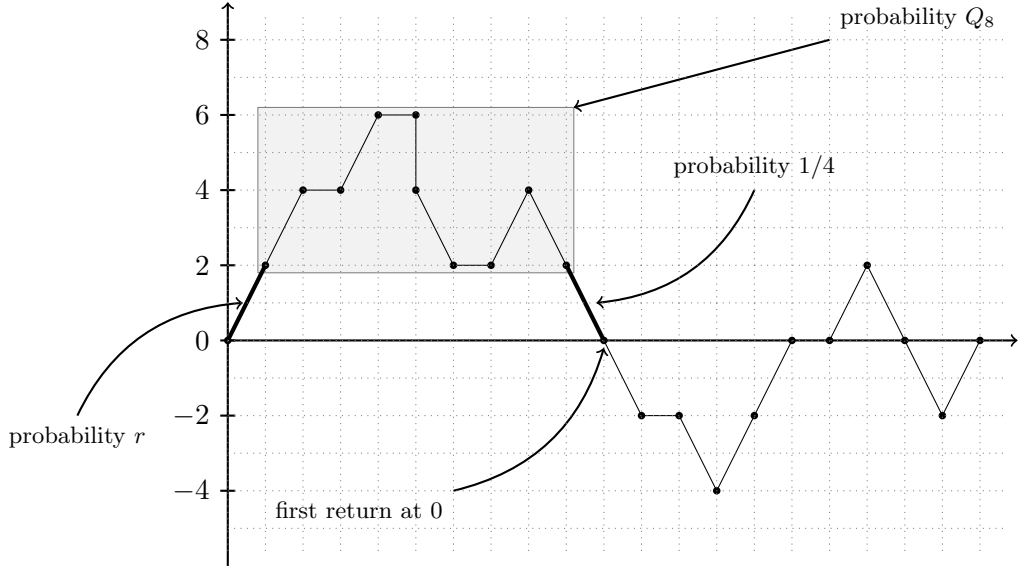

By conditioning on the first return at zero of the auxiliary random walk, one can see that $Q_t$ verifies the recurrence 
$$ Q_t = \frac{1}{2}Q_{t-1} +\sum_{i=2}^{t}\frac{1}{16}Q_{i-2}Q_{t-i} \text{ for }t\geqslant 2.$$
This implies that if $Q(z)=\sum_{n\geqslant 0} Q_n z^n $ is the generating function of the sequence $(Q_n)_n$, then 
$$ Q(z)- 1- 1/2 z  = 1/2 z(Q(z)-1) + 1/16 z^2 Q(z)^2.$$
This yields
$$ Q(z) = \frac{8-4z - 8\sqrt{1-z}}{z^2}.$$

Now, let us denote $G(z)=  \sum_{n\geqslant 0} P_n z^n$ the generating function of the sequence $(P_n)_n$. The recurrence (\ref{eq:recurrencePt}) implies that 
$$ G(z)-1-(1-2r)z = (1-2r) z (G(z)-1) + 2 r (1/4) G(z)Q(z).$$
This yields 
$$ G(z) =\frac{1}{1+z(4r-1) + 4 r (\sqrt{1-z}-1)}.$$
The function $G(z)$ is analytic in the unit open disk, and can be developed in series around $0$ with radius of convergence $1$. The nature of its singularities on the unit circle gives the leading order for the asymptotic behaviour of its series coefficients. As $z\to 1$ (for $z\in \C\setminus D$ where $D$ is the cone $D=\lbrace  z : \vert\arg(z-1) \vert<\epsilon \rbrace$, for some $\epsilon>0$ arbitrarily small, and taking the branch cut of $\sqrt{1-z} $ along $\R_{\geqslant 1}$),
$$ G(z) \sim \frac{1}{4r\sqrt{1-z}}, $$
where $\sim$ means that the ratio of the two sides tends to $1$ as $z\to 1$ and $z$ belongs to the domain described above. We deduce (from e.g. \cite[Corollary VI.1]{flajolet2009analytic}) that
$$ P_t \sim \frac{1}{4r} \frac{1}{\sqrt{\pi t}}.$$ 

This clearly implies that
 $$\frac{\sum_{s=0}^{t-1}P_s}{\sqrt{t}} \xrightarrow[t\to \infty]{} \frac{1}{2r\sqrt{\pi}}.$$ 
  Since $r= \frac{\alpha}{4\alpha+2}$ and using (\ref{eq:relationcovarianceoverlap}), we get 
$$ \sqrt{t}\mathcal{E}\left[\frac{X^{(i)}_t}{\sqrt{t}} \frac{X^{(j)}_t }{\sqrt{t}}\right] \xrightarrow[t\to \infty]{} \frac{1}{\alpha\sqrt{\pi}}.$$
\end{proof}

\subsubsection*{Comparison to correlated  Gaussian variables} Consider for simplicity only the case $\alpha=\beta$. We denote as before  $R_t^{(i)}=X_t^{(i)}/\sqrt{t}$. As already mentioned in Section \ref{subsec:localizationintro},  $R_t^{(i)}$ converges weakly as $t$ goes to infinity to the Gaussian distribution  $\mathcal{N}(0,1)$ (whether we condition on the environment or not). 
It is tempting to ask if the same limit theorem for the maximum holds when one replaces the  $R_t^{(i)}$ by the corresponding limiting collection of Gaussian random variables (it would correspond to taking first the limit when $t$ goes to infinity and then study the maximum as $N$ goes to infinity). The theory of extreme value statistics provides a negative answer. 

Let $\Sigma_N(\lambda)$  be the matrix of size $N$
$$ \Sigma_N(\lambda):= \left( \begin{matrix}
1 & \frac{\lambda}{\sqrt{\log(N)}} &\dots & \frac{\lambda}{\sqrt{\log(N)}}\\
\frac{\lambda}{\sqrt{\log(N)}}&1&  &\vdots\\
\vdots &  & \ddots & \frac{\lambda}{\sqrt{\log(N)}}\\
\frac{\lambda}{\sqrt{\log(N)}}& \dots & \frac{\lambda}{\sqrt{\log(N)}} &1
\end{matrix}\right), $$
where $\lambda>0$ is a parameter. If we set $N=\lfloor e^{ct}\rfloor$, and look at the maximum of the sequence $\lbrace R_t^{(i)}\rbrace_{1\leqslant i\leqslant N} $ as $t$ goes to infinity, the correlation matrix of the sequence is asymptotically  $\Sigma_N(\lambda)$ with $\lambda= \frac{\sqrt{c/\pi}}{\alpha}$ (cf. Proposition \ref{prop:calculprobaoverlap}).  

Let $G_N:=(G^{(1)}, \dots , G^{(N)})$ be a Gaussian vector with covariance matrix $\Sigma_N(\lambda)$ and let denote the maximum $M_N:= \max_{i=1, \dots, N}\lbrace G^{(i)} \rbrace$.  Theorem 3.8.1 in \cite{galambos1987asymptotic} implies that we have the convergence in distribution 
$$ \frac{M_N- \sqrt{2\log(N)} + \lambda\sqrt{2}}{\left(\lambda^{-1}\sqrt{\log(N)}\right)^{-1/2}} \Longrightarrow \mathcal{N}(0,1).$$
In particular, we have the convergence in probability of $M_N/\sqrt{\log(N)}$ to $\sqrt{2}$. 

Thus, we have seen that the maximum of $\left(R_t^{(i)}\right)_{1\leqslant i\leqslant N}$ and the maximum of $\left(G^{(i)}\right)_{1\leqslant i\leqslant N}$ obey to very different limit theorems: both the scales and the limiting laws are different. 
\begin{remark}
By Corollary \ref{cor:extreme}, we have the convergence in probability 
 $$\frac{\max_{i=1, \dots, N}\lbrace R_{\log(N)/c}^{(i)} \rbrace}{\sqrt{\log(N)}} \xrightarrow[N\to\infty]{\mathcal{P}} \frac{x_0}{\sqrt{c}}, $$
 where $c=I(x_0)$. Since for any $\alpha$ and $\beta=\alpha$, $I''(0)=1$, we notice that when $x_0\to 0$, the approximation at the first order coincide with the Gaussian case. To substantiate this parallel, one must extend to the full parameter range  $\alpha, \beta>0$ and  $0<c<1$ in Corollary \ref{cor:extreme} byond $\alpha=\beta=1$ and $c>2/5$  (see also Remark \ref{rem:restrictiononc}).   
\end{remark}
\begin{remark}
It is clear that the sequence $\left(X_t^{(i)} \right)_{1\leqslant i \leqslant N}$ is exchangeable. There exist general results for maxima of exchangeable sequences. In some cases, one can prove that the maximum, properly renormalized, converges to a mixture of one of the classical extreme laws 
(see in \cite{galambos1987asymptotic} the discussion in Section 3.2 and the results of Section 3.6). However, it seems that our particular setting does not fit into this theory. 
\end{remark}

\section{Asymptotic analysis of the Bernoulli-Exponential directed first passage percolation}
\label{sec:FPPasymptotics}

\subsection{Statement of the result}

We investigate the behaviour of the first passage time $T(n,\kappa n)$ when $n$ goes to infinity, for some slope $\kappa>\frac{a}{b}$. When $\kappa=\frac a b$, the first passage time $T(n, \kappa n)$ should go to zero. The case $\kappa<\frac a b$ is similar with $\kappa>\frac a b$ by symmetry. 

As in Theorem \ref{thm:RWasymptotics}, we parametrize the slope $\kappa$ by a parameter $\theta$ (which turns out to be the position of the critical point in the asymptotic analysis). Let 
\begin{align}
\kappa(\theta) &:= \dfrac{\dfrac{1}{\theta^2} -\dfrac{1}{(a+\theta)^2}}{\dfrac{1}{(a+\theta)^2}-\dfrac{1}{(a+b+\theta)^2}}, 
 \label{eq:defkappa}
\\
\tau(\theta) &:= \frac{1}{a+\theta} - \frac{1}{\theta}  +\kappa(\theta)\left(\frac{1}{a+\theta}-\frac{1}{a+b+\theta}\right) = \frac{a(a+b)}{\theta^2 (2a+b+2\theta)},
\label{eq:deftau}
\end{align}
and 
\begin{equation}
 \rho(\theta) := \left[\frac{1}{\theta^3} -\frac{1}{(a+\theta)^3} + \kappa(\theta)\left(\frac{1}{(a+b+\theta)^3} -\frac{1}{(a+\theta)^3} \right) \right]^{1/3}.
\label{eq:defrho}
\end{equation}
When $\theta$ ranges from $0$ to $+\infty$, $\kappa(\theta)$ ranges from $+\infty$ to $a/b$ and $\tau(\theta)$ ranges from $+\infty$ to $0$. 
\begin{theorem}
We have that for any $\theta>0$ and parameters $a,b>0$, 
$$ \lim_{n \to \infty}\mathbb{P}\left(\frac{T\big(n, \kappa(\theta)n\big) - \tau(\theta)n}{\rho(\theta)n^{1/3}}\geqslant -y \right) = F_{\mathrm{GUE}}(y).$$
\label{thmTWFPP}
\end{theorem}
By Theorem \ref{thmProb}, we have a Fredholm determinant representation for the probability 
$$\PP\Big(T\big(n, \kappa(\theta) n\big)>r\Big).$$ 
We set $ r=\tau(\theta) n -\rho(\theta) n^{1/3} y$.
Thus, we have that 
$$\PP\Big(T\big(n, \kappa(\theta) n\big)>\tau(\theta) n -\rho(\theta) n^{1/3} y\Big) = \det(I-K^{\mathrm{FPP}}_r)_{\mathbb{L}^2(C'_0)},$$
 where 
$$ K^{\mathrm{FPP}}_r(u,u') = \frac{1}{2i\pi} \int_{1/2-i\infty}^{1/2+i\infty} \exp\Big(n(H(u+s)- H(u))- \rho(\theta) n^{1/3} ys\Big)\frac{u+s}{u} \frac{\mathrm{d}s}{s(s+u-u')},$$
and 
$$ H(z):= \tau(\theta) z +\log\left(\frac{z}{a+z} \right)+\kappa(\theta)\log\left( \frac{a+b+z}{a+z}\right).$$
We have 
$$ H'(z) = \tau(\theta) +\frac{1}{z} - \frac{1}{a+z} +\kappa(\theta)\left(\frac{1}{a+b+z}-\frac{1}{a+z}\right).$$
and
$$ H''(z) = \frac{1}{(a+z)^2} -\frac{1}{z^2} +\kappa(\theta) \left(\frac{1}{(a+z)^2} - \frac{1}{(a+b+z)^2}\right).$$
We can see from the expressions for the derivatives of $H$ why it is natural to paramametrize $\kappa, \tau$ and $\rho$ as in (\ref{eq:defkappa}), (\ref{eq:deftau}) and (\ref{eq:defrho}): with this choice, we have that $H'(\theta) = H''(\theta)=0$. 

As in Section \ref{sec:RWREasymptotics}, we assume for the moment that the Fredholm determinant contour is a small circle around $0$. We do the change of variables $z=u+s$ in the definition of the kernel, so that 
\begin{equation}
 K^{\mathrm{FPP}}_r(u,u') = \frac{1}{2i\pi} \int_{1/2-i\infty}^{1/2+i\infty} \exp\big(n(H(z)- H(u))- \rho(\theta) n^{1/3} y(z-u)\big)\frac{z}{u} \frac{\mathrm{d}z}{(z-u)(z-u')}.
 \label{eq:KFPPy}
\end{equation}
\begin{lemma}
For any parameters $a,b>0$ and $\theta>0$, we have $H'''(\theta)>0$.
\label{lem:thirdderivativeFPP} 
\end{lemma}
\begin{proof}
We have 
$$H'''(\theta) = \frac{2}{\theta^3} -\frac{2}{(a+\theta)^3} + \dfrac{\frac{1}{\theta^2} -\frac{1}{(a+\theta)^2}}{\frac{1}{(a+\theta)^2}-\frac{1}{(a+b+\theta)^2}}\left(\frac{2}{(a+b+\theta)^3} -\frac{2}{(a+\theta)^3} \right).$$ 
Hence we have to show that 
\begin{multline}
\left( \frac{2}{\theta^3} -\frac{2}{(a+\theta)^3}\right)\left(\frac{1}{(a+\theta)^2}-\frac{1}{(a+b+\theta)^2} \right) >\\
\left(\frac{2}{(a+\theta)^3}-\frac{2}{(a+b+\theta)^3} \right) \left( \frac{1}{\theta^2} -\frac{1}{(a+\theta)^2}\right).
\end{multline}
By putting each side to the same denominator, we arrive at 
\begin{align*}
&b(a+b+\theta)(2\theta+2a + b)\left((a+\theta)^3-  \theta^3\right) > a\theta(2\theta+ a) \left((a+b+\theta)^3-(a+\theta)^3 \right)\\
\Leftrightarrow\ \ & a b (a + b) (a + \theta)^2 (2 a + b + 3 \theta)>0.
\end{align*}
which clearly holds.
\end{proof}
We notice that given the expression (\ref{eq:defrho}), $H'''(\theta) = 2\big(\rho(\theta)\big)^3$. By Taylor expansion around $\theta$,
\begin{equation}
 H(z) -H(\theta) = \frac{(\rho(\theta)(z-\theta))^3}{3}  + \mathcal{O}((z-\theta)^4). 
 \label{eq:taylorexpansionFPP}
\end{equation}

\subsection{Deformation of contours} 
We need to find steep-descent contours for the variables $z$ and $u$. For the $z$ variable, we choose the contour $\mathcal{D}_{\theta} = \theta+i\R$ as in Section \ref{sec:RWREasymptotics}. For the $u$ variable, we notice that since we are integrating on a finite contour, it will be enough that $\Real[H(z)]> \Real[H(\theta)]$ along the contour (See \cite{tracy2009asymptotics} and \cite{borodin2014stochastic}). 
\begin{lemma}
The contour $\mathcal{D}_{\theta}$ is steep-descent for the function $\Real[H]$ in the sense that $y\mapsto \Real[H(\theta+iy)]$ is decreasing for $y$ positive and increasing for $y$ negative. 
\label{lem:steepdescentHz}
\end{lemma}
\begin{proof}
Since $\frac{\mathrm{d}}{\mathrm{d}y}\Real[H(\theta+iy)] = \Imag[H'(\theta+iy)]$, and using symmetry with respect to the real axis, it is enough to show that for $y>0$, $\Imag[H'(\theta+iy)]>0$. We have
$$\Imag[H'(\theta+iy)] = \frac{y}{(\theta+a)^2+y^2} - \frac{y}{\theta^2+y^2} +\kappa(\theta)\left(\frac{y}{(\theta+a)^2+y^2}-\frac{y}{(\theta+a+b)^2+y^2}\right).$$
Given the expression (\ref{eq:defkappa}) for $\kappa(\theta)$, we have to show that
\begin{multline}
 \left( \frac{1}{\theta^2+y^2} - \frac{1}{(\theta+a)^2+y^2}\right)\left(\frac{1}{(a+\theta)^2}-\frac{1}{(a+b+\theta)^2} \right) <\\
  \left(\frac{1}{(\theta+a)^2+y^2}-\frac{1}{(\theta+a+b)^2+y^2}\right)\left( \frac{1}{\theta^2} -\frac{1}{(a+\theta)^2}\right) .
  \label{eq:inequalitysteepdescentz}
  \end{multline}
Factoring both sides in the inequality (\ref{eq:inequalitysteepdescentz}) and cancelling equal factors, one readily sees that it is equivalent to 
$$ \frac{1}{(\theta^2+y^2)(a+b+\theta)^2} <\frac{1}{ \left((\theta+a+b)^2+y^2\right)\theta^2}, $$
which is always satisfied. 
\end{proof}
Instead of finding a steep-descent path for the $\mathbb{L}^2$ contour   as in Section \ref{sec:RWREasymptotics}, we prove that we can find a contour with suitable properties for asymptotics analysis, following the approach of \cite{borodin2014stochastic}. 
\begin{lemma}
There exists a closed continuous path $\gamma$ in the complex plane, such that 
\begin{itemize}
\item The path $\gamma$ encloses $0$ but not $-a-b$,
\item The path $\gamma$ crosses the point $\theta$ and departs $\theta$ with angles $\phi$ and $-\phi$, for some $\phi\in(\pi/2, 5\pi/6)$,
\item Let $B(\theta, \epsilon)$ the ball of radius $\epsilon$ centred at $\theta$. For any $\epsilon>0$, there exists $\eta>0$ such that  for all $z\in \gamma\setminus B(\theta, \epsilon)$,  $\Real[H(z)]-\Real[H[\theta]]>\eta$.
\end{itemize}
\label{lem:steepdescentHu}
\end{lemma}
\begin{proof}
Since $H$ is analytic away from its singularities, $\Real[H]$ is a harmonic function. It turns out that the shape of level lines $\Real[H(z)] = \Real[H(\theta)]$ are constrained by the nature and the positions of the singularities of $H$, and provided $H$ is not too complicated (does not have too many singularities), one can describe these level lines. 

We know that level lines can cross only at singularities or critical points. In our case, three branches of the level line $\Real[H(z)] = \Real[H(\theta)]$ cross at $\theta$ making angles $\pi/6, \pi/2$ and $5\pi/6$. This can be seen from the Taylor expansion (\ref{eq:taylorexpansionFPP}).

The function $H$ has only three singularities of logarithmic type at $0$, $-a$ and $-a+b$. When $z$ goes to infinity, $\Real[H(z)] = \Real[H(\theta)]$ implies $\Re[\tau(\theta) z ]\approx \Real[H(\theta)]$. Hence, there are two branches that goes to infinity in the direction $\pm\infty i+\Real[H(\theta)]/\tau(\theta)$. 
Additionally, one knows by the maximum principle that any closed path formed by portions of level lines must enclose a singularity. Finally, one knows the sign of $\Real[H(z)]$ around each singularity:
\begin{itemize}
\item  $\Real[H(z)]<0$ for $z$ near $0$, 
\item $\Real[H(z)]<0$ for $z$ near $-a-b$, 
\item $\Real[H(z)]>0$ for $z$ near $-a$.
\end{itemize}
This is enough to conclude that the level lines of $\Real[H(z)] = \Real[H(\theta)]$ are necessarily as shown in Figure \ref{fig:contourplot} (modulo a continuous deformation of the lines that does not cross any singularity).
\begin{figure}
\includegraphics[scale=0.4]{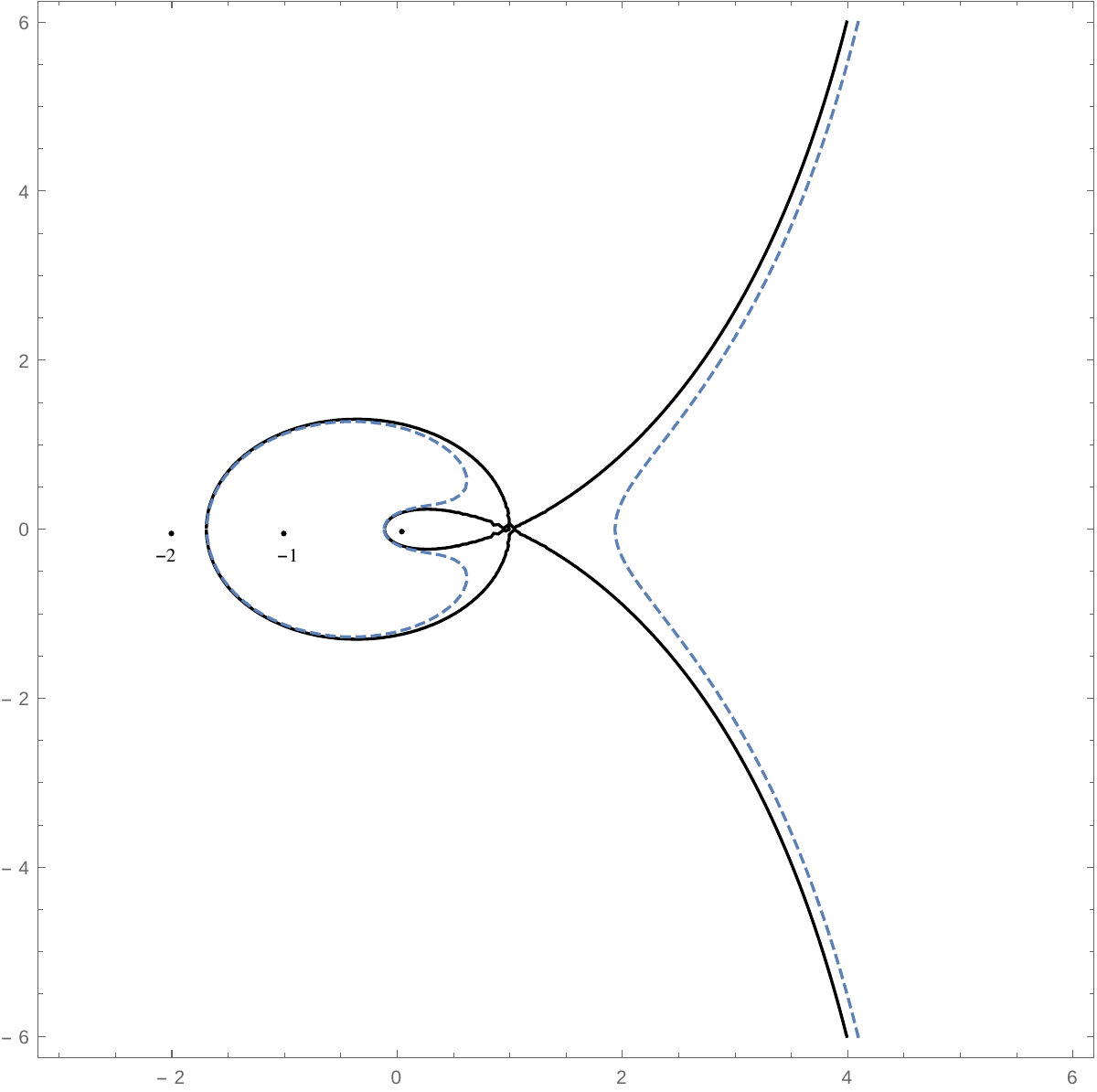}
 \caption{The solid lines are contour lines $\Real[H(z)] = \Real[H(\theta)]$ in the case $\theta=a=b=1$. Dashed lines are contour lines $\Real[H(z)] = \Real[H(\theta)]+2\eta$ with $\eta=0.05$. }
 \label{fig:contourplot}
\end{figure}
It follows that one can find a path $\gamma$ having the required properties. It would depart $\theta$ with angles $\pm \phi$ with $\phi\in (\pi/2, 5\pi/6)$, and stay between the level lines that depart $\theta$ with angles $ \pm\pi/2$ and the level lines that departs $\theta$ with angles $\pm 5\pi/6$ (For instance, one could follow the level lines of $\Real[H(z)] = \Real[H(\theta)]+ 2\eta$ outside of a neighbourhood of $\theta$). 
\end{proof}
We have the analogue of Proposition \ref{prop:localization}. 
\begin{proposition}
\label{prop:FPPlocalization}
Let $B(\theta, \epsilon)$ be the ball of radius $\epsilon$ centred at $\theta$. We denote by $\gamma^{\epsilon}$ (resp. $\mathcal{D}_{\theta}^{\epsilon}$) the part of the contour $\gamma$ (resp. $\mathcal{D}_{\theta}$) inside the ball $B(\theta, \epsilon)$. Then, for any $\epsilon>0$, 
$$ \lim_{t\to\infty} \det(I-K^{\mathrm{FPP}}_r)_{\mathbb{L}^2(\mathcal{C}_{\theta})} = \lim_{t\to\infty} \det(I-K^{\mathrm{FPP}}_{y, \epsilon})_{\mathbb{L}^2(\gamma^{\epsilon})}$$
where $K^{\mathrm{FPP}}_{y, \epsilon}$ is defined by the integral kernel
\begin{equation}
K^{\mathrm{FPP}}_{y, \epsilon}(u,u') = \frac{1}{2i\pi} \int_{\mathcal{D}_{\theta}^{\epsilon}} \frac{\pi}{\sin(\pi (z-u))} \exp\left( t(H(z)-H(u)) -t^{1/3}\rho(\theta)y (z-u)\right)\frac{\mathrm{d}z}{z - u'}.
\label{eq:kerneltruncatedFPP}
\end{equation} 
\end{proposition}
\begin{proof}
The proof is similar to the proof of Proposition \ref{prop:localization}. The two main differences are 
\begin{enumerate}
\item The integral defining $K^{\mathrm{FPP}}_y$ in (\ref{eq:KFPPy}) is an improper integral, which forbids to use dominated convergence. 
\item The $\mathbb{L}^2$ contour (i.e. the contour $\gamma$) is not steep-descent. 
\end{enumerate}
The point (2) is not an issue since in the proof of Proposition \ref{prop:localization}, we actually only used the fact that for any $\epsilon>0$ there exists a constants $C'>0$ such that $\Real[h(z)]-\Real[h(\theta)]>C'$ for $z\in\mathcal{C}_{\theta}\setminus \mathcal{C}_{\theta}^{\epsilon}$.  This property is still satisfied by the contour $\gamma$. 

The point (1) is resolved by bounding the integral over $\mathcal{D}_{\theta}\setminus \mathcal{D}_{\theta}^{\epsilon}$ with the same kind of estimates as in the proof of Theorem \ref{thmProb}. More precisely, one writes
\begin{multline}
\bigg\vert \frac{1}{2i\pi} \int_{\theta+i\epsilon}^{\theta+i\infty} \exp\big(n(H(z)- H(u))- \rho(\theta) n^{1/3} y(z-u)\big)\frac{z}{u} \frac{\mathrm{d}z}{(z-u)(z-u')} \bigg\vert <\\  \exp\Big(-C n - n^{1/3}\rho(\theta)y (\theta-u) \big) \bigg\vert   \frac{1}{2i\pi} \int_{\theta+i\epsilon}^{\theta+i\infty} \exp\Big(-i \rho(\theta) n^{1/3} y \Imag[z]\Big)\frac{z}{u} \frac{\mathrm{d}z}{(z-u)(z-u')} \bigg\vert.
\label{eq:estimatedecaykernelFPP}
\end{multline}
The integral in the R.H.S of (\ref{eq:estimatedecaykernelFPP}) is an oscillatory integral that can be bounded uniformly in $n$ (actually it goes to zero by Riemann-Lebesgue's lemma) so that it goes to zero when multiplied by $\exp\big(-C n - n^{1/3}\rho(\theta)y (\theta-u) \big)$. 
\end{proof}
The rest of the proof is similar to Section \ref{sec:RWREasymptotics}. One makes the change of variables 
$$ z=\theta+\tilde{z}n^{-1/3},\  u=\theta+\tilde{u}n^{-1/3}, u'=\theta+\tilde{u}'n^{-1/3}.$$
It is again convenient to deform slightly the contours for $u$ and $u'$ so that the contour for $\tilde{u}$ and $\tilde{u}'$ is $\mathcal{C}^{\epsilon n^{1/3}}$ as in Section \ref{sec:RWREasymptotics} ($\mathcal{C}^L$ is defined in (\ref{eq:defnewcontour})).  
\begin{proposition}
\label{prop:limitkernelFPP}
We have that 
$$ \lim_{t\to\infty} \det(I-K^{\mathrm{FPP}}_{y,\epsilon})_{\mathbb{L}^2(\gamma^{\epsilon})}= \det(I+K_y)_{\mathbb{L}^2(\mathcal{C})},$$
where the contour $\mathcal{C}$ is defined in (\ref{eq:defcontourC}) and $K_y$ is defined by its integral kernel 
$$ K_y(w,w') = \frac{1}{2i\pi} \int_{\infty e^{-i\pi/3}}^{\infty e^{i\pi/3}} \frac{\mathrm{d}z}{(z-w')(w-z)} \frac{e^{z^3/3-yz} }{e^{w^3/3-yw}}$$
and the contour for $z$ does not intersect $\mathcal{C}$. 
\label{prop:limitfredholmFPP}
\end{proposition}
\begin{proof}
The proof is similar to the proof of Proposition \ref{prop:limitfredholm}. It is actually slightly simpler, as  the contour for the variables $u,u'$ (i.e. the variables of the kernel $K^{\mathrm{FPP}}_{y, \epsilon}(u,u')$) departs $\theta$ with angle $\pm \phi$, where $\phi\in (\pi/2, 5\pi/6)$ does not depend on $\epsilon$. Hence, using a Taylor expansion of the function $H$ to the third order as in \eqref{eq:thirdorderTaylor} suffices to show that
\begin{equation*}
\mathrm{Re}\left[-t (H(u)-H(\theta) \right] <-c  \vert \tilde u \vert^3,
\end{equation*}
for some constant $c>0$, which allows to bound the kernel appropriately.
\end{proof}

\subsection{Limit shape of the percolation cluster for fixed $t$.}
\label{subsec:nonrigorous}
\label{subsec:limitshapefinitetime}
\begin{figure}
\includegraphics[scale=0.8]{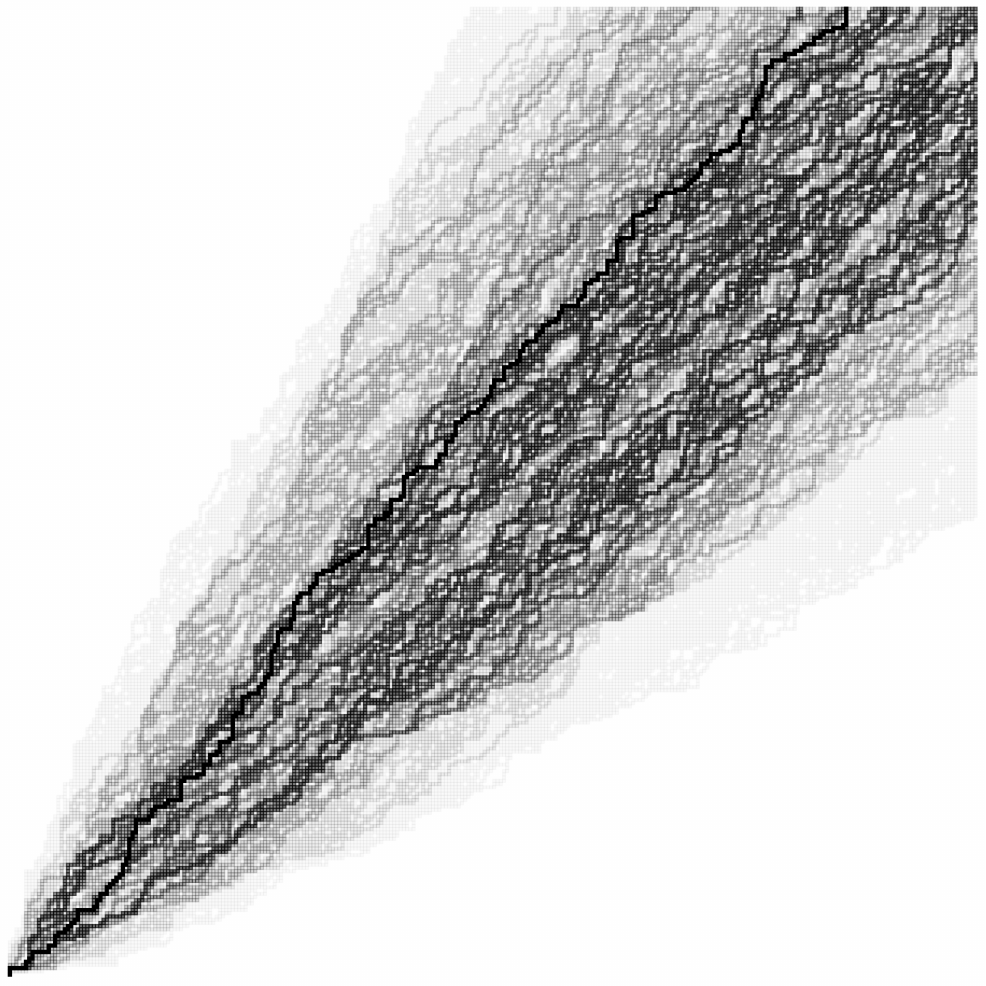}
\caption{Percolation set in the Bernoulli-FPP model at different times for parameters $a=b=1$. The different shades of gray corresponds to times $0$, $0.1$, $0.2$, $0.3$, $0.4$, $0.6$, $1$ and  $4$. Although it seems on the picture that the convex envelope of the percolation cluster at time $t=4$ is asymptotically a cone, this is an effect due to the relatively small size of the grid ($300\times 300$), and it is not true asymptotically: $n=300$ is not enough to discriminate between $c n$ and $c'n^{2/3}$ (see Section \ref{subsec:limitshapefinitetime}).}
\label{fig:simuFPP}
\end{figure}

As $\theta$ goes to infinity, $\kappa(\theta), \tau(\theta)$ and $\rho(\theta)$ are approximated by 
\begin{align*}
\kappa(\theta) = \frac{a}{b} +  \frac{3a(a+b)}{2b} \left( \frac{1}{\theta}\right) + \mathcal{O}\left( \frac{1}{\theta}\right)^2, \\
\tau(\theta) = \frac{1}{2} a(a+b) \left( \frac{1}{\theta}\right)^3+  \mathcal{O}\left( \frac{1}{\theta}\right)^4, \\
\rho(\theta) = \left(\frac{3}{2} a(a+b)\right)^{1/3}\left( \frac{1}{\theta}\right)^{5/3}.
\end{align*}
On the other hand, we have from Theorem \ref{thmTWFPP} the convergence in distribution
$$\frac{ T(n, \kappa(\theta)n) -\tau(\theta) n  }{\rho(\theta)  n^{1/3}}   \Longrightarrow \mathcal{L}_{GUE},$$
where $\mathcal{L}_{GUE}$ is the GUE Tracy-Widom distribution. 

Scaling $\theta$ by $n^{1/3}$ suggests a limit theorem for the shape of the convex envelope of the percolation cluster after a fixed time. Of course, there is a non-rigorous interchange of limits here, and one should use the Fredholm determinant representation in order to make this rigorous (we do not include this here). 

Let us set $\theta=n^{1/3}$. Then  
$$\kappa(\theta)n = \frac{a}{b}n+\frac{3a(a+b)}{2b}n^{2/3} +  \mathcal{O}(n^{1/3})$$
and
$$ \tau(\theta)n = \frac{1}{2} a(a+b) + \mathcal{O}(n^{-1/3}).$$
This suggests that the border of the percolation cluster  at time $\frac{1}{2} a(a+b)$ is asymptotically at a distance  $\frac{3a(a+b)}{2b}n^{2/3}$ from the point $\frac{a}{b}n$ (See Figure \ref{fig:simuFPP}). 
The fact that $\rho(\theta)n^{1/3} = \mathcal{O}(n^{-2/9})$ suggests an anomalous scaling for the fluctuations of the border of the  percolation cluster. We leave this for future consideration.

\bibliographystyle{amsalpha}
\bibliography{betapolymer.bib}

\end{document}